\documentclass[a4paper]{article}

\usepackage{latexsym,amsthm,amscd}
\usepackage[all]{xy}
\usepackage{rotating}
\usepackage[utf8]{inputenc}
\usepackage[T1]{fontenc}
\usepackage[english]{babel}
\usepackage{amsmath}
\usepackage{graphicx}
\usepackage{amssymb}
\usepackage{amsfonts}
\usepackage[svgnames]{xcolor}
\usepackage{geometry}
\usepackage{fancyhdr}
\usepackage{lastpage} 
\usepackage{array}
\usepackage{stmaryrd}
\usepackage{hyperref}
\usepackage{tikz}
\usepackage{bm}
\usepackage[french]{minitoc}
\usepackage{silence}

\usetikzlibrary{patterns}
\usepgflibrary{decorations.pathmorphing}
\usetikzlibrary{decorations.pathmorphing}
\usetikzlibrary{shadows}
\usetikzlibrary{positioning}
\usetikzlibrary{arrows,automata}
\usetikzlibrary{arrows.meta}

\newtheorem*{rep@theorem}{\rep@title}
\newcommand{\newreptheorem}[2]{%
\newenvironment{rep#1}[1]{%
 \def\rep@title{#2 \ref{##1}}%
 \begin{rep@theorem}}%
 {\end{rep@theorem}}}

\newtheorem{theorem}{Theorem}
\newtheorem{proposition}[theorem]{Proposition}
\newtheorem{definition}[theorem]{Definition}

\newtheorem{corollary}[theorem]{Corollary}
\newtheorem{lemma}[theorem]{Lemma}

\newtheorem{remark}{Remark}

\newreptheorem{theorem}{Theorem}
\newreptheorem{lemma}{Lemma}

\title{A characterization of 
entropy dimensions of minimal $\Z^3$-SFT}

\author{Silvère Gangloff, Mathieu Sablik}

\newcommand{\todo}[1]{\textcolor{red}{To do:} \textcolor{blue}{#1}}

\def\N{\mathbb N}
\def\R{\mathbb R}
\def\Z{\mathbb Z}
\def\Q{\mathbb Q}
\def\C{\mathbb C}
\def\K{\mathbb K}
\def\D{\mathbb D}
\def\ds{\displaystyle}

\newcommand{\Lang}{\mathcal{L}}

\newcommand{\id}{\operatorname{Id} }
\newcommand{\tr}{\operatorname{Tr} }

\newcommand{\A}{\mathcal{A}}
\newcommand{\s}{\sigma}
\newcommand{\MT}{\mathcal{M}}

\renewcommand{\vec}[1]{\textbf{#1}}
\newcommand{\U}{\mathbb{U}}

\newcommand{\define}[1]{\textbf{#1}}
\newcommand{\supp}[1]{\textrm{supp}\left(#1\right)}

\newcommand{\blacksq}{
\begin{tikzpicture}
\fill[Black] (0,0) rectangle (0.3,0.3) ;
\draw (0,0)--(0.3,0) ;
\draw (0,0)--(0,0.3) ;
\draw (0,0.3)--(0.3,0.3) ; 
\draw (0.3,0)--(0.3,0.3) ;
\end{tikzpicture}}

\newcommand{\greensq}{
\begin{tikzpicture}
\fill[YellowGreen] (0,0) rectangle (0.3,0.3) ;
\draw (0,0)--(0.3,0) ;
\draw (0,0)--(0,0.3) ;
\draw (0,0.3)--(0.3,0.3) ; 
\draw (0.3,0)--(0.3,0.3) ;
\end{tikzpicture}}

\newcommand{\redsq}{
\begin{tikzpicture}
\fill[Red] (0,0) rectangle (0.3,0.3) ;
\draw (0,0)--(0.3,0) ;
\draw (0,0)--(0,0.3) ;
\draw (0,0.3)--(0.3,0.3) ; 
\draw (0.3,0)--(0.3,0.3) ;
\end{tikzpicture}}

\newcommand{\purplesq}{
\begin{tikzpicture}
\fill[Purple] (0,0) rectangle (0.3,0.3) ;
\draw (0,0)--(0.3,0) ;
\draw (0,0)--(0,0.3) ;
\draw (0,0.3)--(0.3,0.3) ; 
\draw (0.3,0)--(0.3,0.3) ;
\end{tikzpicture}}

\newcommand{\orangesq}{
\begin{tikzpicture}
\fill[Orange] (0,0) rectangle (0.3,0.3) ;
\draw (0,0)--(0.3,0) ;
\draw (0,0)--(0,0.3) ;
\draw (0,0.3)--(0.3,0.3) ; 
\draw (0.3,0)--(0.3,0.3) ;
\end{tikzpicture}}

\newcommand{\blanksq}{
\begin{tikzpicture}
\draw (0,0)--(0.3,0) ;
\draw (0,0)--(0,0.3) ;
\draw (0,0.3)--(0.3,0.3) ; 
\draw (0.3,0)--(0.3,0.3) ;
\end{tikzpicture}}

\newcommand{\robionegauche}[2]{
\draw (#1,#2) rectangle (#1+2,#2+2);
\draw [-latex] (#1+2,#2+1) -- (#1+0,#2+1);
\draw [-latex] (#1+1,#2+0) -- (#1+1,#2+1); 
\draw [-latex] (#1+1,#2+2) -- (#1+1,#2+1);}
\newcommand{\robionedroite}[2]{
\draw (#1,#2) rectangle (#1+2,#2+2);
\draw [-latex] (#1+0,#2+1) -- (#1+2,#2+1);
\draw [-latex] (#1+1,#2+0) -- (#1+1,#2+1); 
\draw [-latex] (#1+1,#2+2) -- (#1+1,#2+1);}
\newcommand{\robionebas}[2]{
\draw (#1,#2) rectangle (#1+2,#2+2);
\draw [-latex] (#1+1,#2+0) -- (#1+1,#2+2);
\draw [-latex] (#1+0,#2+1) -- (#1+1,#2+1); 
\draw [-latex] (#1+2,#2+1) -- (#1+1,#2+1);}
\newcommand{\robionehaut}[2]{
\draw (#1,#2) rectangle (#1+2,#2+2);
\draw [-latex] (#1+1,#2+2) -- (#1+1,#2+0);
\draw [-latex] (#1+0,#2+1) -- (#1+1,#2+1); 
\draw [-latex] (#1+2,#2+1) -- (#1+1,#2+1);}

\newcommand{\robitwobas}[2]{
\draw (#1,#2) rectangle (#1+2,#2+2);
\draw [-latex] (#1+1,#2+2) -- (#1+1,#2+0) ;
\draw [-latex] (#1+0,#2+1) -- (#1+1,#2+1) ; 
\draw [-latex] (#1+0,#2+0.5) -- (#1+1,#2+0.5) ; 
\draw [-latex] (#1+2,#2+1) -- (#1+1,#2+1) ;
\draw [-latex] (#1+2,#2+0.5) -- (#1+1,#2+0.5) ;}
\newcommand{\robitwohaut}[2]{
\draw (#1,#2) rectangle (#1+2,#2+2);
\draw [-latex] (#1+1,#2+0) -- (#1+1,#2+2) ;
\draw [-latex] (#1+0,#2+1) -- (#1+1,#2+1) ; 
\draw [-latex] (#1+0,#2+1.5) -- (#1+1,#2+1.5) ; 
\draw [-latex] (#1+2,#2+1) -- (#1+1,#2+1) ;
\draw [-latex] (#1+2,#2+1.5) -- (#1+1,#2+1.5) ;}
\newcommand{\robitwodroite}[2]{
\draw (#1,#2) rectangle (#1+2,#2+2);
\draw [-latex] (#1+0,#2+1) -- (#1+2,#2+1) ;
\draw [-latex] (#1+1,#2+0) -- (#1+1,#2+1) ; 
\draw [-latex] (#1+1.5,#2+0) -- (#1+1.5,#2+1) ; 
\draw [-latex] (#1+1,#2+2) -- (#1+1,#2+1) ;
\draw [-latex] (#1+1.5,#2+2) -- (#1+1.5,#2+1) ;}
\newcommand{\robitwogauche}[2]{
\draw (#1,#2) rectangle (#1+2,#2+2);
\draw [-latex] (#1+2,#2+1) -- (#1+0,#2+1) ;
\draw [-latex] (#1+1,#2+0) -- (#1+1,#2+1) ; 
\draw [-latex] (#1+0.5,#2+0) -- (#1+0.5,#2+1) ; 
\draw [-latex] (#1+1,#2+2) -- (#1+1,#2+1) ;
\draw [-latex] (#1+0.5,#2+2) -- (#1+0.5,#2+1) ;}

\newcommand{\robithreebas}[2]{
\draw (#1,#2) rectangle (#1+2,#2+2) ;
\draw [-latex] (#1+1,#2+2) -- (#1+1,#2+0) ;
\draw [-latex] (#1+0.5,#2+2) -- (#1+0.5,#2+0) ; 
\draw [-latex] (#1+0,#2+1) -- (#1+0.5,#2+1) ; 
\draw [-latex] (#1+2,#2+1) -- (#1+1,#2+1) ;}
\newcommand{\robithreehaut}[2]{
\draw (#1,#2) rectangle (#1+2,#2+2) ;
\draw [-latex] (#1+1,#2+0) -- (#1+1,#2+2) ;
\draw [-latex] (#1+0.5,#2+0) -- (#1+0.5,#2+2) ; 
\draw [-latex] (#1+0,#2+1) -- (#1+0.5,#2+1) ; 
\draw [-latex] (#1+2,#2+1) -- (#1+1,#2+1) ;}
\newcommand{\robithreegauche}[2]{
\draw (#1,#2) rectangle (#1+2,#2+2) ;
\draw [-latex] (#1+2,#2+1) -- (#1+0,#2+1) ;
\draw [-latex] (#1+2,#2+0.5) -- (#1+0,#2+0.5) ; 
\draw [-latex] (#1+1,#2+0) -- (#1+1,#2+0.5) ; 
\draw [-latex] (#1+1,#2+2) -- (#1+1,#2+1) ;}
\newcommand{\robithreedroite}[2]{
\draw (#1,#2) rectangle (#1+2,#2+2) ;
\draw [-latex] (#1+0,#2+1) -- (#1+2,#2+1) ;
\draw [-latex] (#1+0,#2+0.5) -- (#1+2,#2+0.5) ; 
\draw [-latex] (#1+1,#2+0) -- (#1+1,#2+0.5) ; 
\draw [-latex] (#1+1,#2+2) -- (#1+1,#2+1) ;}

\newcommand{\robifourbas}[2]{
\draw (#1,#2) rectangle (#1+2,#2+2) ;
\draw [-latex] (#1+1,#2+2) -- (#1+1,#2+0) ;
\draw [-latex] (#1+1.5,#2+2) -- (#1+1.5,#2+0) ; 
\draw [-latex] (#1+0,#2+1) -- (#1+1,#2+1) ; 
\draw [-latex] (#1+2,#2+1) -- (#1+1.5,#2+1) ;}
\newcommand{\robifourgauche}[2]{
\draw (#1,#2) rectangle (#1+2,#2+2) ;
\draw [-latex] (#1+2,#2+1) -- (#1+0,#2+1) ;
\draw [-latex] (#1+2,#2+1.5) -- (#1+0,#2+1.5) ; 
\draw [-latex] (#1+1,#2+0) -- (#1+1,#2+1) ; 
\draw [-latex] (#1+1,#2+2) -- (#1+1,#2+1.5) ;}
\newcommand{\robifourhaut}[2]{
\draw (#1,#2) rectangle (#1+2,#2+2) ;
\draw [-latex] (#1+1,#2+0) -- (#1+1,#2+2) ;
\draw [-latex] (#1+1.5,#2+0) -- (#1+1.5,#2+2) ; 
\draw [-latex] (#1+0,#2+1) -- (#1+1,#2+1) ; 
\draw [-latex] (#1+2,#2+1) -- (#1+1.5,#2+1) ;}
\newcommand{\robifourdroite}[2]{
\draw (#1,#2) rectangle (#1+2,#2+2) ;
\draw [-latex] (#1+0,#2+1) -- (#1+2,#2+1) ;
\draw [-latex] (#1+0,#2+1.5) -- (#1+2,#2+1.5) ; 
\draw [-latex] (#1+1,#2+0) -- (#1+1,#2+1) ; 
\draw [-latex] (#1+1,#2+2) -- (#1+1,#2+1.5) ;}

\newcommand{\robifivebas}[2]{
\draw (#1,#2) rectangle (#1+2,#2+2) ;
\draw [-latex] (#1+1,#2+2) -- (#1+1,#2+0) ; 
\draw [-latex] (#1+0,#2+1) -- (#1+1,#2+1) ; 
\draw [-latex] (#1+2,#2+1) -- (#1+1,#2+1) ;
\draw [-latex] (#1+0,#2+0.5) -- (#1+1,#2+0.5) ; 
\draw [-latex] (#1+2,#2+0.5) -- (#1+1,#2+0.5) ;}
\newcommand{\robifivegauche}[2]{
\draw (#1,#2) rectangle (#1+2,#2+2) ;
\draw [-latex] (#1+2,#2+1) -- (#1+0,#2+1) ; 
\draw [-latex] (#1+1,#2+0) -- (#1+1,#2+1) ; 
\draw [-latex] (#1+1,#2+2) -- (#1+1,#2+1) ;
\draw [-latex] (#1+0.5,#2+0) -- (#1+0.5,#2+1) ; 
\draw [-latex] (#1+0.5,#2+2) -- (#1+0.5,#2+1) ;}
\newcommand{\robifivehaut}[2]{
\draw (#1,#2) rectangle (#1+2,#2+2) ;
\draw [-latex] (#1+1,#2+0) -- (#1+1,#2+2) ; 
\draw [-latex] (#1+0,#2+1) -- (#1+1,#2+1) ; 
\draw [-latex] (#1+2,#2+1) -- (#1+1,#2+1) ;
\draw [-latex] (#1+0,#2+1.5) -- (#1+1,#2+1.5) ; 
\draw [-latex] (#1+2,#2+1.5) -- (#1+1,#2+1.5) ;}
\newcommand{\robifivedroite}[2]{
\draw (#1,#2) rectangle (#1+2,#2+2) ;
\draw [-latex] (#1+0,#2+1) -- (#1+2,#2+1) ; 
\draw [-latex] (#1+1,#2+0) -- (#1+1,#2+1) ; 
\draw [-latex] (#1+1,#2+2) -- (#1+1,#2+1) ;
\draw [-latex] (#1+1.5,#2+0) -- (#1+1.5,#2+1) ; 
\draw [-latex] (#1+1.5,#2+2) -- (#1+1.5,#2+1) ;}

\newcommand{\robisixbas}[2]{
\draw (#1,#2) rectangle (#1+2,#2+2) ;
\draw [-latex] (#1+1,#2+2) -- (#1+1,#2+0) ;
\draw [-latex] (#1+0.5,#2+2) -- (#1+0.5,#2+0) ; 
\draw [-latex] (#1+0,#2+1) -- (#1+0.5,#2+1) ; 
\draw [-latex] (#1+2,#2+1) -- (#1+1,#2+1) ;
\draw [-latex] (#1+0,#2+0.5) -- (#1+0.5,#2+0.5) ; 
\draw [-latex] (#1+2,#2+0.5) -- (#1+1,#2+0.5) ;}
\newcommand{\robisixgauche}[2]{
\draw (#1,#2) rectangle (#1+2,#2+2) ;
\draw [-latex] (#1+2,#2+1) -- (#1+0,#2+1) ;
\draw [-latex] (#1+2,#2+0.5) -- (#1+0,#2+0.5) ; 
\draw [-latex] (#1+1,#2+0) -- (#1+1,#2+0.5) ; 
\draw [-latex] (#1+1,#2+2) -- (#1+1,#2+1) ;
\draw [-latex] (#1+0.5,#2+0) -- (#1+0.5,#2+0.5) ; 
\draw [-latex] (#1+0.5,#2+2) -- (#1+0.5,#2+1) ;}
\newcommand{\robisixhaut}[2]{
\draw (#1,#2) rectangle (#1+2,#2+2) ;
\draw [-latex] (#1+1,#2+0) -- (#1+1,#2+2) ;
\draw [-latex] (#1+0.5,#2+0) -- (#1+0.5,#2+2) ; 
\draw [-latex] (#1+0,#2+1) -- (#1+0.5,#2+1) ; 
\draw [-latex] (#1+2,#2+1) -- (#1+1,#2+1) ;
\draw [-latex] (#1+0,#2+1.5) -- (#1+0.5,#2+1.5) ; 
\draw [-latex] (#1+2,#2+1.5) -- (#1+1,#2+1.5) ;}
\newcommand{\robisixdroite}[2]{
\draw (#1,#2) rectangle (#1+2,#2+2) ;
\draw [-latex] (#1+0,#2+1) -- (#1+2,#2+1) ;
\draw [-latex] (#1+0,#2+0.5) -- (#1+2,#2+0.5) ; 
\draw [-latex] (#1+1,#2+0) -- (#1+1,#2+0.5) ; 
\draw [-latex] (#1+1,#2+2) -- (#1+1,#2+1) ;
\draw [-latex] (#1+1.5,#2+0) -- (#1+1.5,#2+0.5) ; 
\draw [-latex] (#1+1.5,#2+2) -- (#1+1.5,#2+1) ;}

\newcommand{\robisevenbas}[2]{
\draw (#1,#2) rectangle (#1+2,#2+2) ;
\draw [-latex] (#1+1,#2+2) -- (#1+1,#2+0) ;
\draw [-latex] (#1+1.5,#2+2) -- (#1+1.5,#2+0) ; 
\draw [-latex] (#1+0,#2+1) -- (#1+1,#2+1) ; 
\draw [-latex] (#1+2,#2+1) -- (#1+1.5,#2+1) ;
\draw [-latex] (#1+0,#2+0.5) -- (#1+1,#2+0.5) ; 
\draw [-latex] (#1+2,#2+0.5) -- (#1+1.5,#2+0.5) ;}
\newcommand{\robisevengauche}[2]{
\draw (#1,#2) rectangle (#1+2,#2+2) ;
\draw [-latex] (#1+2,#2+1) -- (#1+0,#2+1) ;
\draw [-latex] (#1+2,#2+1.5) -- (#1+0,#2+1.5) ; 
\draw [-latex] (#1+1,#2+0) -- (#1+1,#2+1) ; 
\draw [-latex] (#1+1,#2+2) -- (#1+1,#2+1.5) ;
\draw [-latex] (#1+0.5,#2+0) -- (#1+0.5,#2+1) ; 
\draw [-latex] (#1+0.5,#2+2) -- (#1+0.5,#2+1.5) ;}
\newcommand{\robisevenhaut}[2]{
\draw (#1,#2) rectangle (#1+2,#2+2) ;
\draw [-latex] (#1+1,#2+0) -- (#1+1,#2+2) ;
\draw [-latex] (#1+1.5,#2+0) -- (#1+1.5,#2+2) ; 
\draw [-latex] (#1+0,#2+1) -- (#1+1,#2+1) ; 
\draw [-latex] (#1+2,#2+1) -- (#1+1.5,#2+1) ;
\draw [-latex] (#1+0,#2+1.5) -- (#1+1,#2+1.5) ; 
\draw [-latex] (#1+2,#2+1.5) -- (#1+1.5,#2+1.5) ;}
\newcommand{\robisevendroite}[2]{
\draw (#1,#2) rectangle (#1+2,#2+2) ;
\draw [-latex] (#1+0,#2+1) -- (#1+2,#2+1) ;
\draw [-latex] (#1+0,#2+1.5) -- (#1+2,#2+1.5) ; 
\draw [-latex] (#1+1,#2+0) -- (#1+1,#2+1) ; 
\draw [-latex] (#1+1,#2+2) -- (#1+1,#2+1.5) ;
\draw [-latex] (#1+1.5,#2+0) -- (#1+1.5,#2+1) ; 
\draw [-latex] (#1+1.5,#2+2) -- (#1+1.5,#2+1.5) ;}

\newcommand{\robibluebasgauche}[2]{
\fill[blue!40] (#1+0.5,#2+0.5) rectangle (#1+1,#2+2) ;
\fill[blue!40] (#1+0.5,#2+0.5) rectangle (#1+2,#2+1) ;
\node[scale=0.75] at (#1+1.5,#2+1.5) {0};
\draw (#1,#2) rectangle (#1+2,#2+2) ;
\draw [-latex] (#1+0.5,#2+0.5) -- (#1+0.5,#2+2) ; 
\draw [-latex] (#1+0.5,#2+0.5) -- (#1+2,#2+0.5) ;
\draw [-latex] (#1+1,#2+1) -- (#1+1,#2+2) ; 
\draw [-latex] (#1+1,#2+1) -- (#1+2,#2+1) ; 
\draw [-latex] (#1+0.5,#2+1) -- (#1+0,#2+1) ; 
\draw [-latex] (#1+1,#2+0.5) -- (#1+1,#2+0) ;}
\newcommand{\robibluebasdroite}[2]{
\fill[blue!40] (#1+1.5,#2+0.5) rectangle (#1+1,#2+2) ;
\fill[blue!40] (#1+1.5,#2+0.5) rectangle (#1+0,#2+1) ;
\draw (#1,#2) rectangle (#1+2,#2+2) ;
\node[scale=0.75] at (#1+0.5,#2+1.5) {0};
\draw [-latex] (#1+1.5,#2+0.5) -- (#1+1.5,#2+2) ; 
\draw [-latex] (#1+1.5,#2+0.5) -- (#1+0,#2+0.5) ;
\draw [-latex] (#1+1,#2+1) -- (#1+1,#2+2) ; 
\draw [-latex] (#1+1,#2+1) -- (#1+0,#2+1) ; 
\draw [-latex] (#1+1.5,#2+1) -- (#1+2,#2+1) ; 
\draw [-latex] (#1+1,#2+0.5) -- (#1+1,#2+0) ;}
\newcommand{\robibluehautgauche}[2]{
\fill[blue!40] (#1+2,#2+1) rectangle (#1+0.5,#2+1.5) ;
\fill[blue!40] (#1+1,#2+0) rectangle (#1+0.5,#2+1.5) ;
\draw (#1,#2) rectangle (#1+2,#2+2) ;
\node[scale=0.75] at (#1+1.5,#2+0.5) {0};
\draw [-latex] (#1+0.5,#2+1.5) -- (#1+0.5,#2+0) ; 
\draw [-latex] (#1+0.5,#2+1.5) -- (#1+2,#2+1.5) ;
\draw [-latex] (#1+1,#2+1) -- (#1+1,#2+0) ; 
\draw [-latex] (#1+1,#2+1) -- (#1+2,#2+1) ; 
\draw [-latex] (#1+0.5,#2+1) -- (#1+0,#2+1) ; 
\draw [-latex] (#1+1,#2+1.5) -- (#1+1,#2+2) ;}
\newcommand{\robibluehautdroite}[2]{
\fill[blue!40] (#1+0,#2+1) rectangle (#1+1.5,#2+1.5) ;
\fill[blue!40] (#1+1,#2+0) rectangle (#1+1.5,#2+1.5) ;
\draw (#1,#2) rectangle (#1+2,#2+2) ;
\node[scale=0.75] at (#1+0.5,#2+0.5) {0};
\draw [-latex] (#1+1.5,#2+1.5) -- (#1+1.5,#2+0) ; 
\draw [-latex] (#1+1.5,#2+1.5) -- (#1+0,#2+1.5) ;
\draw [-latex] (#1+1,#2+1) -- (#1+1,#2+0) ; 
\draw [-latex] (#1+1,#2+1) -- (#1+0,#2+1) ; 
\draw [-latex] (#1+1.5,#2+1) -- (#1+2,#2+1) ; 
\draw [-latex] (#1+1,#2+1.5) -- (#1+1,#2+2) ;}

\newcommand{\robibluebastgauche}[2]{
\fill[blue!40] (#1+0.5,#2+0.5) rectangle (#1+1,#2+2) ;
\fill[blue!40] (#1+0.5,#2+0.5) rectangle (#1+2,#2+1) ;
\draw (#1,#2) rectangle (#1+2,#2+2) ;
\draw [-latex] (#1+0.5,#2+0.5) -- (#1+0.5,#2+2) ; 
\draw [-latex] (#1+0.5,#2+0.5) -- (#1+2,#2+0.5) ;
\draw [-latex] (#1+1,#2+1) -- (#1+1,#2+2) ; 
\draw [-latex] (#1+1,#2+1) -- (#1+2,#2+1) ; 
\draw [-latex] (#1+0.5,#2+1) -- (#1+0,#2+1) ; 
\draw [-latex] (#1+1,#2+0.5) -- (#1+1,#2+0) ;}
\newcommand{\robibluebastdroite}[2]{
\fill[blue!40] (#1+1.5,#2+0.5) rectangle (#1+1,#2+2) ;
\fill[blue!40] (#1+1.5,#2+0.5) rectangle (#1+0,#2+1) ;
\draw (#1,#2) rectangle (#1+2,#2+2) ;
\draw [-latex] (#1+1.5,#2+0.5) -- (#1+1.5,#2+2) ; 
\draw [-latex] (#1+1.5,#2+0.5) -- (#1+0,#2+0.5) ;
\draw [-latex] (#1+1,#2+1) -- (#1+1,#2+2) ; 
\draw [-latex] (#1+1,#2+1) -- (#1+0,#2+1) ; 
\draw [-latex] (#1+1.5,#2+1) -- (#1+2,#2+1) ; 
\draw [-latex] (#1+1,#2+0.5) -- (#1+1,#2+0) ;}
\newcommand{\robibluehauttgauche}[2]{
\fill[blue!40] (#1+2,#2+1) rectangle (#1+0.5,#2+1.5) ;
\fill[blue!40] (#1+1,#2+0) rectangle (#1+0.5,#2+1.5) ;
\draw (#1,#2) rectangle (#1+2,#2+2) ;
\draw [-latex] (#1+0.5,#2+1.5) -- (#1+0.5,#2+0) ; 
\draw [-latex] (#1+0.5,#2+1.5) -- (#1+2,#2+1.5) ;
\draw [-latex] (#1+1,#2+1) -- (#1+1,#2+0) ; 
\draw [-latex] (#1+1,#2+1) -- (#1+2,#2+1) ; 
\draw [-latex] (#1+0.5,#2+1) -- (#1+0,#2+1) ; 
\draw [-latex] (#1+1,#2+1.5) -- (#1+1,#2+2) ;}
\newcommand{\robibluehauttdroite}[2]{
\fill[blue!40] (#1+0,#2+1) rectangle (#1+1.5,#2+1.5) ;
\fill[blue!40] (#1+1,#2+0) rectangle (#1+1.5,#2+1.5) ;
\draw (#1,#2) rectangle (#1+2,#2+2) ;
\draw [-latex] (#1+1.5,#2+1.5) -- (#1+1.5,#2+0) ; 
\draw [-latex] (#1+1.5,#2+1.5) -- (#1+0,#2+1.5) ;
\draw [-latex] (#1+1,#2+1) -- (#1+1,#2+0) ; 
\draw [-latex] (#1+1,#2+1) -- (#1+0,#2+1) ; 
\draw [-latex] (#1+1.5,#2+1) -- (#1+2,#2+1) ; 
\draw [-latex] (#1+1,#2+1.5) -- (#1+1,#2+2) ;}

\newcommand{\robibluebasgauchek}[2]{
\fill[blue!40] (#1+0.5,#2+0.5) rectangle (#1+1,#2+2) ;
\fill[blue!40] (#1+0.5,#2+0.5) rectangle (#1+2,#2+1) ;
\draw (#1,#2) rectangle (#1+2,#2+2) ;
\draw [-latex] (#1+0.5,#2+0.5) -- (#1+0.5,#2+2) ; 
\draw [-latex] (#1+0.5,#2+0.5) -- (#1+2,#2+0.5) ;
\draw [-latex] (#1+1,#2+1) -- (#1+1,#2+2) ; 
\draw [-latex] (#1+1,#2+1) -- (#1+2,#2+1) ; 
\draw [-latex] (#1+0.5,#2+1) -- (#1+0,#2+1) ; 
\draw [-latex] (#1+1,#2+0.5) -- (#1+1,#2+0) ;}
\newcommand{\robibluebasdroitek}[2]{
\fill[blue!40] (#1+1.5,#2+0.5) rectangle (#1+1,#2+2) ;
\fill[blue!40] (#1+1.5,#2+0.5) rectangle (#1+0,#2+1) ;
\draw (#1,#2) rectangle (#1+2,#2+2) ;
\draw [-latex] (#1+1.5,#2+0.5) -- (#1+1.5,#2+2) ; 
\draw [-latex] (#1+1.5,#2+0.5) -- (#1+0,#2+0.5) ;
\draw [-latex] (#1+1,#2+1) -- (#1+1,#2+2) ; 
\draw [-latex] (#1+1,#2+1) -- (#1+0,#2+1) ; 
\draw [-latex] (#1+1.5,#2+1) -- (#1+2,#2+1) ; 
\draw [-latex] (#1+1,#2+0.5) -- (#1+1,#2+0) ;}
\newcommand{\robibluehautgauchek}[2]{
\fill[blue!40] (#1+2,#2+1) rectangle (#1+0.5,#2+1.5) ;
\fill[blue!40] (#1+1,#2+0) rectangle (#1+0.5,#2+1.5) ;
\draw (#1,#2) rectangle (#1+2,#2+2) ;
\draw [-latex] (#1+0.5,#2+1.5) -- (#1+0.5,#2+0) ; 
\draw [-latex] (#1+0.5,#2+1.5) -- (#1+2,#2+1.5) ;
\draw [-latex] (#1+1,#2+1) -- (#1+1,#2+0) ; 
\draw [-latex] (#1+1,#2+1) -- (#1+2,#2+1) ; 
\draw [-latex] (#1+0.5,#2+1) -- (#1+0,#2+1) ; 
\draw [-latex] (#1+1,#2+1.5) -- (#1+1,#2+2) ;}
\newcommand{\robibluehautdroitek}[2]{
\fill[blue!40] (#1+0,#2+1) rectangle (#1+1.5,#2+1.5) ;
\fill[blue!40] (#1+1,#2+0) rectangle (#1+1.5,#2+1.5) ;
\draw (#1,#2) rectangle (#1+2,#2+2) ;
\draw [-latex] (#1+1.5,#2+1.5) -- (#1+1.5,#2+0) ; 
\draw [-latex] (#1+1.5,#2+1.5) -- (#1+0,#2+1.5) ;
\draw [-latex] (#1+1,#2+1) -- (#1+1,#2+0) ; 
\draw [-latex] (#1+1,#2+1) -- (#1+0,#2+1) ; 
\draw [-latex] (#1+1.5,#2+1) -- (#1+2,#2+1) ; 
\draw [-latex] (#1+1,#2+1.5) -- (#1+1,#2+2) ;}

\newcommand{\robiredbasgauche}[2]{
\fill[red!40] (#1+0.5,#2+0.5) rectangle (#1+1,#2+2) ;
\fill[red!40] (#1+0.5,#2+0.5) rectangle (#1+2,#2+1) ;
\draw (#1,#2) rectangle (#1+2,#2+2) ;
\draw [-latex] (#1+0.5,#2+0.5) -- (#1+0.5,#2+2) ; 
\draw [-latex] (#1+0.5,#2+0.5) -- (#1+2,#2+0.5) ;
\draw [-latex] (#1+1,#2+1) -- (#1+1,#2+2) ; 
\draw [-latex] (#1+1,#2+1) -- (#1+2,#2+1) ; 
\draw [-latex] (#1+0.5,#2+1) -- (#1+0,#2+1) ; 
\draw [-latex] (#1+1,#2+0.5) -- (#1+1,#2+0) ;}
\newcommand{\robiredbasdroite}[2]{
\fill[red!40] (#1+1.5,#2+0.5) rectangle (#1+1,#2+2) ;
\fill[red!40] (#1+1.5,#2+0.5) rectangle (#1+0,#2+1) ;
\draw (#1,#2) rectangle (#1+2,#2+2) ;
\draw [-latex] (#1+1.5,#2+0.5) -- (#1+1.5,#2+2) ; 
\draw [-latex] (#1+1.5,#2+0.5) -- (#1+0,#2+0.5) ;
\draw [-latex] (#1+1,#2+1) -- (#1+1,#2+2) ; 
\draw [-latex] (#1+1,#2+1) -- (#1+0,#2+1) ; 
\draw [-latex] (#1+1.5,#2+1) -- (#1+2,#2+1) ; 
\draw [-latex] (#1+1,#2+0.5) -- (#1+1,#2+0) ;}
\newcommand{\robiredhautgauche}[2]{
\fill[red!40] (#1+2,#2+1) rectangle (#1+0.5,#2+1.5) ;
\fill[red!40] (#1+1,#2+0) rectangle (#1+0.5,#2+1.5) ;
\draw (#1,#2) rectangle (#1+2,#2+2) ;
\draw [-latex] (#1+0.5,#2+1.5) -- (#1+0.5,#2+0) ; 
\draw [-latex] (#1+0.5,#2+1.5) -- (#1+2,#2+1.5) ;
\draw [-latex] (#1+1,#2+1) -- (#1+1,#2+0) ; 
\draw [-latex] (#1+1,#2+1) -- (#1+2,#2+1) ; 
\draw [-latex] (#1+0.5,#2+1) -- (#1+0,#2+1) ; 
\draw [-latex] (#1+1,#2+1.5) -- (#1+1,#2+2) ;}
\newcommand{\robiredhautdroite}[2]{
\fill[red!40] (#1+0,#2+1) rectangle (#1+1.5,#2+1.5) ;
\fill[red!40] (#1+1,#2+0) rectangle (#1+1.5,#2+1.5) ;
\draw (#1,#2) rectangle (#1+2,#2+2) ;
\draw [-latex] (#1+1.5,#2+1.5) -- (#1+1.5,#2+0) ; 
\draw [-latex] (#1+1.5,#2+1.5) -- (#1+0,#2+1.5) ;
\draw [-latex] (#1+1,#2+1) -- (#1+1,#2+0) ; 
\draw [-latex] (#1+1,#2+1) -- (#1+0,#2+1) ; 
\draw [-latex] (#1+1.5,#2+1) -- (#1+2,#2+1) ; 
\draw [-latex] (#1+1,#2+1.5) -- (#1+1,#2+2) ;}

\newcommand{\robigraybasgauche}[2]{
\fill[gray!40] (#1+0.5, #2+0.5) rectangle (#1+1, #2+2) ;
\fill[gray!40] (#1+0.5, #2+0.5) rectangle (#1+2, #2+1) ;
\draw [-latex] (#1+0.5,#2+0.5) -- (#1+0.5,#2+2) ; 
\draw [-latex] (#1+0.5,#2+0.5) -- (#1+2,#2+0.5) ;
\draw [-latex] (#1+1,#2+1) -- (#1+1,#2+2) ; 
\draw [-latex] (#1+1,#2+1) -- (#1+2,#2+1) ; 
\draw [-latex] (#1+0.5,#2+1) -- (#1+0,#2+1) ; 
\draw [-latex] (#1+1,#2+0.5) -- (#1+1,#2+0) ;}
\newcommand{\robigraybasdroite}[2]{
\fill[gray!40] (#1+1.5,#2+0.5) rectangle (#1+1,#2+2) ;
\fill[gray!40] (#1+1.5,#2+0.5) rectangle (#1+0,#2+1) ;
\draw [-latex] (#1+1.5,#2+0.5) -- (#1+1.5,#2+2) ; 
\draw [-latex] (#1+1.5,#2+0.5) -- (#1+0,#2+0.5) ;
\draw [-latex] (#1+1,#2+1) -- (#1+1,#2+2) ; 
\draw [-latex] (#1+1,#2+1) -- (#1+0,#2+1) ; 
\draw [-latex] (#1+1.5,#2+1) -- (#1+2,#2+1) ; 
\draw [-latex] (#1+1,#2+0.5) -- (#1+1,#2+0) ;}
\newcommand{\robigrayhautgauche}[2]{
\fill[gray!40] (#1+2,#2+1) rectangle (#1+0.5,#2+1.5) ;
\fill[gray!40] (#1+1,#2+0) rectangle (#1+0.5,#2+1.5) ;
\draw [-latex] (#1+0.5,#2+1.5) -- (#1+0.5,#2+0) ; 
\draw [-latex] (#1+0.5,#2+1.5) -- (#1+2,#2+1.5) ;
\draw [-latex] (#1+1,#2+1) -- (#1+1,#2+0) ; 
\draw [-latex] (#1+1,#2+1) -- (#1+2,#2+1) ; 
\draw [-latex] (#1+0.5,#2+1) -- (#1+0,#2+1) ; 
\draw [-latex] (#1+1,#2+1.5) -- (#1+1,#2+2) ;}
\newcommand{\robigrayhautdroite}[2]{
\fill[gray!40] (#1+0,#2+1) rectangle (#1+1.5,#2+1.5) ;
\fill[gray!40] (#1+1,#2+0) rectangle (#1+1.5,#2+1.5) ;
\draw [-latex] (#1+1.5,#2+1.5) -- (#1+1.5,#2+0) ; 
\draw [-latex] (#1+1.5,#2+1.5) -- (#1+0,#2+1.5) ;
\draw [-latex] (#1+1,#2+1) -- (#1+1,#2+0) ; 
\draw [-latex] (#1+1,#2+1) -- (#1+0,#2+1) ; 
\draw [-latex] (#1+1.5,#2+1) -- (#1+2,#2+1) ; 
\draw [-latex] (#1+1,#2+1.5) -- (#1+1,#2+2) ;}

\newcommand{\supertiletwobasgauche}[2]{
\robibluebasgauche{#1+0}{#2+0};
\robibluebasgauche{#1+8}{#2+0};
\robibluebasgauche{#1+0}{#2+8};
\robibluebasgauche{#1+8}{#2+8};
\robiredbasgauche{#1+2}{#2+2};
\robiredbasgauche{#1+6}{#2+6};
\robibluehautgauche{#1+0}{#2+4};
\robibluehautgauche{#1+8}{#2+4};
\robibluehautgauche{#1+0}{#2+12};
\robibluehautgauche{#1+8}{#2+12};
\robiredhautgauche{#1+2}{#2+10};
\robibluebasdroite{#1+4}{#2+0};
\robibluebasdroite{#1+12}{#2+0};
\robibluebasdroite{#1+4}{#2+8};
\robibluebasdroite{#1+12}{#2+8};
\robiredbasdroite{#1+10}{#2+2};
\robibluehautdroite{#1+4}{#2+4};
\robibluehautdroite{#1+12}{#2+4};
\robibluehautdroite{#1+4}{#2+12};
\robibluehautdroite{#1+12}{#2+12};
\robiredhautdroite{#1+10}{#2+10};
\robitwobas{#1+2}{#2+0}
\robitwobas{#1+10}{#2+0}
\robitwobas{#1+6}{#2+2}
\robitwogauche{#1+0}{#2+2}
\robitwogauche{#1+2}{#2+6}
\robitwogauche{#1+0}{#2+10}
\robitwodroite{#1+12}{#2+2}
\robitwodroite{#1+12}{#2+10}
\robitwohaut{#1+2}{#2+12}
\robitwohaut{#1+10}{#2+12}
\robionehaut{#1+6}{#2+0}
\robionehaut{#1+6}{#2+4}
\robionegauche{#1+0}{#2+6}
\robionegauche{#1+4}{#2+6}
\robisixbas{#1+2}{#2+8}
\robisixdroite{#1+4}{#2+2}
\robisixhaut{#1+2}{#2+4}
\robisevendroite{#1+4}{#2+10}
\robithreehaut{#1+6}{#2+8}
\robisixhaut{#1+6}{#2+10}
\robithreehaut{#1+6}{#2+12}
\robisixgauche{#1+8}{#2+2}
\robisevengauche{#1+8}{#2+10}
\robisevenbas{#1+10}{#2+8}
\robisevenhaut{#1+10}{#2+4}
\robisixdroite{#1+10}{#2+6}
\robithreedroite{#1+8}{#2+6}
\robithreedroite{#1+12}{#2+6}}

\newcommand{\supertiletwobasdroite}[2]{
\robibluebasgauche{#1+0}{#2+0};
\robibluebasgauche{#1+8}{#2+0};
\robibluebasgauche{#1+0}{#2+8};
\robibluebasgauche{#1+8}{#2+8};
\robiredbasgauche{#1+2}{#2+2};
\robiredbasdroite{#1+6}{#2+6};
\robibluehautgauche{#1+0}{#2+4};
\robibluehautgauche{#1+8}{#2+4};
\robibluehautgauche{#1+0}{#2+12};
\robibluehautgauche{#1+8}{#2+12};
\robiredhautgauche{#1+2}{#2+10};
\robibluebasdroite{#1+4}{#2+0};
\robibluebasdroite{#1+12}{#2+0};
\robibluebasdroite{#1+4}{#2+8};
\robibluebasdroite{#1+12}{#2+8};
\robiredbasdroite{#1+10}{#2+2};
\robibluehautdroite{#1+4}{#2+4};
\robibluehautdroite{#1+12}{#2+4};
\robibluehautdroite{#1+4}{#2+12};
\robibluehautdroite{#1+12}{#2+12};
\robiredhautdroite{#1+10}{#2+10};
\robitwobas{#1+2}{#2+0}
\robitwobas{#1+10}{#2+0}
\robitwobas{#1+6}{#2+2}
\robitwogauche{#1+0}{#2+2}
\robitwodroite{#1+10}{#2+6}
\robitwogauche{#1+0}{#2+10}
\robitwodroite{#1+12}{#2+2}
\robitwodroite{#1+12}{#2+10}
\robitwohaut{#1+2}{#2+12}
\robitwohaut{#1+10}{#2+12}
\robionehaut{#1+6}{#2+0}
\robionehaut{#1+6}{#2+4}
\robionedroite{#1+12}{#2+6}
\robionedroite{#1+8}{#2+6}
\robisixbas{#1+2}{#2+8}
\robisixdroite{#1+4}{#2+2}
\robisixhaut{#1+2}{#2+4}
\robisevendroite{#1+4}{#2+10}
\robifourhaut{#1+6}{#2+8}
\robisevenhaut{#1+6}{#2+10}
\robifourhaut{#1+6}{#2+12}
\robisixgauche{#1+8}{#2+2}
\robisevengauche{#1+8}{#2+10}
\robisevenbas{#1+10}{#2+8}
\robisevenhaut{#1+10}{#2+4}
\robisixgauche{#1+2}{#2+6}
\robithreegauche{#1+4}{#2+6}
\robithreegauche{#1+0}{#2+6}}

\newcommand{\supertiletwohautdroite}[2]{
\robibluebasgauche{#1+0}{#2+0};
\robibluebasgauche{#1+8}{#2+0};
\robibluebasgauche{#1+0}{#2+8};
\robibluebasgauche{#1+8}{#2+8};
\robiredbasgauche{#1+2}{#2+2};
\robiredhautdroite{#1+6}{#2+6};
\robibluehautgauche{#1+0}{#2+4};
\robibluehautgauche{#1+8}{#2+4};
\robibluehautgauche{#1+0}{#2+12};
\robibluehautgauche{#1+8}{#2+12};
\robiredhautgauche{#1+2}{#2+10};
\robibluebasdroite{#1+4}{#2+0};
\robibluebasdroite{#1+12}{#2+0};
\robibluebasdroite{#1+4}{#2+8};
\robibluebasdroite{#1+12}{#2+8};
\robiredbasdroite{#1+10}{#2+2};
\robibluehautdroite{#1+4}{#2+4};
\robibluehautdroite{#1+12}{#2+4};
\robibluehautdroite{#1+4}{#2+12};
\robibluehautdroite{#1+12}{#2+12};
\robiredhautdroite{#1+10}{#2+10};
\robitwobas{#1+2}{#2+0}
\robitwobas{#1+10}{#2+0}
\robitwohaut{#1+6}{#2+10}
\robitwogauche{#1+0}{#2+2}
\robitwodroite{#1+10}{#2+6}
\robitwogauche{#1+0}{#2+10}
\robitwodroite{#1+12}{#2+2}
\robitwodroite{#1+12}{#2+10}
\robitwohaut{#1+2}{#2+12}
\robitwohaut{#1+10}{#2+12}
\robionebas{#1+6}{#2+12}
\robionebas{#1+6}{#2+8}
\robionedroite{#1+12}{#2+6}
\robionedroite{#1+8}{#2+6}
\robisixbas{#1+2}{#2+8}
\robisixdroite{#1+4}{#2+2}
\robisixhaut{#1+2}{#2+4}
\robisevendroite{#1+4}{#2+10}
\robifourbas{#1+6}{#2+4}
\robisevenbas{#1+6}{#2+2}
\robifourbas{#1+6}{#2+0}
\robisixgauche{#1+8}{#2+2}
\robisevengauche{#1+8}{#2+10}
\robisevenbas{#1+10}{#2+8}
\robisevenhaut{#1+10}{#2+4}
\robisevengauche{#1+2}{#2+6}
\robifourgauche{#1+4}{#2+6}
\robifourgauche{#1+0}{#2+6}}

\newcommand{\supertiletwohautgauche}[2]{
\robibluebasgauche{#1+0}{#2+0};
\robibluebasgauche{#1+8}{#2+0};
\robibluebasgauche{#1+0}{#2+8};
\robibluebasgauche{#1+8}{#2+8};
\robiredbasgauche{#1+2}{#2+2};
\robiredhautgauche{#1+6}{#2+6};
\robibluehautgauche{#1+0}{#2+4};
\robibluehautgauche{#1+8}{#2+4};
\robibluehautgauche{#1+0}{#2+12};
\robibluehautgauche{#1+8}{#2+12};
\robiredhautgauche{#1+2}{#2+10};
\robibluebasdroite{#1+4}{#2+0};
\robibluebasdroite{#1+12}{#2+0};
\robibluebasdroite{#1+4}{#2+8};
\robibluebasdroite{#1+12}{#2+8};
\robiredbasdroite{#1+10}{#2+2};
\robibluehautdroite{#1+4}{#2+4};
\robibluehautdroite{#1+12}{#2+4};
\robibluehautdroite{#1+4}{#2+12};
\robibluehautdroite{#1+12}{#2+12};
\robiredhautdroite{#1+10}{#2+10};
\robitwobas{#1+2}{#2+0}
\robitwobas{#1+10}{#2+0}
\robitwohaut{#1+6}{#2+10}
\robitwogauche{#1+0}{#2+2}
\robitwogauche{#1+2}{#2+6}
\robitwogauche{#1+0}{#2+10}
\robitwodroite{#1+12}{#2+2}
\robitwodroite{#1+12}{#2+10}
\robitwohaut{#1+2}{#2+12}
\robitwohaut{#1+10}{#2+12}
\robionebas{#1+6}{#2+12}
\robionebas{#1+6}{#2+8}
\robionegauche{#1+0}{#2+6}
\robionegauche{#1+4}{#2+6}
\robisixbas{#1+2}{#2+8}
\robisixdroite{#1+4}{#2+2}
\robisixhaut{#1+2}{#2+4}
\robisevendroite{#1+4}{#2+10}
\robithreebas{#1+6}{#2+4}
\robisixbas{#1+6}{#2+2}
\robithreebas{#1+6}{#2+0}
\robisixgauche{#1+8}{#2+2}
\robisevengauche{#1+8}{#2+10}
\robisevenbas{#1+10}{#2+8}
\robisevenhaut{#1+10}{#2+4}
\robisevendroite{#1+10}{#2+6}
\robifourdroite{#1+8}{#2+6}
\robifourdroite{#1+12}{#2+6}}

\newcommand{\supertilesynchrotwo}[2]{
\fill[gray!90] (#1-2,#2-2) rectangle (#1+32,#2+32);
\fill[gray!40] (#1+0,#2+0) rectangle (#1+30,#2+30);
\fill[red!40] (#1+0,#2+0) rectangle (#1+14,#2+14);
\fill[orange!40] (#1+16,#2+16) rectangle (#1+30,#2+30);
\fill[purple!40] (#1+0,#2+16) rectangle (#1+14,#2+30);
\fill[yellow!40] (#1+16,#2+0) rectangle (#1+30,#2+14); 
\fill[gray!90] (#1+2,#2+2) rectangle (#1+12,#2+12);
\fill[gray!90] (#1+2,#2+18) rectangle (#1+12,#2+28);
\fill[gray!90] (#1+18,#2+18) rectangle (#1+28,#2+28);
\fill[gray!90] (#1+18,#2+2) rectangle (#1+28,#2+12);
\fill[gray!40] (#1+4,#2+20) rectangle (#1+10,#2+26);
\fill[gray!40] (#1+20,#2+20) rectangle (#1+26,#2+26);
\fill[gray!40] (#1+20,#2+4) rectangle (#1+26,#2+10);
\fill[gray!40] (#1+4,#2+4) rectangle (#1+10,#2+10);
\fill[red!40] (#1+4,#2+4) rectangle (#1+6,#2+6);
\fill[yellow!40] (#1+8,#2+4) rectangle (#1+10,#2+6);
\fill[purple!40] (#1+4,#2+8) rectangle (#1+6,#2+10);
\fill[orange!40] (#1+8,#2+8) rectangle (#1+10,#2+10);
\fill[red!40] (#1+20,#2+4) rectangle (#1+22,#2+6);
\fill[yellow!40] (#1+24,#2+4) rectangle (#1+26,#2+6);
\fill[purple!40] (#1+20,#2+8) rectangle (#1+22,#2+10);
\fill[orange!40] (#1+24,#2+8) rectangle (#1+26,#2+10);
\fill[red!40] (#1+20,#2+20) rectangle (#1+22,#2+22);
\fill[yellow!40] (#1+24,#2+20) rectangle (#1+26,#2+22);
\fill[purple!40] (#1+20,#2+24) rectangle (#1+22,#2+26);
\fill[orange!40] (#1+24,#2+24) rectangle (#1+26,#2+26);
\fill[red!40] (#1+4,#2+20) rectangle (#1+6,#2+22);
\fill[yellow!40] (#1+8,#2+20) rectangle (#1+10,#2+22);
\fill[purple!40] (#1+4,#2+24) rectangle (#1+6,#2+26);
\fill[orange!40] (#1+8,#2+24) rectangle (#1+10,#2+26);}

\newcommand{\supertilesynchroone}[2]{
\fill[gray!90] (#1+0,#2+0) rectangle (#1+10,#2+10);
\fill[gray!40] (#1+2,#2+2) rectangle (#1+8,#2+8);
\fill[red!40] (#1+2,#2+2) rectangle (#1+4,#2+4);
\fill[yellow!40] (#1+6,#2+2) rectangle (#1+8,#2+4);
\fill[purple!40] (#1+2,#2+6) rectangle (#1+4,#2+8);
\fill[orange!40] (#1+6,#2+6) rectangle (#1+8,#2+8);
}

\newcommand{\supertilesynchrothree}{
\fill[gray!90] (-2,-2) rectangle (2*64,2*64);
\fill[gray!40] (0,0) rectangle (2*63,2*63);
\fill[red!40] (0,0) rectangle (2*31,2*31);
\fill[orange!40] (2*32,2*32) rectangle (2*63,2*63);
\fill[yellow!40] (2*32,2*31) rectangle (2*63,0);
\fill[purple!40] (2*31,2*32) rectangle (0,2*63);
\supertilesynchrotwo{16}{16}
\supertilesynchrotwo{16}{16+4*16}
\supertilesynchrotwo{16+4*16}{16+4*16}
\supertilesynchrotwo{16+4*16}{16}
\supertilesynchroone{66}{66}
\supertilesynchroone{-4*16+66}{66}
\supertilesynchroone{-3*16+66}{66}
\supertilesynchroone{-2*16+66}{66}
\supertilesynchroone{-1*16+66}{66}
\supertilesynchroone{16+66}{66}
\supertilesynchroone{2*16+66}{66}
\supertilesynchroone{3*16+66}{66}
\supertilesynchroone{-4*16+66}{16+66}
\supertilesynchroone{-3*16+66}{16+66}
\supertilesynchroone{-2*16+66}{16+66}
\supertilesynchroone{-1*16+66}{16+66}
\supertilesynchroone{16+66}{16+66}
\supertilesynchroone{66}{16+66}
\supertilesynchroone{2*16+66}{16+66}
\supertilesynchroone{3*16+66}{16+66}
\supertilesynchroone{-4*16+66}{2*16+66}
\supertilesynchroone{-3*16+66}{2*16+66}
\supertilesynchroone{-2*16+66}{2*16+66}
\supertilesynchroone{-1*16+66}{2*16+66}
\supertilesynchroone{66}{2*16+66}
\supertilesynchroone{16+66}{2*16+66}
\supertilesynchroone{2*16+66}{2*16+66}
\supertilesynchroone{3*16+66}{2*16+66}

\supertilesynchroone{-4*16+66}{3*16+66}
\supertilesynchroone{-3*16+66}{3*16+66}
\supertilesynchroone{-2*16+66}{3*16+66}
\supertilesynchroone{-1*16+66}{3*16+66}
\supertilesynchroone{66}{3*16+66}
\supertilesynchroone{16+66}{3*16+66}
\supertilesynchroone{2*16+66}{3*16+66}
\supertilesynchroone{3*16+66}{3*16+66}

\supertilesynchroone{-4*16+66}{-1*16+66}
\supertilesynchroone{-3*16+66}{-1*16+66}
\supertilesynchroone{-2*16+66}{-1*16+66}
\supertilesynchroone{-1*16+66}{-1*16+66}
\supertilesynchroone{66}{-1*16+66}
\supertilesynchroone{16+66}{-1*16+66}
\supertilesynchroone{2*16+66}{-1*16+66}
\supertilesynchroone{3*16+66}{-1*16+66}

\supertilesynchroone{-4*16+66}{-2*16+66}
\supertilesynchroone{-3*16+66}{-2*16+66}
\supertilesynchroone{-2*16+66}{-2*16+66}
\supertilesynchroone{-1*16+66}{-2*16+66}
\supertilesynchroone{66}{-2*16+66}
\supertilesynchroone{16+66}{-2*16+66}
\supertilesynchroone{2*16+66}{-2*16+66}
\supertilesynchroone{3*16+66}{-2*16+66}

\supertilesynchroone{-4*16+66}{-3*16+66}
\supertilesynchroone{-3*16+66}{-3*16+66}
\supertilesynchroone{-2*16+66}{-3*16+66}
\supertilesynchroone{-1*16+66}{-3*16+66}
\supertilesynchroone{66}{-3*16+66}
\supertilesynchroone{16+66}{-3*16+66}
\supertilesynchroone{2*16+66}{-3*16+66}
\supertilesynchroone{3*16+66}{-3*16+66}

\supertilesynchroone{-4*16+66}{-4*16+66}
\supertilesynchroone{-3*16+66}{-4*16+66}
\supertilesynchroone{-2*16+66}{-4*16+66}
\supertilesynchroone{-1*16+66}{-4*16+66}
\supertilesynchroone{66}{-4*16+66}
\supertilesynchroone{16+66}{-4*16+66}
\supertilesynchroone{2*16+66}{-4*16+66}
\supertilesynchroone{3*16+66}{-4*16+66}
}

\renewcommand{\headrulewidth}{0pt}
\renewcommand{\footrulewidth}{0pt}

\newcommand{\IFF}[2]
{
\left[ #1 , #2 \right]
}

\begin{document}

\maketitle

\begin{abstract} We prove in this text that
the possible entropy dimensions of 
minimal $\Z^3$-SFT 
are the $\Delta_2$-computable 
 numbers in $[0,2]$,
using Goldbach's theorem on Fermat numbers.
\end{abstract}

\section{Introduction}

Multidimensional subshifts of finite type (SFT) are discrete dynamical 
systems described as the action of the shift 
on a compact set of symbol displays on an infinite 
regular grid. This set is defined 
by a finite set of local rules on the symbols. 
It has been established that 
their dynamicals 
are related to computability. The most emblematic 
result in this sense is the characterization 
of the possible values of the entropy of 
these systems, by M. Hochman and T. Meyerovitch~\cite{Hochman-Meyerovitch-2010}, as the 
set of $\Pi_1$-computable numbers. The $\Pi_1$-computability means that the number can be approximated 
from above by a sequence of rational numbers produced 
by an algorithm. This result 
was followed for instance by the characterization 
of the entropy dimensions~\cite{Meyerovitch2011}
of the possible sets of 
periods~\cite{Jeandel-Vanier-2014}, with
similar recursion-theoretic criteria. These results 
follow a common outline, which consists in 
implementing Turing machines in hierarchical structures
of particular SFT.

A recent trend is to see the effect 
of dynamical restrictions on these characterization 
results. It was already proved in~\cite{Hochman-Meyerovitch-2010} that the values of entropy of strongly irreducible 
bidimensional subshifts are all computable numbers. 
This means that there is an algorithm which on input $n$ 
outputs an approximation up to $\frac{1}{n}$.
Moreover, the authors proved that this can 
be done in time $e^{O(n^2)}$.
Then R. Pavlov and M. Schraudner provided a construction 
that allows to realize a subclass of this class. 

Other types of restrictions have been studied, 
for instance minimality.
For instance, in~\cite{HV14}, the authors proved that minimal subshifts can exhibit 
complex computational behaviors, by exhibiting minimal $\Z^d$-subshifts
that have complex Turing spectrum.
In~\cite{JLK17}, the authors provided a construction that allows the realization of all the
non-negative real numbers smaller than $2$ as 
entropy dimension 
of minimal $\Z^2$-subshifts (which are not though subshifts of finite type). 

The restriction of minimality acts in a different 
way on SFT. Notably, the entropy of 
minimal SFT is zero. As a consequence, this 
invariant is not pertinent for the study 
of these systems. Moreover, one 
could think that this prevents embedding 
Turing computations in these subshifts, 
since entropy is often assimilated with complexity. 

We prove in this text that this is not the case. 
First, this is possible to embed Turing computation 
in minimal SFT. Second, we show that the entropy 
dimension is a more pertinent invariant for 
these subshifts.

To our knowledge, the 
only other constructions of minimal SFT are 
due to B. Durand and A. Romashchenko~\cite{DR17}. 
They use a fixed-point construction in which 
are implemented computing machines that check 
if no forbidden pattern in a recursively enumerable 
set appear in bi-infinite words on the alphabet $\{0,1\}$.
Since in these constructions the 
computation areas of the machines are sparse, 
the degenerated behaviors are simple.
Controling the growth of the computing 
units, they could attribute the function 
of some particular sub-units in order to 
simulate these behaviors in any computing unit.
The minimality of the architecture 
used for the control on the apparition of 
the forbidden words is ensured by this simulation. 

The idea of simulation is present in both 
constructions. However, the construction 
of \cite{DR17} relies on a very rigid architecture. 
We propose here a more flexible way to 
ensure the minimality, although complex to 
formulate. This allows 
the simulation of a consequent set 
of patterns.

In this text, we prove the following theorem: 

\begin{theorem}
\label{thm.dim.entropique.intro}
The possible entropy dimension of minimal $\Z^3$-SFT are 
the $\Delta_2$-computable numbers in $[0,2]$.
\end{theorem}

Thus, there is a difference with the set of numbers that are 
the entropy of $\Z^3$-SFT: 
the $\Delta_2$-computable numbers in $[0,3]$. 

The construction used to prove the realization
part of this characterization
is an adaptation of the construction 
presented in~\cite{Meyerovitch2011}. 
This construction is not minimal for the 
reason that pathological 
behaviors of the Turing machines can appear, and 
that it uses 
a display of random bits that breaks the 
minimality property.

As in our previous article~\cite{GS17}, 
there is an attractive analogy between counters 
used in the construction and DNA. This 
comes from the 
separation into coding part and non-coding one. 
This analogy suggests that the non-coding part 
is implied in global properties of a living system.

The structure layer used in this construction 
is a $3d$ version of the Robinson subshift constructed 
with three copies of a rigid 
version of this subshift. 
It exhibits 
similar hierarchical structures as 
the two-dimensional version. 
Although we don't prove how here, 
this structure allows, by adding colors, 
some structures used in~\cite{Hochman-2009} 
and \cite{Pavlov-Schraudner-2014} 
in order to construct 
$\Z^3$-SFTs to be recovered.

Let us also remark that some 
of the mechanisms described in the 
construction presented in this text can be reformulated using 
substitutions and S-adic systems. However, we don't use Mozes theorem (proved in~\cite{Mozes1989})
or the result proved in~\cite{aubrun-sablik-14}. 
(this theorem 
states that multidimensional S-adic systems 
are sofic) for two reasons. 
The first one is that some 
of the substitution 
mechanisms are localized in restricted 
parts in each configuration. 
When the mechanism is global, there 
is still an obstacle for the use 
of these theorems:
the need of more precise properties on the SFT 
than stated in these theorems.
One of these properties is the repetition, in each configuration of the SFT, of patterns 
whose size some power of two, with period equal as well to some power of two. 
We instead use directly similar techniques as~\cite{Mozes1989} in the construction.

This text is organized as follows: In section~\ref{sec.obstruction.min}, we prove 
that the entropy dimensions of a minimal SFT are smaller than $d-1$. Then in section~\ref{sec.realization.min}, 
we prove the realization part of the characterization.

\section{Definitions}

In this section, we recall some definitions of 
subshifts, entropy dimension, 
$\Delta_2$-computable numbers. 

\subsection{Subshifts dynamical systems}

Let $\mathcal{A}$ be some finite set, called
\textbf{alphabet}. Let $d \ge 1$ be an integer.
The set 
$\mathcal{A}^{\Z^d}$ is a topological 
space with the product of the discrete 
topology on $\mathcal{A}$. Its 
elements are called \textbf{configurations}. 
We denote $(\vec{e}^1 , ... , \vec{e}^d)$
the canonical sequence of generators 
of $\Z^d$. Let us 
denote $\sigma$ the action 
of $\Z^d$ on this space defined 
by the following equality for all $\vec{u} \in 
\Z^d$ and $x$ element of the space:
\[\left(\sigma^{\vec{u}}.x\right)_{\vec{v}} 
= x_{\vec{v}+\vec{u}}.\]
A compact subset 
$X$ of this space is called a \textbf{subshift} 
when this subset is stable 
under the action of the shift. This means 
that for all $\vec{u} \in \Z^d$: 
\[\sigma^{\vec{u}}.X \subset X.\]

Consider some finite subset $\mathbb{U}$ 
of $\Z^d$. An element $p$ 
of $\mathcal{A}^{\mathbb{U}}$ 
is called a \textbf{pattern} on \textbf{support}
$\mathbb{U}$. This pattern \textbf{appears} in 
a configuration $x$ when there 
exists a translate $\mathbb{V}$ of $\mathbb{U}$ 
such that $x_{\mathbb{V}}=p$.
It appears in a subshift $X$ when it appears 
in a configuration of $X$. 
The set of patterns of $X$ that appear in 
it is called the \textbf{language} of $X$. 
The number of patterns on 
support $\llbracket 1,n \rrbracket ^d$ that 
appear in $X$ 
is denoted $N_n (X)$.

We say that a subshift $X$ is 
\textbf{minimal} when any pattern 
in its language appears in any of its configurations.

A subshift $X$ defined by forbidding patterns 
in some finite set $\mathcal{F}$ 
to appear in the configurations, 
formally: 
\[X= \bigcap_{\mathbb{U}
\subset \Z^2} \left\{ x \in \mathcal{A}^{\Z^2} : 
x_{\mathbb{U}} \notin \mathcal{F} \right\}\]
is called 
a subshift of \textbf{finite type} (SFT).

\subsection{Computability notions}

\begin{definition}
A real number $x$ is said to be $\Delta_2$-computable 
when there exists a Turing machine 
which given as input an integer $n$ 
outputs a rational number $r_n$ 
such that $x = \lim_n r_n$.
\end{definition}

\begin{definition}
A sequence $(a_n)_n \in \{0,1\}^{\N}$ is said $\Pi_1$-computable if there exists 
a Turing machine that, taking as input a 
couple of integers $(n,i)$, outputs some 
$\epsilon_{n,i} \in \{0,1\}$ such 
that for all $n$, $a_n = \inf_i \epsilon_{n,i}$.
\end{definition}

The following lemma establishes a link between $\Delta_2$-computable 
real numbers and $\Pi_1$-computable sequences.

\begin{lemma} \label{lem.density} A real number $z \in [0,2]$ 
is $\Delta_2$-computable if and only if 
there exists some $\Pi_1$-computable sequence $(a_j)_j$
such that 
\[z=\lim_{n \rightarrow \infty} \frac{2}{n} \sum_{j=1}^{n} a_j\]
\end{lemma}

This lemma is stated as part of 
Lemma 4.1 in~\cite{Meyerovitch2011}.

\subsection{Entropy dimensions}

The \textbf{upper entropy dimension}  
of $X$ is the number
\[\overline{D}_h (X)
= \limsup_{n} \frac{\log_2 \circ \log_2 (N_n (X))}{\log_2 (n)}.
\]
The \textbf{lower entropy dimension} of $X$
is: 
\[\underline{D}_h (X)
= \liminf_{n} \frac{\log_2 \circ \log_2 (N_n (X))}{\log_2 (n)}.
\]
When these numbers are equal, 
they are referred to as the 
\textbf{entropy dimension} of $X$, and denoted $D_h (X)$.

\begin{proposition}[\cite{Meyerovitch2011}]
\label{proposition.conjugacy.invariant}
The upper and lower entropy dimensions are topological invariants, 
as well as the existence and value of the entropy dimension.
\end{proposition}

\begin{theorem}[\cite{Meyerovitch2011}]
\label{thm.meyerovitch}
The possible values of the entropy dimension 
for $\Z^d$-SFT 
are the $\Delta_2$-computable 
numbers in $[0,d]$.
\end{theorem}

We state the following proposition 
in order to show that 
the entropy dimension 
is very different from the 
entropy. This type 
of properties has an impact 
on the techniques used in 
order to realize numbers 
as topological invariants 
of subshifts.

\begin{proposition}
\label{proposition.sup.entropy.dimension}
Let $X$ and $Z$ be two 
$\Z^d$-subshifts having 
entropy dimension. The subshift 
$X \times Z$ has 
entropy dimension:
\[D_h (X \times Z) = 
\max (D_h (X),D_h (Z)).\]
\end{proposition}

\begin{proof}
We have the following for 
all $n \ge 1$:
\[N_n (X \times Z) = 
N_n (X) \times N_n (Z).\]
As a consequence, 
\[\log_2 (N_n (X \times Z)) 
= \log_2 (N_n (X))+ 
\log_2 (N_n (Z)).\]
Thus we have 
\[
\frac{\max (\log_2 \circ \log_2 (N_n (X)), \log_2 \circ \log_2 (N_n(Z)))}{\log_2 (n)} \le \frac{\log_2 \circ \log_2 (N_n (X \times Z))}{\log_2 (n)}\]
and 
\[\frac{\log_2 \circ \log_2 (N_n (X \times Z))}{\log_2 (n)}
\le  \frac{ \log_2 ( 2 \log_2 ( 
\max(N_n (X),N_n(Z)))}{\log_2 (n)}
\]
From this follows that 
\[\frac{\log_2 \circ \log_2 (N_n (X \times Z))}{\log_2 (n)} \le 
\frac{ \log_2 (2) + \max (\log_2 \circ \log_2 (N_n (X)), \log_2 \circ \log_2 (N_n(Z)))}{\log_2 (n)}.\]
As a consequence of the previous 
inequalities, 
we have, taking $n \rightarrow +\infty$,
\[D_h (X\times Z) = 
\max (D_h (X),D_h(Z)).\]
\end{proof}

\section{Robinson subshift - a rigid version}

The Robinson subshift was constructed by R. Robinson~\cite{R71} in order to prove 
undecidability results. It has been used by in other constructions 
of subshifts of finite type as a structure layer~\cite{Pavlov-Schraudner-2014}.

In this section, we present a version of this subshift which is 
adapted to constructions under the 
dynamical constraints 
that we consider. In order to understand this 
section, it is preferable to read before the description 
of the Robinson subshift done in~\cite{R71}.
Some results are well known and we don't give a proof.
We refer instead to the initial article of R. Robinson. 

Let us denote $X_{adR}$ this subshift, which is constructed as the product of two layers.
We present the first layer in Section~\ref{sec.valued.robinson}, and then 
describe some hierarchical structures appearing in this layer in Section~\ref{sec.hierarchical.structures}.
In Section~\ref{sec.alignment.positioning}, we describe the second layer. This layer allows 
adding rigidity to 
the first layer, in order 
to enforce dynamical properties.

\subsection{\label{sec.valued.robinson} Robinson layer}

The first layer has the following {\textit{symbols}}, and their transformation by 
rotations by $\frac{\pi}{2}$, $\pi$ or $\frac{3\pi}{2}$: 

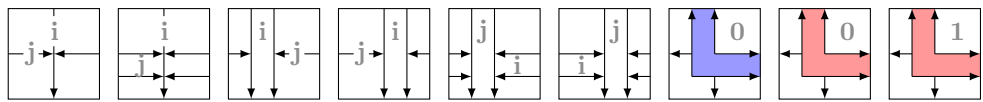
\[
\]

The symbols $i$ and $j$ can have 
value $0,1$ and 
are attached respectively to vertical and 
horizontal arrows. In the text, we refer to this value as 
the value of the \textbf{$0,1$-counter}. In 
order to simplify the representations, these 
values will often be omitted on the figures. \bigskip

In the text we will often designate as \textbf{corners} the two last symbols.
The other ones are called \textbf{arrows symbols} and are 
specified by the number of arrows in the symbol. 
For instance a six arrows symbols are the images by rotation of the fifth and sixth symbols. \bigskip

The {\textit{rules}} 
are the following ones: \begin{enumerate} \item the outgoing arrows and incoming ones 
correspond for two adjacent symbols. For instance, 
the pattern 
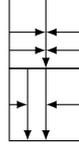
\[
\]
is forbidden, but the pattern 
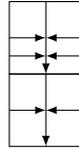
\[%
\]
is allowed. 
\item in every $2 \times 2$ square there is a blue 
symbol and the presence of a blue symbol in position $\vec{u} \in \Z^2$ 
forces the presence of a blue symbol in the positions $\vec{u}+(0,2), \vec{u}-(0,2), 
\vec{u}+(2,0)$ and $\vec{u}-(2,0)$.
\item on a position having 
mark $(i,j)$, the first coordinate is transmitted 
to the horizontally adjacent positions and the second one is transmitted 
to the vertically adjacent positions.
\item on a six arrows symbol, like 
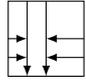
\[%
,\] 
or a five arrow symbol, 
like 
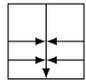
\[%
,\]
the marks $i$ and $j$ are different.
\end{enumerate} \bigskip

The Figure~\ref{figure.order2supertile} shows some pattern in the language of 
this layer. The subshift on this alphabet 
and generated by these rules is denoted $X_R$: this 
is the Robinson subshift.\bigskip

The main aspect of this subshift is 
the following property: 

\begin{definition}
A $\Z^d$-subshift $X$ is said aperiodic when 
for all configuration $x$ in the subshift, 
and $\vec{u} \in \Z^d \backslash (0,0)$, 
\[{\sigma}^{\vec{u}} (x) \neq x.\]
\end{definition}

\begin{theorem}[\cite{R71}]
The subshift $X_R$ is non-empty and aperiodic.
\end{theorem}

In the following, we state some properties of 
this subshift. The proofs of these properties 
can also be found in \cite{R71}.

\subsection{\label{sec.hierarchical.structures} 
Hierarchical structures}

In this section we describe some observable hierarchical 
structures in the elements of the Robinson subshift.

Let us recall that for 
all $d \ge 1$ and $k \ge 1$, 
we denote $\mathbb{U}^{(d)}_{k}$
the set $\llbracket 0,k-1 \rrbracket^d$.

\subsubsection{Finite supertiles}

\begin{figure}[h!]
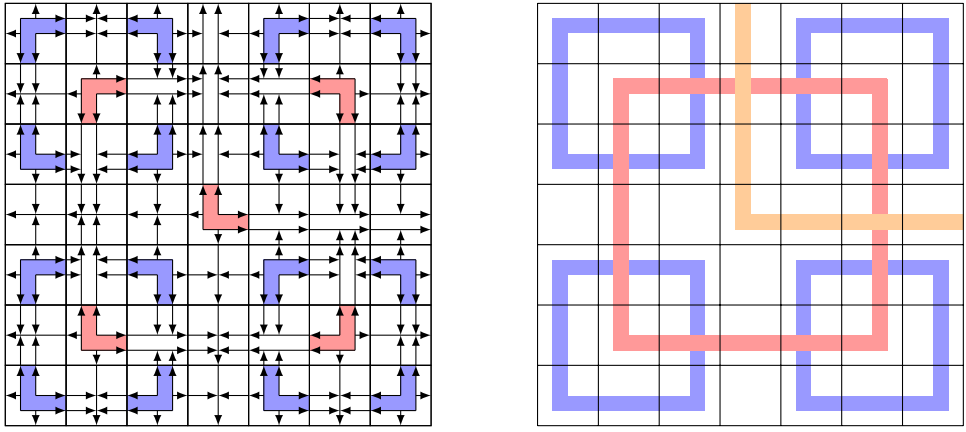

\[
\]
\caption{The south west order $2$ supertile denoted $St_{sw}(2)$ 
and petals intersecting it.}\label{figure.order2supertile}
\end{figure} \bigskip

Let us define by induction the south west 
(resp. south east, north west, north east) 
\define{supertile of order $n\in\N$}. For $n=0$, one has
\[
St_{sw} (0)=\vbox to 14pt{\hbox{
%
}}.
\] 

For $n\in\N$, the support of the supertile $St_{sw} (n+1)$ 
(resp. $St_{se} (n+1)$, $St_{nw} (n+1)$, $St_{ne} (n+1)$) is 
$\U^{(2)}_{2^{n+2}-1}$. 
On position $\vec{u}=(2^{n+1}-1,2^{n+1}-1)$ write
\[
St_{sw} (n+1) _{\vec{u}}=\vbox to 14pt{\hbox{
%
}}.
\]

Then complete the supertile such that the restriction 
to $\mathbb{U}^{(2)}_{2^{n+1}-1}$ (resp. 
$(2^{n+1},0)+\mathbb{U}^{(2)}_{2^{n+1}-1}$, 
 $(0,2^{n+1})+\mathbb{U}^{(2)}_{2^{n+1}-1}$,
$(2^{n+1},2^{n+1})+\mathbb{U}^{(2)}_{2^{n+1}-1}$)
is 
$St_{sw}(n)$ (resp. $St_{se} (n)$, $St_{nw} (n)$,
$St_{ne} (n)$).

Then complete the cross with 
the symbol \[%
\]
or with the symbol
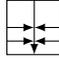
\[%
\] in the south vertical arm with the first symbol when there is one incoming arrow, and the second when there are two. 
The other arms are completed in a similar way. For instance, Figure \ref{figure.order2supertile}
shows the south west supertile of order two.

\begin{proposition}[\cite{R71}]
For all configuration $x$, if an order $n$ supertile 
appears in this configuration, 
then there is an order $n+1$ supertile, having 
this order $n$ supertile as sub-pattern, which appears 
in the same configuration.
\end{proposition}

\subsubsection{Infinite supertiles}

\begin{figure}[ht]
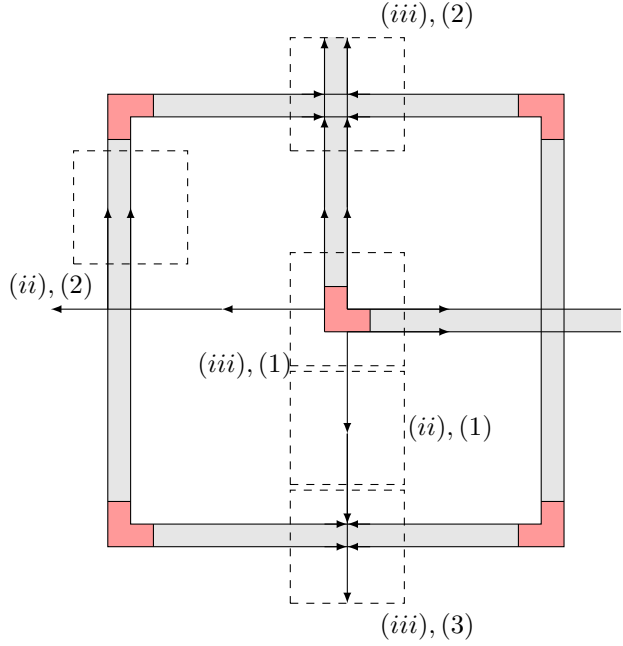

\[%
\]
\caption{\label{fig.infinite.supertiles} Correspondence 
between infinite supertiles and sub-patterns of order $n$ supertiles. The whole picture 
represents a schema of some finite order supertile.}
\end{figure}

Let $x$ be a configuration in the first layer and consider 
the equivalence relation $\sim_x$ on $\Z^2$ defined by 
$\vec{i} \sim_x \vec{j}$ if there is a finite
supertile in $x$ which contains $\vec{i}$ and $\vec{j}$.
An \define{infinite order} supertile is an infinite 
pattern over an equivalence 
class of this relation. Each configuration is amongst the 
following types (with types corresponding with 
types numbers on Figure~\ref{fig.infinite.supertiles}): 
\begin{itemize}
\item[(i)] A unique infinite order supertile which covers $\Z^2$.
\item[(ii)] Two infinite order supertiles separated by a line or a column with only three-arrows symbols
(1) or only four arrows symbols (2).  
In such a configuration, the order $n$
finite supertiles appearing in the two 
infinite supertiles 
are not necessary aligned, whereas this 
is the case in a type (i) or (iii) configuration.
\item[(iii)] Four infinite order supertiles, separated by a cross, whose center is superimposed with:
\begin{itemize}
\item a red symbol, and arms are filled with arrows symbols induced by the red one. (1)
\item a six arrows symbol, and arms are filled with double arrow symbols 
induced by this one. (2)
\item a five arrow symbol, and arms are filled with double arrow symbols and simple 
arrow symbols induced 
by this one. (3)
\end{itemize}
\end{itemize} 

Informally, the types of infinite 
supertiles correspond to configurations 
that are limits (for type $(ii)$ infinite supertiles 
this will be true after alignment 
[Section~\ref{sec.alignment.positioning}]) of 
a sequence of configurations centered on 
particular sub-patterns of 
finite supertiles of order $n$. 
This correspondence is illustrated on 
Figure~\ref{fig.infinite.supertiles}.
We notice this fact so that it helps 
to understand how 
patterns in configurations having multiple 
infinite supertiles are sub-patterns of 
finite supertiles. \bigskip

We say that a 
pattern $p$ on support $\mathbb{U}$ appears 
periodically in the horizontal (resp. vertical) direction 
in a configuration $x$ of a subshift $X$
when there exists some $T>0$ and 
$\vec{u}_0 \in \Z^2$ 
such that 
for all $k \in \Z$, 
\[x_{\vec{u}_0+\mathbb{U}+kT(1,0)}=p\]
(resp. $x_{\vec{u}_0+\mathbb{U}+kT(0,1)}=p$).
The number $T$ is called the period of this periodic appearance.

\begin{lemma}[\cite{R71}] 
\label{lem.repetition.supertiles}
For all $n$ and $m$ integers 
such that $n \ge m$, any order $m$ supertile 
appears periodically, horizontally and vertically, in 
any supertile of order $n \ge m$ with period $2^{m+2}$.
This is also true inside any infinite supertile. 
\end{lemma}

\subsubsection{Petals}

For a configuration 
$x$ of the Robinson subshift some finite subset of $\Z^2$ 
which has the following properties is called 
a \textbf{petal}.

\begin{itemize}
\item this set is minimal with respect to the inclusion, 
\item it contains some 
symbol with more than three arrows, 
\item  if a position is in the petal, 
the next position in the direction, or the opposite one, 
of the double arrows, is also in it,
\item and in the case of a six arrows 
symbol, the previous 
property is true only for one couple of arrows.
\end{itemize}

These sets are represented on 
the figures as squares joining 
four corners when these corners 
have the right orientations.
\bigskip

Petals containing blue symbols are called 
order $0$ petals. Each one intersect 
a unique greater order petal.
The other ones intersect 
four smaller petals and a greater 
one: if the intermediate petal is of order $n \ge 1$, 
then the four smaller are of order $n-1$ and the greatest 
one is of order $n+1$. Hence they 
form a hierarchy, and 
we refer to this in the text as 
the \textbf{petal hierarchy} (or hierarchy).

We usually call 
the petals valued with $1$ 
\textbf{support petals}, 
and the other ones 
are called 
\textbf{transmission petals}.

\begin{lemma}[\cite{R71}]
For all $n$, an order $n$ petal has size $2^{n+1}+1$.
\end{lemma}

We call order $n$ \textbf{two dimensional cell} the part 
of $\Z^2$ which is enclosed in an order 
$2n+1$ petal, for $n \ge 0$.
We also sometimes refer 
to the order $2n+1$ petals 
as the cells borders.

In particular, order $n \ge 0$ 
two-dimensional cells 
have size $4^{n+1}+1$ and 
repeat periodically with 
period $4^{n+2}$, vertically 
and horizontally, 
in every cell or supertile 
having greater order. 

See an illustration on 
Figure~\ref{figure.order2supertile}.

\subsection{ \label{sec.alignment.positioning} 
Alignment positioning
} 

If a configuration of the first layer has two infinite order 
supertiles, then the two sides of the column or line which separates them are 
non dependent. The two infinite order supertiles of this configuration 
can be shifted vertically (resp. horizontally) one from each 
other, while the configuration obtained stays 
an element of the subshift.
This is an obstacle to dynamical properties such as 
minimality or transitivity, 
since a pattern which crosses the separating line can not 
appear in the other configurations.
In this section, we describe additional layers 
that allow aligning all the supertiles 
having the same order and 
eliminate this phenomenon. \bigskip

Here is a description of the second layer: \bigskip

\textit{Symbols:} $nw, ne, sw, se$, 
and a blank symbol. \bigskip

The {\textit{rules}} are 
the following ones: 

\begin{itemize}
\item \textbf{Localization:} 
the symbols $nw$, $ne$ ,$sw$ and $se$
are superimposed only on three arrows and five arrows 
symbols in the Robinson layer.
\item \textbf{Induction 
of the orientation:} on a position with 
a three arrows symbol such 
that the long arrow 
originate in a corner
is superimposed a symbol 
corresponding to the orientation 
of the corner.
\item \textbf{Transmission rule:} 
on a three or five arrows symbol position, 
the symbol in this layer is 
transmitted to the position in the direction pointed 
by the long arrow when the Robinson symbol 
 is a three or five arrows symbol 
with long arrow pointing in the same direction. 
\item \textbf{Synchronization rule:} 
On the pattern 
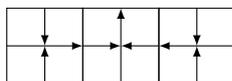
\[\begin{tikzpicture}[scale=0.5]
\robionedroite{0}{0}
\robionebas{2}{0}
\robionegauche{4}{0}
\end{tikzpicture}\]
or 
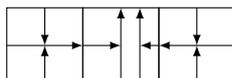
\[\begin{tikzpicture}[scale=0.5]
\robionedroite{0}{0}
\robifourhaut{2}{0}
\robionegauche{4}{0}
\end{tikzpicture}\]
in the Robinson layer, if 
the symbol on the left side is $ne$
(resp. $se$), 
then the symbol on the right side is $nw$
(resp. $sw$).
On the images by rotation of these patterns, 
we impose similar rules.
\item \textbf{Coherence rule:} 
the other couples of symbols are forbidden 
on these patterns.
\end{itemize}
 \bigskip

\textit{Global behavior:} the 
symbols 
$ne,nw,sw,se$ designate 
orientations: 
north east, north west, south 
west and south east. We will 
re-use this symbolisation 
in the following. The 
localization rule implies 
that these symbols 
are superimposed on and only on 
straight paths connecting 
the corners of adjacent order $n$ 
cells for some integer $n$.

The effect of transmission and synchronization rules 
is stated by the following lemma: 

\begin{lemma}
In any configuration $x$ of the subshift $X_{adR}$, 
any order $n$ supertile appears 
periodically in the whole configuration, 
with period $2^{n+2}$, horizontally and vertically.
\end{lemma}

\begin{proof}
\begin{itemize}
\item This property is true in an infinite supertile: 
this is the statement of Lemma~\ref{lem.repetition.supertiles}.
Hence the statement is true in a type $(i)$ configuration. 
This is also true in a type $(iii)$ configuration, 
since the infinite supertiles are aligned, 
and that the positions where the order $n$ supertiles appear 
are the same in any infinite supertile. This statement uses 
the property that 
an order $n$ supertile forces the 
presence of an order $n+1$ 
one. 
\item Consider a configuration of the subshift $X_{adR}$ 
which is of type $(ii)$. Let 
us assume that the separating line is vertical, the 
other case being similar. In order to simplify 
the exposition we assume that this column intersects 
$(0,0)$.

\begin{enumerate}

\item \textbf{Positions of the supertiles along 
the infinite line:}

From Lemma~\ref{lem.repetition.supertiles}, 
there exists a sequence of numbers $ 0 \le z_n < 2^{n+2}-1$ 
and $ 0 \le z'_n < 2^{n+2}-1$ 
such that for all $k \in \Z$,
the orientation symbol on positions 
$(-1,z_n + k.2^{n+2})$ (in 
the column on the left of the 
separating one) is $se$ and the 
orientation symbols on positions 
$(1,z'_n + k.2^{n+2})$ is $sw$. The symbol 
on positions $(-1,z_n + k.2^{n+2}+2^{n+1})$ 
is then $ne$ and is $nw$ on 
positions 
on positions $(1,z'_n + k.2^{n+2}+2^{n+1})$: 
this comes from the fact that an order $n$ petal 
has size $2^{n+1}+1$. 

Let us prove 
that for all $n$, $z_n = z'_n$. 
This means that the supertiles 
of order $n$ on the 
two sides of the separating line 
are aligned.

\item \textbf{Periodicity of these positions:}

Since for all $n$, there is a space of
$2^{n}$ columns 
between the rightmost or 
leftmost order $n$ supertile in 
a greater order supertile and the border 
of this supertile (by a recurrence argument), 
this means that the space between 
the rightmost order $n$ supertiles of the left 
infinite supertile and the leftmost order $n$ supertile 
of the right infinite supertile 
is $2^{n+1}+1$. Since two adjacent of these supertiles 
have opposite orientations, this implies 
that each supertile appears periodically 
in the horizontal 
direction (and hence both horizontal 
and vertical directions) 
with period $2^{n+2}$. \bigskip

\begin{figure}[ht]
\[\begin{tikzpicture}[scale=0.4]
\fill[gray!20] (0,0) rectangle (1,20);
\draw[dashed] (0,0) rectangle (1,20);  
\draw[-latex] (0.5,-1) -- (0.5,21);
\draw (0.25,1) -- (0.75,1);
\node at (7.5,1) {$z'_m + k' 2^{m+2}$};

\draw[-latex] (-4,1) -- (-4,17); 
\node at (-5,9) {$2^{n+1}$};

\robiredbasgauche{2}{1}
\robiredbasdroite{-3}{1}

\robiredhautgauche{2}{5}
\robiredbasgauche{2}{9}
\robiredhautgauche{2}{13}
\robiredbasgauche{2}{17}

\robiredhautdroite{-3}{17}
\end{tikzpicture} \]
\caption{\label{fig.decalage.toeplitz} Schema
of the proof. The separating line 
is colored gray.}
\end{figure}
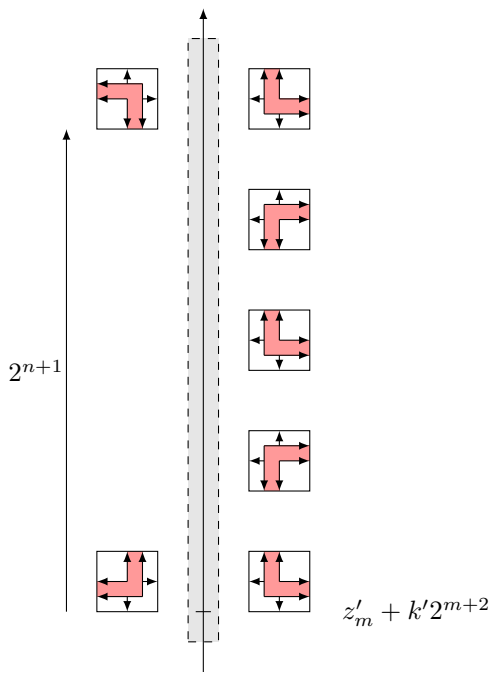

\item \textbf{The orientation symbols 
force alignment:}

Assume that there exists 
some $n$ such that $z_n \neq z'_n$.
Since $\{z_n + k.2^{n+2}, \ (n,k) \in \N \times \Z\} = \Z$,
this implies that there exist some $m \neq n$ and 
some $k,k'$ such that 
\[z_n + k.2^{n+2} = z'_m +k'.2^{m+2}.\] 
One can assume without loss of generality that 
$m<n$, exchanging $m$ and $n$ if necessary.
Then the position 
$(-1,z_n+k.2^{n+2}+2^{n+1})$ has
orientation symbol equal to $ne$.
As a consequence, the position 
$(1,z'_m +k'.2^{m+2} + 2^{n-m-1}.2^{m+2})$ 
has the same symbol. However, by definition,
this position has symbol $se$: there is 
a contradiction. This situation is illustrated on 
Figure~\ref{fig.decalage.toeplitz}. 

\end{enumerate}
\end{itemize}
\end{proof}

\subsection{Completing blocks}

Let $\chi : \N^{*} \rightarrow \N^{*}$ such 
that for all $n \ge 1$, 
\[\chi(n) = \Bigl\lceil \log_2 (n) \Bigr\rceil + 4.\]
Let us also denote $\chi'$ the function 
such that for all $n \ge 1$,
\[\chi'(n) =  \Bigl\lceil\frac{\Bigl\lceil \log_2 (n) 
\Bigr\rceil}{2} \Bigr\rceil + 2.\]

The following lemma will be extensively used 
in the following of this text, in order to prove 
dynamical properties of the constructed subshifts: 

\begin{lemma} \label{prop.complete.rob} 
For all $n \ge 1$, any $n$-block in the language of 
$X_{adR}$ is sub-pattern of some order $\chi(n)$ 
supertile, and is sub-pattern of some order $\chi'(n)$ 
order cell.
\end{lemma}

\begin{proof}

\begin{enumerate}
\item \textbf{Completing into 
an order $2^{\lceil \log_2 (n) \rceil +1}-1$ block:}

Consider some $n$-block $p$ that appears 
in some configuration $x$ 
of the SFT $X_{adR}$. We can complete it into 
a $2^{\lceil \log_2 (n) \rceil +1}-1$ block, 
since $2^{\lceil \log_2 (n) \rceil +1} - 1 
\ge 2n-1 \ge n$ for all $n\ge 1$.

\item \textbf{Intersection with four order 
$\lceil \log_2 (n) \rceil$ supertiles:}

From the periodic appearance property of the order 
$\lceil \log_2 (n) \rceil$ supertiles 
in each configuration, this last block 
intersects at most four
supertiles having this order. Let us complete $p$ 
into the block whose support is the union 
of the supports of the supertiles and the 
cross separating these. 

\item \textbf{Possible patterns after this completion
according to the center symbol:}

Since this pattern 
is determined by the symbol at the center of 
the cross and the orientations of the supertiles, 
the possibilities for this pattern 
are listed on Figure~\ref{fig.orientation.supertiles0}, 
Figure~\ref{fig.orientation.supertiles1} and 
Figure~\ref{fig.orientation.supertiles2}. Indeed, 
when the orientations of the supertiles are 
like on Figure~\ref{fig.orientation.supertiles0}, 
each of the supertiles forcing the 
presence of an order $\lfloor \log_2 (n) \rfloor+1$
supertile, the center is a red corner.
When the orientations of the supertiles are 
like on Figure~\ref{fig.orientation.supertiles1},
the center of the block can not be superimposed with 
a red corner since the two west supertiles 
force an order $\lfloor \log_2 (n) \rfloor+1$ 
supertile, as well as the two east supertiles. 
This forces a non-corner symbol on the position 
considered.

For type $4,5,9,10$ patterns, there are 
two possibilities: the values of the two arms 
of the central cross are equal or not. 
Hence the notation $4,4'$, where $4'$ 
designates the case where the two values are different.

One completes the alignment layer on $p$ 
according to the restriction of the configuration
$x$.

On these patterns, the value of symbols on the 
cross is opposed to the value of the 
symbols on the crosses of the four supertiles composing 
it.

\item \textbf{Localization of these patterns 
as part of a greater cell:}

The way to complete the obtained pattern 
is described as follows: 

\begin{enumerate}
\item When the pattern is the one 
on Figure~\ref{fig.orientation.supertiles0}, this is an order 
$\lfloor \log_2 (n) \rfloor+1$ supertile 
and the statement is proved. Indeed,  
any order $\lceil \log_2 (n) \rceil+1$ supertile
is 
a sub-pattern of any order 
$\lfloor \log_2 (n) \rfloor+4$ one.

\item One can see the patterns on
Figure~\ref{fig.orientation.supertiles1} and 
Figure~\ref{fig.orientation.supertiles2} 
in an order 
$\lfloor \log_2 (n) \rfloor+1$, 
$\lfloor \log_2 (n) \rfloor+2$, 
$\lfloor \log_2 (n) \rfloor+3$, or
$\lfloor \log_2 (n) \rfloor +4$ supertile,
depending on how was completed the initial pattern 
thus far (this correspondance is shown 
on Figure~\ref{fig.completing.supertile}), 
hence a sub-pattern of an order 
$\lfloor \log_2 (n) \rfloor+4$
supertile.
\end{enumerate}

The orientation of the greater order 
supertiles implied 
in this completion are chosen according 
to the symbols of the alignment layer. 
This layer is then completed.

\item This
implies that 
any $n$-block is the sub-pattern 
of an order $2 ( \Bigl \lceil 
\frac{1}{2} \lceil \log_2 (n) \rceil \Bigr \rceil +2)
+1$ supertile, which is included into 
an order $\Bigl \lceil 
\frac{1}{2} \lceil \log_2 (n) 
\rceil \Bigr \rceil +2$
cell.
\end{enumerate}
\end{proof}

\begin{proposition}
The subshift $X_{adR}$ 
is a minimal SFT.
\end{proposition}

\begin{remark}
Let use note that the existence of a minimal $\Z^2$-SFT is already known. 
For instance, see~\cite{Ballierthesis} (Theorem 1.35). 
\end{remark}

\begin{proof}
Any cell appears in all the configurations of this 
subshift. Since any pattern is sub-pattern of 
a cell, it then appears in all the configurations. 
Hence $X_{adR}$ is a minimal SFT.
\end{proof}

\section{ \label{sec.obstruction.min} Obstruction}

In this section we prove that the entropy dimensions of a minimal SFT are 
constrained as follows: 

\begin{proposition} 
\label{proposition.obstruction.minimality} Let $d$ be 
some positive integer. 
Let $X$ be a minimal $\Z^d$-SFT.
Then $\overline{D_h}(X) \le d-1$.
\end{proposition}

\begin{remark}
The proof of this proposition was 
communicated to us by P. Guillon.
\end{remark}

\begin{proof} \textbf{Idea:} \textit{the 
idea of the proof is to construct an 
element of the subshift having low complexity. This 
construction using only the fact that the subshift is 
of finite type. Since the subshift is minimal, 
the complexity of the subshift is equal to 
the complexity 
of this element. We deduce the upper bound 
on the entropy dimension.} \bigskip

Let $X$ be some SFT on alphabet $\A$.
Let $r>0$ be the rank of the SFT. 

\begin{enumerate}
\item \textbf{Definition of the 
annulus and cross supports:} 

Let us denote, for all $n \ge 0$, $r_n = 2^{n+2} r$, 
and define
\[\mathbb{O}_{n} = \llbracket 1,r_n\rrbracket^d \backslash 
\llbracket 1+r,r_n-r+1\rrbracket^d,\]
\[\mathbb{C}_{n} = \bigcup_{1 \le k \le d} \llbracket 1,r_n\rrbracket ^{k-1} 
\times \llbracket r_{n-1}-r+1, r_{n-1}+r\llbracket \times 
\llbracket 1,r_n\rrbracket ^{d-k},\] 
when $n \ge 1$, and $\mathbb{C}_0$ is the set $\llbracket r+1,2r\rrbracket ^d$.
Informally, we call $\mathbb{O}_n$ the outside of set 
$\llbracket 1, r_n \rrbracket ^d$, and the 
inside is the complementary 
of $\mathbb{O}_n$ in this set.
\bigskip

\begin{figure}[h!]
\[\begin{tikzpicture}[scale=0.3]
\begin{scope}
\node[scale=1.5] at (-7,5) {$\mathbb{O}_n$};
\fill[gray!20] (-3,-3) rectangle (13,13);
\fill[gray!65] (0,0) rectangle (10,10); 
\fill[gray!20] (1,1) rectangle (9,9);
\draw (0,0) rectangle (10,10);
\draw (1,1) rectangle (9,9); 
\draw[dashed] (-3,-3) grid (13,13);
\draw (4,0) -- (4,10); 
\draw (6,0) -- (6,10);
\draw (0,4) -- (10,4);
\draw (0,6) -- (10,6);
\end{scope}

\begin{scope}[xshift=23cm]
\node[scale=1.5] at (-7,5) {$\mathbb{C}_n$};
\fill[gray!20] (-3,-3) rectangle (13,13);
\fill[gray!65] (0,0) rectangle (10,10); 
\fill[gray!20] (6,6) rectangle (10,10); 
\fill[gray!20] (0,0) rectangle (4,4);
\fill[gray!20] (6,0) rectangle (10,4);
\fill[gray!20] (0,6) rectangle (4,10);
\draw (0,0) rectangle (10,10);
\draw (1,1) rectangle (9,9); 
\draw[dashed] (-3,-3) grid (13,13);
\draw (4,0) -- (4,10); 
\draw (6,0) -- (6,10);
\draw (0,4) -- (10,4);
\draw (0,6) -- (10,6);
\draw[latex-latex] (4,-2) -- (6,-2);
\draw[latex-latex] (12,0) -- (12,10);
\node at (14,5) {$r_{n}$};
\node at (5,-4) {$2r$};
\end{scope}
\end{tikzpicture}\]
\caption{\label{fig.sets.couronnes} Illustration of the definition of the sets $\mathbb{O}_n$ and
$\mathbb{C}_n$ (respectively dark gray 
set on the left and on the right).}
\end{figure}
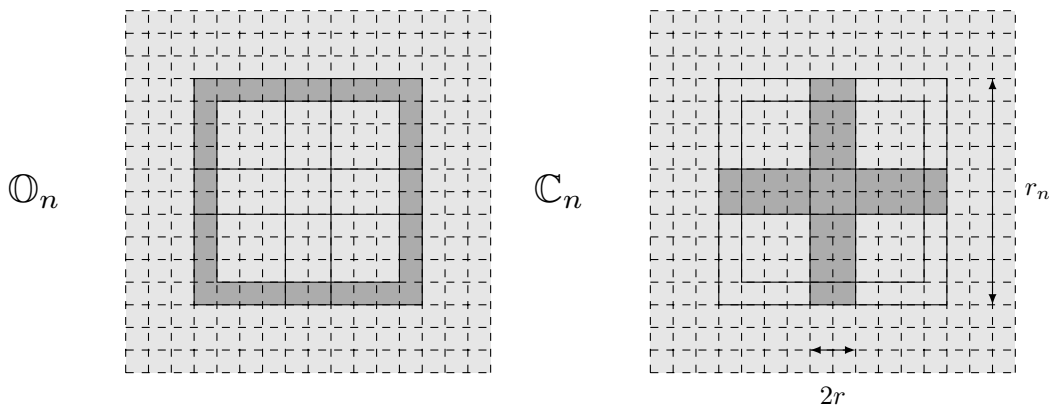

\item \textbf{An association between patterns 
on the annuli and patterns on the crosses:}

Let $\psi_n : \mathcal{L}_{\mathbb{O}_n} (X) 
\rightarrow \mathcal{L}_{\mathbb{C}_n} (X)$ be a function which to some pattern $p$ associates 
a possible completion of $p$ on 
$\mathbb{O}_n \bigcup \mathbb{C}_n$. \bigskip

Let us construct recursively 
a function 
$\varphi_n : \mathcal{L}_{\mathbb{O}_n} (X)
\rightarrow {\mathcal{L}}_{\llbracket 1, r_n
\rrbracket ^d} (X)$ as follows (See an illustration on Figure~\ref{fig.varphi}).
Consider $p$ a pattern in $\mathcal{L}_{\mathbb{O}_n} (X)$. 

\begin{enumerate}
\item Extend it with the pattern $\psi_n (p)$.
\item \begin{itemize}
\item if $n \ge 1$, consider the restrictions of the obtained pattern 
on the sets consisting in a product of $d$ elements of 
$\left\{ \llbracket 1, r_{n-1} \rrbracket , \llbracket r_{n-1} +1 , r_n\rrbracket \right\}$. 
These sets are translate of $\mathbb{O}_{n-1}$. Apply the function 
$\varphi_{n-1}$ then 
the completion on these copies of $\mathbb{O}_{n-1}$. 
\item if $n=0$, then the construction is done.
\end{itemize}
\end{enumerate} 

\begin{figure}[h!]
\[\begin{tikzpicture}[scale=0.3]
\begin{scope}
\node[scale=1.5] at (-7,5) {$1.$};
\fill[gray!20] (-3,-3) rectangle (13,13);
\fill[gray!65] (0,0) rectangle (10,10); 
\fill[gray!20] (1,1) rectangle (9,9);
\draw (0,0) rectangle (10,10);
\draw (1,1) rectangle (9,9); 
\draw[dashed] (-3,-3) grid (13,13);
\draw (4,0) -- (4,10); 
\draw (6,0) -- (6,10);
\draw (0,4) -- (10,4);
\draw (0,6) -- (10,6);
\draw[-latex] (-5,1) -- (0,1) ;
\node at (-6,1) {$p$};
\end{scope}

\begin{scope}[xshift=23cm]
\node[scale=1.5] at (-7,5) {$2.$};
\fill[gray!20] (-3,-3) rectangle (13,13);
\fill[gray!65] (0,0) rectangle (10,10); 
\fill[gray!20] (6,6) rectangle (9,9); 
\fill[gray!20] (1,1) rectangle (4,4);
\fill[gray!20] (6,1) rectangle (9,4);
\fill[gray!20] (1,6) rectangle (4,9);
\draw (0,0) rectangle (10,10);
\draw (1,1) rectangle (9,9); 
\draw[dashed] (-3,-3) grid (13,13);
\draw (4,0) -- (4,10); 
\draw (6,0) -- (6,10);
\draw (0,4) -- (10,4);
\draw (0,6) -- (10,6);
\draw[-latex] (15,3) -- (6,3);
\node at (16.5,3) {$\psi_n (p)$};
\end{scope}

\begin{scope}[xshift=11.5cm, yshift=-19cm]
\node[scale=1.5] at (-7,5) {$3.$};
\fill[gray!20] (-3,-3) rectangle (13,13);
\fill[gray!65] (0,0) rectangle (10,10); 
\fill[gray!20] (6,6) rectangle (9,9); 
\fill[gray!20] (1,1) rectangle (4,4);
\fill[gray!20] (6,1) rectangle (9,4);
\fill[gray!20] (1,6) rectangle (4,9);
\draw (0,0) rectangle (10,10);
\draw (1,1) rectangle (9,9); 
\draw[dashed] (-3,-3) grid (13,13);
\draw (4,0) -- (4,10); 
\draw (6,0) -- (6,10);
\draw (0,4) -- (10,4);
\draw (0,6) -- (10,6);

\node at (20,5) {\begin{tikzpicture}[scale=0.3]
\draw[color=red,line width=0.5mm] (0,0) rectangle (1,1);
\end{tikzpicture}: copies of $\mathbb{O}_{n-1}$};

\draw[color=red,line width=0.5mm,step=5] 
(0,0) grid (10,10);
\end{scope}
\end{tikzpicture}\]
\caption{\label{fig.varphi} Illustration of the construction of 
the functions $\varphi_n$.}
\end{figure}
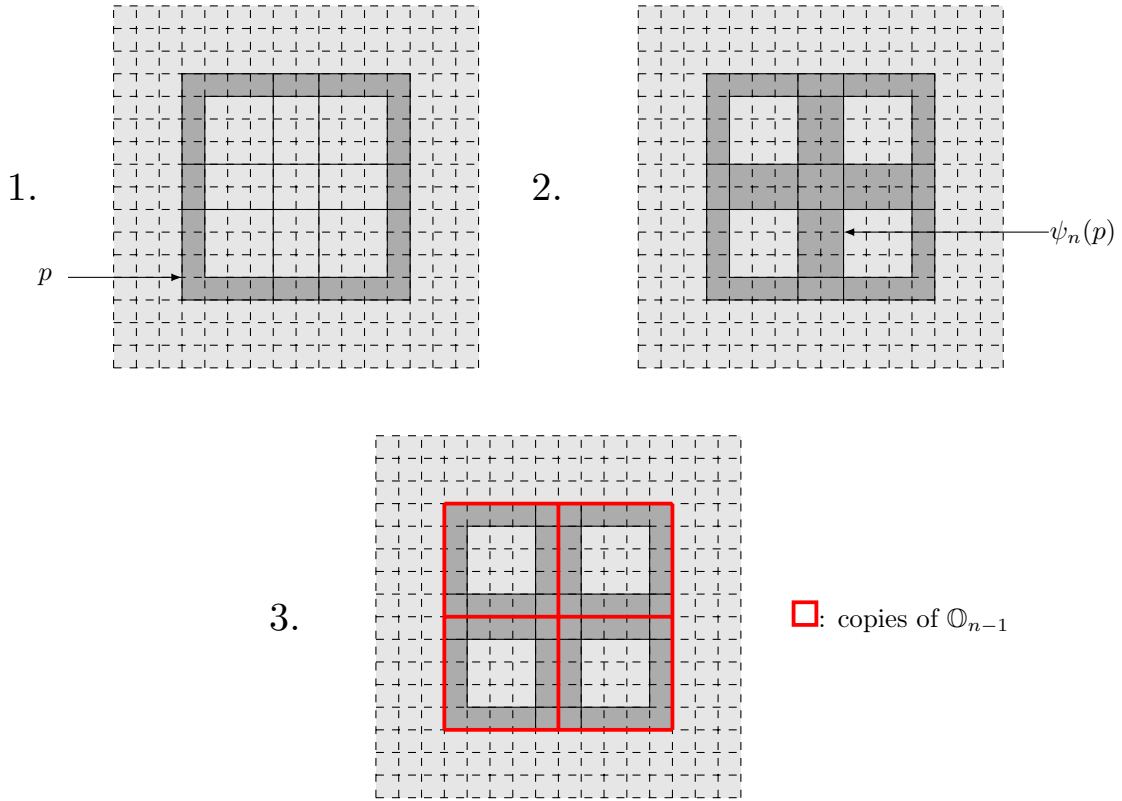

\item \textbf{Construction of the configuration:}

Consider the a sequence of patterns 
$(p_n)_n$ such that for all $n$, 
$p_n$ has support
$\mathbb{O}_n$, and $(x_n)_n$ a sequence of configurations whose restriction on 
\[\Bigl \llbracket 1-\frac{r_n}{2} , 
\frac{r_n}{2} \Bigr \rrbracket^d\]
is equal to $\varphi_n (p_n)$. 
There exists such a sequence, since 
the restriction of this pattern on 
$\mathbb{O}_n$ is globally admissible. \bigskip 

By compactness, we can extract 
some infinite sub-sequence of 
$(x_n)_n$ that converges towards 
some configuration $x$ of $X$.

\item \textbf{Complexity of the configuration $x$:}

Let us fix some $k \ge 1$ and let us 
give an upper bound
of the number of $r_k$-block that appear in $x$.
Consider $p$ some $r_k$-block that appear in $x$, 
in position $\vec{u}$. There exists some $n \ge k$ 
such that $\llbracket 1-\frac{r_n}{2} , \frac{r_n}{2} \rrbracket^d$ 
contains $\vec{u}+\mathbb{C}_k$, and the configuration $x_n$ coincides 
with $x$ on $\vec{u}+\mathbb{C}_k$. As we can decompose 
$\llbracket 1-\frac{r_n}{2} , \frac{r_n}{2} \rrbracket^d$ into copies 
of $\llbracket 1-\frac{r_k}{2} , \frac{r_k}{2} \rrbracket^d$ 
on which the patterns have their inside 
determined by their outside, the pattern $p$ is a sub-pattern 
of a concatenation of $2^{d}$ such blocks (See an illustration 
of Figure~\ref{fig.inclusion.blocks}). 

\begin{figure}[h!]
\label{fig.inclusion.blocks}
\[\begin{tikzpicture}[scale=0.3]
\fill[gray!65] (0,0) rectangle (16,16);
\fill[gray!20] (0.25,0.25) rectangle (15.75,15.75);
\fill[gray!65] (10,10) rectangle (14,14);
\fill[gray!20] (10.25,10.25) rectangle (11.75,11.75);
\fill[gray!20] (12.25,12.25) rectangle (13.75,13.75);
\fill[gray!20] (12.25,10.25) rectangle (13.75,11.75);
\fill[gray!20] (10.25,12.25) rectangle (11.75,13.75);
\draw[dashed, step=2] (0,0) grid (16,16);
\draw[-latex] (19,4) -- (16,4);
\draw[-latex] (19,11) -- (14,11); 
\draw[color=red, line width=0.5mm] (11,11) rectangle (13,13);
\draw[-latex] (11,18) -- (11,13); 
\node at (11,19) {$p$};
\node at (26,11) {Copies of $\llbracket 1-\frac{r_k}{2} , \frac{r_k}{2} \rrbracket^d$ };
\node at (21,4) {$\mathbb{C}_n$};
\end{tikzpicture}\]
\end{figure}

Hence 
the number of $r_k$ blocks in $x$ is smaller than 
the number of elements of $((\A)^{\mathbb{O}_k})^{2^d}$, which is smaller than 
\[|\A| ^{2^d |\mathbb{O}_k|} \le 
|\A|^{2^d * 2d * r (r_k)^{d-1}},\]
because $\mathbb{O}_k$ is the union of $2d$
cuboids having dimensions $r, r_k , ... , r_k$ (as illustrated on 
Figure~\ref{fig.dec.cuboids}).

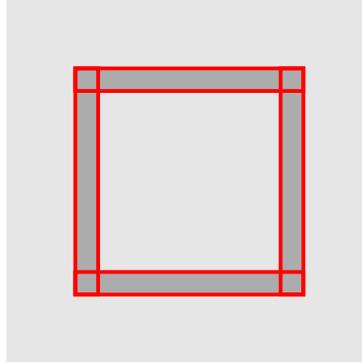
\begin{figure}[ht]
\[\begin{tikzpicture}[scale=0.3]
\fill[gray!20] (-3,-3) rectangle (13,13);
\fill[gray!65] (0,0) rectangle (10,10); 
\fill[gray!20] (1,1) rectangle (9,9);
\draw (0,0) rectangle (10,10);
\draw (1,1) rectangle (9,9);
\draw[color=red,line width=0.5mm] (0,0) rectangle (1,10); 
\draw[color=red,line width=0.5mm] (0,0) rectangle (10,1);
\draw[color=red,line width=0.5mm] (10,10) rectangle (9,0);
\draw[color=red,line width=0.5mm] (10,10) rectangle (0,9);
\end{tikzpicture}\]
\caption{\label{fig.dec.cuboids} The set $\mathbb{O}_k$ can be decomposed 
into $2d$ cuboids. On this picture, $d=2$.}
\end{figure}

Since $X$ is minimal, any globally admissible pattern 
appears in $x$. As a consequence,
\[N_{r_k} (X) \le |\A|^{2^d.2d . r . (r_k)^{d-1}},\]
which means that $\overline{D_h} (X) \le d-1$.

\end{enumerate}

\end{proof}

There is also a computational obstruction: 

\begin{proposition}
Let $X$ be a $\Z^3$-SFT having entropy dimension. 
Then $D_h (X)$ is a $\Delta_2$-computable number.
\end{proposition}

\begin{corollary}
Let $X$ be some minimal $\Z^3$-SFT admitting an entropy dimension. 
Then its entropy dimension is a $\Delta_2$-computable 
number in $[0,2]$. 
\end{corollary}

\section{\label{sec.realization.min} Realization}

In this section, we prove the following theorem:

\begin{reptheorem}{thm.dim.entropique.intro}[\cite{GS17ED}]
Every $\Delta_2$-computable number 
in $[0,2]$ is the entropy dimension 
of a minimal $\Z^3$-SFT.
\end{reptheorem}

\subsection{\label{sec.construction.abstract} Abstract of the construction}

In this section, we provide an abstract 
of the proof of Theorem~\ref{thm.dim.entropique.intro}.
We first present the principles of Meyerovitch's
construction for the proof of 
Theorem~\ref{thm.meyerovitch}, 
since we use these principles in 
our construction.

\subsubsection{Principles 
of the proof of Theorem~\ref{thm.meyerovitch}}

Let $z \in ]0,2]$ a $\Delta_2$-computable 
number ($0$ is easily realized 
as the entropy dimension of a minimal SFT). 
Let us describe 
how to construct a $\Z^2$-SFT whose 
entropy dimension is $z$.
From Lemma~\ref{lem.density}, 
there exists some $\Pi_1$-computable sequence
$(a_n)_{n} \in \{0,1\}^{\N}$
such that 
$$z = 2 \lim_n \frac{\sum_{k=0}^n a_k}{n+1}.$$
This construction relies on the propagation 
of a signal through the petal hierarchy, 
which is transformed through the intersections 
of petals. The signal has two possible colors 
$
$. In our construction, 
we call them hierarchy bits.
The transformation rule is abstracted on 
Figure~\ref{fig.abstract.meyerovitch.transformation}. It depends on a bit $f_n$ in 
$\{0,1\}$ attached to order $n$ petals.
These bits are called frequency bits 
in our construction.

\begin{figure}[ht]
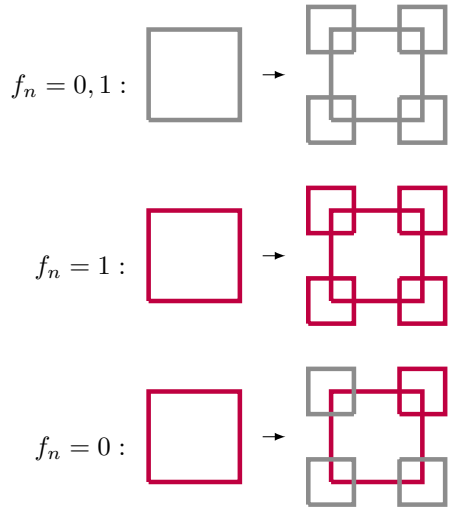

\[%
\]
\caption{\label{fig.abstract.meyerovitch.transformation}
Illustration of the signal 
transformation in Meyerovitch's construction.}
\end{figure}

On the corners of the order $0$ petals 
are superimposed random bits 
in $\{0,1\}$.
The bits $f_n$, $n \ge 0$ are imposed to 
be $a_n$
by a Turing machine using a similar construction 
as in~\cite{Hochman-Meyerovitch-2010}.
This architecture is set up 
so that it does not contribute 
to the entropy dimension, by using thin 
computation areas. Thus only 
the random bits do contribute to the 
entropy dimension.
The number of corners in an order zero 
petal which is inside an order $n$ cell whose 
border is colored $\begin{tikzpicture}[scale=0.3]
\fill[purple] (0,0) rectangle (1,1);
\draw (0,0) rectangle (1,1);
\end{tikzpicture}$
is equal to \[4^{\sum_{k=0}^n a_k}.\] 
This comes from iterating 
the transformation rule. 
This number is $0$ when the border
of the cell is colored 
$\begin{tikzpicture}[scale=0.3]
\fill[gray!90] (0,0) rectangle (1,1);
\draw (0,0) rectangle (1,1);
\end{tikzpicture}$.

As a consequence, the 
number of possible sets of random 
bits over an order $n$ cell is approximately
$$2^{4^{\sum_{k=0}^n a_k}}.$$

Moreover, 
the number 
of possibilities for the 
random bits 
over a block of the structure 
having the same size as 
an order $n$ cell is 
equal as the number of 
possibilities for the 
random bits on such a cell.

Since the entropy dimension 
is generated by random bits, 
it follows that
the entropy dimension is equal to 
$$\lim_n \frac{\log_2 (\log_2 (2^{4^{\sum_{k=0}^n a_k}}))}{\log(2^{n+1})} = 2 \lim_n \frac{\sum_{k=0}^n a_k}{n+1} =z,$$

since the size of an order $n$ supertile 
is approximately $2^{n+1}$.

\subsubsection{Obstacles to the 
minimality and their solutions}

The obstacles to the minimality property in
this construction are due to: 

\begin{enumerate}
\item the definition of the hierarchy bits, since 
there is some configuration whose 
hierarchy bits are all $\begin{tikzpicture}[scale=0.3]
\fill[gray!90] (0,0) rectangle (1,1);
\draw (0,0) rectangle (1,1);
\end{tikzpicture}$ (hence 
the hierarchy bit 
$\begin{tikzpicture}[scale=0.3]
\fill[purple] (0,0) rectangle (1,1);
\draw (0,0) rectangle (1,1);
\end{tikzpicture}$ does not appear);
\item the definition of the random 
bits, which can be all $0$ in a configuration, 
and all $1$ in another one;
\item the machines computations: 
in the space time diagram of a machine 
in an infinite computation area, 
there can appear some parts that never 
appear in a finite space-time diagram.
\end{enumerate}

In this text we propose a construction, 
abstracted in the following section, 
that overcomes as follows these difficulties 
so that the SFT is minimal: 

\begin{enumerate}
\item The trick used to overcome the difficulties 
relative to hierarchy bits is to modify 
the transformation rules defining 
these bits. This modification 
is done so that for 
a sparse set of levels, no matter the color 
of this level, at least one petal under 
in the hierarchy is colored 
$\begin{tikzpicture}[scale=0.3]
\fill[purple] (0,0) rectangle (1,1);
\draw (0,0) rectangle (1,1);
\end{tikzpicture}$. The set 
is chosen sparse 
enough so that this modification 
does not contribute to the entropy dimension.

\item We use a counter, 
called hierarchical counter, 
in order to 
alternate the possible sets 
of random bits.  
For the incrementation, 
we group the random bits into sets 
forming independent counters.
That is why we need a third dimension 
in order to realize the numbers 
in $[0,2]$, since 
the display of random bits 
is bidimensional. 
The three-dimensional subshifts
that we construct use structures that 
appear in a three-dimensional version 
of the Robinson subshift.

\item 
The difficulties coming from the computations 
are solved by simulating any finite space-time
diagram with any initial tape.
In order to have the minimality property, 
we alternate these space-time diagrams in 
any configuration using counters
called linear counters. They 
code 
for the initial tape 
of the machine. 
Moreover, a machine 
detecting an error sends an error signal to 
its initial tape. This error signal 
is taken into account if and only if 
the machine was well initialized.
\item The two types 
of counters have co-prime 
periods for different 
levels, in 
order to ensure the minimality. 
Moreover, they are incremented 
in orthogonal directions. 
\end{enumerate}

On Figure~\ref{fig.schema.global} 
is some simplified schema of the construction. 

\begin{figure}[ht]
\[\begin{tikzpicture}[scale=0.35]

\fill[purple!30] (0,0) -- (7,2.5) -- (17,2.5) 
-- (10,0) -- (0,0) ;

\fill[blue!30] (0,0) -- (0,10) -- (10,10) 
-- (10,0) -- (0,0);

\fill[yellow!30] (0,10) -- (10,10) -- 
(17,12.5) -- (7,12.5) -- (0,10);

\draw[-latex] (-3,5) -- (4,5);
\node at (-10,5) {Meyerovitch's mechanisms};
\node at (-10,3.5) {Hierarchical counter};

\draw[-latex] (5,5) -- (3,5*12/14);
\node at (4,4) {$+1$};

\draw[dashed,-latex] (12,12.5) -- (12,2.5);

\draw[-latex] (6,5) -- (8,5);
\draw[-latex] (6,5) -- (6,3);
\node at (8,4) {synchr.};

\draw[-latex] (0.5,11.25) -- (7.5,11.25);
\node at (-4.5,11.25) {Linear counter};
\draw[-latex] (8.5,11.25) -- (8.5,13.25);
\node at (10.5,11.25) {synchr.};

\draw[dashed,-latex] (13.5,11.25) -- (13.5,1.25);
\node at (17.5,6.25) {Connections};

\draw[-latex,] (8.5,11.25) -- 
(6.5,11.25 - 10/14);
\draw[-latex] (13.5,11.25) -- 
(15.5,11.25);
\node at (14.5,10) {$+1$};

\draw[-latex] (16.5,1.25) -- (10.5,1.25);

\node at (19,1.25) {Machines};
\node at (19,0) {computations};

\draw (0,0) -- (10,0) -- (10,10) -- 
(0,10) -- (0,0);
\draw (0,10) -- (7,12.5) -- (17,12.5) 
-- (10,10); 
\draw (17,12.5) -- (17,2.5) -- (10,0);
\draw[dashed] (7,12.5) -- (7,2.5) -- (0,0); 
\draw[dashed] (7,2.5) -- (17,2.5);
\end{tikzpicture}\]
\caption{
\label{fig.schema.global} 
Simplified schema of the construction
presented in this text. The cube 
represents a three-dimensional 
version of the cells observable in the 
Robinson subshift.}
\end{figure}
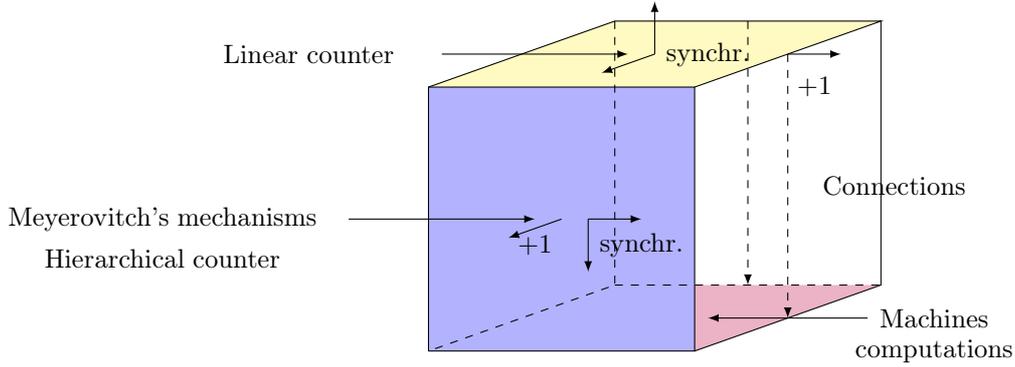

The main arguments in the proof 
of the minimality property of this subshift 
are the following ones: 

\begin{enumerate}
\item Any pattern $P$ can be completed 
into a pattern $P'$ 
over a three-dimensional cell with 
controlled size. Hence it is sufficient 
to prove 
that any pattern over a three-dimensional cell 
appears in any configuration. 
Such a pattern
is characterized by the values of 
the counters of a sequence of cells 
included in its support, intersecting 
all the intermediate levels.
\item One can find back any sequence of 
values for the counters contained 
in the three-dimensional cell starting 
from any cell. This is done 
in two steps. First by jumping multiple times
from a cell to the adjacent one having 
the same order
in the direction of incrementation 
of the linear counter. Then in the direction 
of incrementation of the hierarchical counter. 
This is possible since the periods of the counters 
are co-prime.
\end{enumerate}
This is illustrated on Figure~\ref{fig.minimality.proof.intro}. 
On this figure, $t$ (resp. 
$t'$) is the function which, taking as input
the sequence of values of the counters of 
cells of intermediate level into a cell, 
outputs the sequence of the adjacent cell 
in the incrementation 
direction of the linear (resp. hierarchical) 
counter.

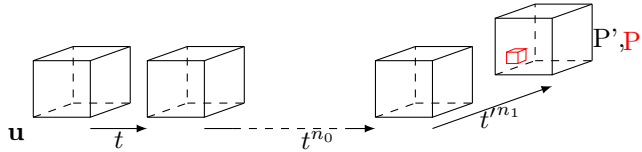
\begin{figure}[ht]
\[\begin{tikzpicture}[scale=0.15]
\begin{scope}
\draw (0,0) -- (0,5) -- (5,5) -- (5,0) -- (0,0) ;
\draw[dashed] (0,0) -- (3.5,1.25) -- (3.5,6.25) ;
\draw[dashed] (3.5,1.25) -- (8.5,1.25);
\draw (8.5,1.25) -- (8.5,6.25) ;
\draw (8.5,6.25) -- (5,5);
\draw (8.5,1.25) -- (5,0);
\draw (0,5) -- (3.5,6.25) -- (8.5,6.25);
\draw[-latex] (5,-1) -- (15.5,2.75);
\node at (11,0) {${t'}^{n_1}$}; 
\end{scope}

\begin{scope}[xshift=-20cm]
\draw (0,0) -- (0,5) -- (5,5) -- (5,0) -- (0,0) ;
\draw[dashed] (0,0) -- (3.5,1.25) -- (3.5,6.25) ;
\draw[dashed] (3.5,1.25) -- (8.5,1.25);
\draw (8.5,1.25) -- (8.5,6.25) ;
\draw (8.5,6.25) -- (5,5);
\draw (8.5,1.25) -- (5,0);
\draw (0,5) -- (3.5,6.25) -- (8.5,6.25);
\draw (5,-1) -- (7.5,-1);
\draw[dashed] (7.5,-1) -- (17.5,-1);
\draw[-latex] (17.5,-1) -- (20,-1);
\node at (15,-2) {$t^{n_0}$};
\end{scope}

\begin{scope}[xshift=-30cm]

\draw (0,0) -- (0,5) -- (5,5) -- (5,0) -- (0,0) ;
\draw[dashed] (0,0) -- (3.5,1.25) -- (3.5,6.25) ;
\draw[dashed] (3.5,1.25) -- (8.5,1.25);
\draw (8.5,1.25) -- (8.5,6.25) ;
\draw (8.5,6.25) -- (5,5);
\draw (8.5,1.25) -- (5,0);
\draw (0,5) -- (3.5,6.25) -- (8.5,6.25);
\node at (-1.5,-1.5) {$\vec{u}$};
\draw[-latex] (5,-1) -- (10,-1);
\node at (7.5,-2) {$t$};
\end{scope}

\begin{scope}[xshift=10.5cm,yshift=3.75cm]
\draw (0,0) -- (0,5) -- (5,5) -- (5,0) -- (0,0) ;
\draw[dashed] (0,0) -- (3.5,1.25) -- (3.5,6.25) ;
\draw[dashed] (3.5,1.25) -- (8.5,1.25);
\draw (8.5,1.25) -- (8.5,6.25) ;
\draw (8.5,6.25) -- (5,5);
\draw (8.5,1.25) -- (5,0);
\draw (0,5) -- (3.5,6.25) -- (8.5,6.25);
\node at (10,3) {P',};
\node[color=red] at (12,3) {P};
\draw[color=red] (1,1) rectangle (2,2);
\draw[color=red] (1,2) -- (1.7,2.3)-- (2.7,2.3) -- (2,2);
\draw[color=red] (2,1) -- (2.7,1.3) -- (2.7,2.3);
\end{scope}

\end{tikzpicture}\]
\caption{\label{fig.minimality.proof.intro} Schema of the 
proof for the minimality property of $X_z$.}
\end{figure}

\subsubsection{Description of the layers}

Let $z$ some $\Delta_2$ number in $]0,2]$ and $p=2^m-1$ be some \textit{Mersenne} number,
for some $m$ such that $m/(2^{m}-1)<z/2$. Let us construct some minimal $\Z^3$-SFT $X$
which has entropy dimension $z$. Using Lemma~\ref{lem.density},
there exists some $\Pi_1$-computable sequence 
$(a_j)_j \in \{0,1\}^{\N}$ such that 
\[z = 2 \lim_{n \rightarrow \infty} \frac{1}{n+1} \sum_{j=0}^{n} a_j.\]
We assume, without loss 
of generality, that for all $k$, $a_{2^k} = 0$. Indeed, 
this does not change the limit 
\[\lim_{n \rightarrow \infty} \frac{1}{n+1} \sum_{j=0}^{n} a_j,\]
and the computability properties 
of the sequence $(a_n)_n$.
We also assume that for all $k \in \llbracket 0,p-1\rrbracket$, 
there exist infinitely many $i \in \N$ such that 
$a_{ip+k} =0$ and infinitely many $i$ such that $a_{ip+k}=1$.

Let us construct a minimal $\Z^3$-SFT $X_z$ which has entropy dimension $1/p+ z(1-1/2p)$. 
This subshift is presented as the 
superposition 
of various layers described as follows:

\begin{itemize}
\item \textit{Structure layer [Section~\ref{sec.structure2}]:}
this layer is a three-dimensional version of the Robinson subshift. We use three superimposed copies of the Robinson subshift, each one 
of these being constant respectively in the directions $\vec{e}^1, \vec{e}^2,\vec{e}^3$.
We analyze in the corresponding section the properties 
of this layer. These properties are similar 
to the ones of the two dimensional version, 
in terms of supertiles (finite and infinite), 
repetition of the supertiles, cells, and completion of the patterns 
into supertiles (ensuring the minimality
property).
Only the order $qp$, $q \ge 0$, 
three-dimensional cells
 will support functions. The 
possible functions are the following ones:

\begin{itemize}
\item to execute some Turing computations, in order to control 
the value of the frequency bits and 
grouping bits, introduced in the next sections.
\item to increment some counter, whose value consists 
in a sequence of symbols. 
There are two type of counters:  
\begin{enumerate}
\item 
\textit{Linear counters},
coding for the initial tape of the Turing machine 
the direction of propagation of its error signal,
the states of the machine heads entering 
on the sides of the area, and the activity 
of lines and columns of the area. When 
the column or line is unactive, this 
column is not used for computation. 
This is used so that any part of a space-time
diagram can be completed into 
a space-time diagram on a face 
of a three-dimensional cell. 
\item \textit{Hierarchical counters}, 
coding for random bits 
that are superimposed to the frequency 
bits, which generate the entropy dimension.
\end{enumerate}
\item to transfer the information contained in the linear counter 
to the bottom (initial tape) and top line (for the error signal 
propagation direction) of the machine face, 
as well as the sides of the face 
(corresponding the the machine heads entering 
on the sides of the area).
\end{itemize}
Each of these functions is ensured on 
a specific face of the cell.
The separation of the machine function 
and the counters functions is a necessity for ensuring the minimality
property. This appears clearly in the proof of 
Proposition~\ref{prop.completion}, 
which states that any pattern can be completed into a valid pattern 
over a (finite) three-dimensional cell.

\item{\textit{Functional areas layer}} [Section~\ref{sec.func.areas}]
This layer serves to construct functional areas on 
the faces of the three-dimensional cells and faces connecting 
their ridges. This is done 
in order to localize the possible 
positions of the counters and machine symbols, 
and to give 
a function to these positions (step of computation, vertical or horizontal 
information 
transfer), that we call functional positions. 
This uses a 
signaling process,  
similar to a substitution, 
through hierarchical structures 
that appear 
on the faces of the 
three-dimensional cells 
and faces connecting 
these cells. 

See Figure~\ref{fig.functional.areas.intro} 
for a schema of 
the functional areas over the faces of 
a three-dimensional cell.

\begin{figure}[ht]
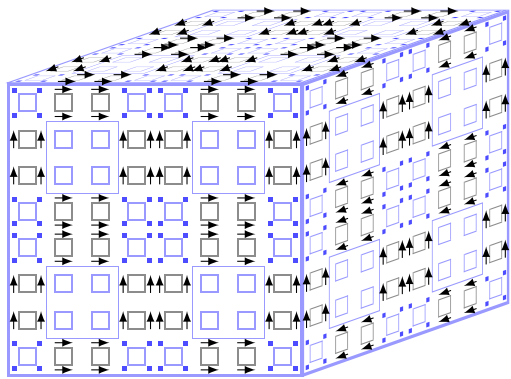

\[
\]
\caption{\label{fig.functional.areas.intro} 
Schema of the 
functional positions on the a face 
of a three-dimensional cell.}
\end{figure}

We then select 
sub-areas of these 
functional areas 
that will support 
computations of the machines
and counters. 
They are constructed 
sparse enough so that the values of the 
linear 
counter and the computations of the Turing machine 
do
not contribute to the entropy 
dimension of the subshift 
(they contribute to the 
complexity function, but 
not to the entropy dimension 
because of Proposition~\ref{proposition.sup.entropy.dimension}). Thus 
the entropy dimension is 
generated by the values of the 
hierarchical counter.
That is where 
the number $p$ is used: 
the selection process 
relies on the 
progressive 
selection of two 
dimensional cells on 
the faces according 
to a horizontal (resp. vertical) 
address in $\{0,1\}^p$ of these cells 
relatively to cells 
higher in the hierarchy. The 
principle 
of this address is similar to the coding 
of a Cantor set: $0$ means that the cell 
is on the left and $1$ on the right (resp. 
on the bottom and on the top for vertical 
addresses). As a consequence, 
this process allows the 
selection of columns (resp. lines)
of the functional areas 
of the faces, through the 
selection of order $0$ cells.

\item{\textit{Frequency bits layer}} [Section~\ref{sec.frqbits}]
The frequency bits are bits in $\{0,1\}$ 
and are 
attached to each order $qp$ 
two dimensional 
cell on their border, 
in each of the copies of the Robinson subshift.
Two cells having the same order 
have the same bit.
The machine can read these bits on its tape (we 
can make the selection process 
of functional areas such that we keep this access to the information), 
and check that the 
frequency bit of order $qp$ cells is $a_q$ for all $q \ge 0$. 

\item \textit{Grouping bits layer} 
[Section~\ref{sec.grouping.bits}]

The grouping bits are bits in $\{0,1\}$ that are
attached to order $qp$ two-dimensional cells 
on their border. The difference with  
frequency bits is their value. When $q=2^k$ 
for some $k$ the grouping bit is $1$. 
Else this bit is $0$. These bits 
are imposed by computations 
of the Turing machines. They are involved
in the mechanism 
that groups the random bits 
into hierarchical 
counter values.

\item {\textit{Linear counter layer}} [Section~\ref{sec.lin.counter}]
This layer supports the value of the linear counter, its incrementation and the 
transport of the value to the adjacent three dimensional cells 
having the same order and to the machine face.
The linear counter has values only on the three-dimensional cells 
that have order $qp$ for some $q \ge 0$ (and machines computations are present 
only on these cells).
The incrementation 
is done on the bottom line of the upper 
face of the three-dimensional cells
 
The value of the counter consists in a word on 
the alphabet 
\[\mathcal{A} \times \mathcal{Q}^3 \times \{\rightarrow,\leftarrow\} \times 
\{
\texttt{on},\texttt{off}\}^2 \times 
\mathcal{D},\] 
where $\mathcal{A}$ is 
the tape alphabet of the machine, 
$\mathcal{Q}$ is the 
state set. These two sets 
are chosen to have cardinality $2^{2^l}$ 
for some integer $l$. This is possible 
by completing the machine's alphabet 
and state set by elements that do not 
interact with the others, in 
order to not alter the 
work of the machine on well initialized tape.
The set 
$\mathcal{D}$ a finite set 
whose cardinality is chosen to be $2^{4.2^l-2}$, 
so that the counter alphabet has 
cardinality which is
\[2^{8.2^l}.\] 

The incrementation mechanism 
is the one of an adding machine acting on the
counter value, except for one step, 
when the counter value is maximal. 
When this happen, the action of the adding 
machine is suspended for one step.

As a consequence, since the number of columns 
in the selected sub-areas of the functional 
areas is a power of $2$ and that these numbers 
are different for two $qp$ order cells, 
the number of values of linear counters for 
two different of these levels 
are numbers $2^{2^{l_1}}$ and $2^{2^{l_2}}$, 
where $l_1 \neq l_2$. 
Since the counters are suspended for one 
step, the periods of the two counters 
are $2^{2^{l_1}}+1$ and $2^{2^{l_2}}+1$.
These are two different, and thus 
co-prime, Fermat numbers.

The symbols in 
the set $\{\rightarrow,\leftarrow\}$ codes for 
the direction of an error signal propagation, 
The symbols in $\{
\texttt{on},\texttt{off}\}$ 
tell which ones of the 
lines and columns are used 
in the area for computations. 
The error signal is triggered 
when the machine ends its 
computation in an error state. 
This signal is sent to of the ends 
of the initial tape (according 
to the propagation direction), 
in order to verify that the tape 
was well initialized. 
We forbid the coexistence of the error signal 
with a signal that certifies that the tape was well initialized.

\item {\textit{Machine layer}} 
[Section~\ref{sec.machine.dim.entropique}]
This layer supports the computations of the machines. These machines  
check that the $q$th frequency bit, shared 
by $qp$ cells, is equal to $a_q$. 
They also check the values of 
the grouping bits.
The initial tape of the machine corresponds 
to the projection 
of the linear counter value on 
$\mathcal{A} \times \mathcal{Q}$, 
where $\mathcal{A}$ is the tape alphabet 
and $\mathcal{Q}$ is 
the state alphabet.

The machines and the linear counters are implemented 
in opposite faces of the three-dimensional cells, 
in order to ensure the separation of the 
information. This principle allows the 
minimality property. The information 
of the linear counter is connected to 
the machine through signals that propagate
in the other faces of the cells.

\item \textit{Hierarchical counter layer [Section~\ref{sec.hier.counter}]:} 
In this layer some bits in $\{\begin{tikzpicture}[scale=0.25]
\fill[purple] (0,0) rectangle (1,1);
\draw (0,0) rectangle (1,1);
\end{tikzpicture}, \begin{tikzpicture}[scale=0.25]
\fill[gray!90] (0,0) rectangle (1,1);
\draw (0,0) rectangle (1,1);
\end{tikzpicture}\}$, called hierarchy bits, 
are superimposed to the two-dimensional cells in the copy of the Robinson subshift 
parallel to the hierarchical 
counter face of the three-dimensional cells.
These bits are determined by a signaling 
process through the hierarchical structures 
of any of the copies of the subshifts 
$X_{adR}$ in the structure layer.
This process relies on the frequency bits. On 
the border of the order $2^k p$ cells, 
the three hierarchy bits are imposed 
to be equal. 

On the blue corners having hierarchy 
bit equal to \begin{tikzpicture}[scale=0.25]
\fill[purple] (0,0) rectangle (1,1);
\draw (0,0) rectangle (1,1);
\end{tikzpicture} are superimposed 
some random bits in $\{0,1\}$.
These bits generate the entropy dimension. 
As a consequence, the machines 
have control over the entropy dimension through 
frequency bits.

For all $k$, the $k$th hierarchical counter value 
is the set of random bits 
on positions with blue corners 
that are in an order $2^k p$ order two-dimensional cell
and not in an order $2^j p$ cell, with $j<k$.

The value of the $k$th hierarchical counter 
is incremented on 
the hierarchical counter face of the 
three dimensional order $2^k p$ cells, 
The incrementation mechanism is 
similar as the one of linear counters, 
except that it uses 
discrete curves to represent the value 
of the counter as a finite sequence. 

This signaling process is done 
in such a way that the number of hierarchy bits 
on a face of an order $2^k p$ cell, $k \ge 0$, 
is strictly growing according to $k$.
Since this number is also a power of $2$, 
the periods of the hierarchical counters 
of two different levels are different Fermat numbers.

The direction of incrementation 
is chosen orthogonal 
to the incrementation direction of the 
linear counter. 
As a consequence, even if a linear 
counter has the same period
as a hierarchical counter, this 
has no influence on the minimality.
On the faces of the other 
three-dimensional faces, the values 
of the counters is not changed.
 
\item \textit{Synchronization layer
[Section~\ref{sec.synchr.dim.entropique}]:} this layer 
is used to synchronize the hierarchical 
counters of three-dimensional 
cells having the same order which are adjacent 
in the directions that are orthogonal to 
their incrementation direction. The linear 
counter is coded to have directly
this synchronization.
\end{itemize}

After that the $X_z$ is constructed 
for all $z$, the proof is as follows: 
take $x$ some $\Delta_2$-computable in $[0,2]$. 
If $x=0$, then any subshift having a unique symbol is minimal and has entropy dimension equal 
to $x$. When $x>0$, for all $m$ such that 
$1/(2^m-1) < x$, there exists some $z_m$ such that $1/p + z_m (1-1/2p) = x$.
We take $m$ such that 
\[\frac{x-1/p}{1-1/2p} > \frac{m}{2^m-1}.\]
For this $m$, $z_m$ is $\Delta_2$-computable, 
$X_{z_m}$ is minimal and 
has entropy dimension equal to $x$.

Let us make explicit the local rules that induce these global behaviors. 

\section{Details of the construction
of the subshifts \texorpdfstring{$X_{z}$}{Xz}:}

\subsection{\label{sec.structure2} Structure layer}

In this section, we describe the construction of 
the structure layer.

For this purpose, we construct a three dimensional 
equivalent of the subshift $X_{adR}$, and analyze this 
subshift from the point of view of supertile hierarchical structure, 
repetition of these supertiles, and infinite supertiles.

This subshift consists in the superposition of three copies of $X_{adR}$, 
respectively parallel to the vectors $(\vec{e}^1$ and $\vec{e}^2)$, $(\vec{e}^2$ and $\vec{e}^3)$
and $(\vec{e}^3$ and $\vec{e}^1)$.

\subsubsection{A minimal three-dimensional version of the Robinson subshift}

This subshift has {\textit{alphabet}} $\A_{adR} ^3$. 

\begin{itemize}
\item \textbf{Robinson rules in the two dimensional 
sections of $\Z^3$:} for $\vec{i} , \vec{j} \in \Z^3$ such that 
$\vec{j}-\vec{i} = \vec{e}^2$, or $\vec{e}^3$ 
(resp. $\vec{e}^1$ or $\vec{e}^3$, resp. $\vec{e}^1$ or $\vec{e}^2$), the first (resp. 
second, resp. third) coordinates of
the triples over these positions verify the rules of the subshift $X_{adR}$.
The orientation of the two-dimensional sections 
of $\Z^3$ where the rules of the subshift 
$X_{adR}$ are verified is given as follows (the 
orientation of the Robinson symbols 
depend on this orientation):
\begin{itemize}
\item for the first coordinate, the horizontal direction
 is $\vec{e}^2$ and 
the vertical one $\vec{e}^3$. This means that 
when looking in the direction opposite to the vector $\vec{e}^3$, 
and orienting $\vec{e}^1$ to the right and $\vec{e}^2$ upwards, we see the 
usual picture of a configuration of $X_{adR}$.
\item for the second one, the horizontal direction is  
$\vec{e}^3$ and 
the vertical one $\vec{e}^1$. 
\item for the last one, the horizontal direction is 
$\vec{e}^1$ and 
the vertical one $\vec{e}^2$. 
\end{itemize}
\item \textbf{Invariance in the orthogonal direction:} 
For $\vec{i}, \vec{j} \in \Z^3$ such that $\vec{j}-\vec{i} = \vec{e}^1$ (resp.  
$\vec{e}^2$, resp. $\vec{e}^3$) 
second (resp. first) coordinates of the couples over these positions are equal.
See Figure~\ref{fig.robinson.copies} for an illustration.
\item \textbf{Coincidence rules:} 
When on some position there are at least two corners symbols of the Robinson subshift, 
then the three symbols are corners having the same color (blue or red).
\item Moreover, the possible triples of blue 
corners are the following ones: 
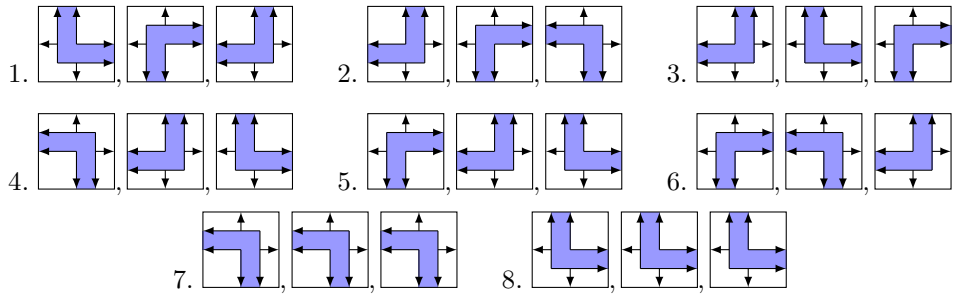
\[1. \ \begin{tikzpicture}[scale=0.5] \robibluebastgauche{0}{0} ; \end{tikzpicture}, 
\begin{tikzpicture}[scale=0.5] \robibluehauttgauche{0}{0} ; \end{tikzpicture}, 
\begin{tikzpicture}[scale=0.5] \robibluebastdroite{0}{0} ; \end{tikzpicture}
 \ \ \ \ \ 2. \ \begin{tikzpicture}[scale=0.5] \robibluebastdroite{0}{0} ; \end{tikzpicture}, 
\begin{tikzpicture}[scale=0.5] \robibluehauttgauche{0}{0} ; \end{tikzpicture}, 
\begin{tikzpicture}[scale=0.5] \robibluehauttdroite{0}{0} ; \end{tikzpicture}
\ \ \ \ \ 3. \ \begin{tikzpicture}[scale=0.5] \robibluebastdroite{0}{0} ; \end{tikzpicture}, 
\begin{tikzpicture}[scale=0.5] \robibluebastgauche{0}{0} ; \end{tikzpicture}, 
\begin{tikzpicture}[scale=0.5] \robibluehauttgauche{0}{0} ; \end{tikzpicture}\]
\[4. \ \begin{tikzpicture}[scale=0.5] \robibluehauttdroite{0}{0} ; \end{tikzpicture}, 
\begin{tikzpicture}[scale=0.5] \robibluebastdroite{0}{0} ; \end{tikzpicture}, 
\begin{tikzpicture}[scale=0.5] \robibluebastgauche{0}{0} ; \end{tikzpicture}
\ \ \ \ \ 5. \ \begin{tikzpicture}[scale=0.5] \robibluehauttgauche{0}{0} ; \end{tikzpicture}, 
\begin{tikzpicture}[scale=0.5] \robibluebastdroite{0}{0} ; \end{tikzpicture}, 
\begin{tikzpicture}[scale=0.5] \robibluebastgauche{0}{0} ; \end{tikzpicture}
\ \ \ \ \ 6. \ \begin{tikzpicture}[scale=0.5] \robibluehauttgauche{0}{0} ; \end{tikzpicture}, 
\begin{tikzpicture}[scale=0.5] \robibluehauttdroite{0}{0} ; \end{tikzpicture}, 
\begin{tikzpicture}[scale=0.5] \robibluebastdroite{0}{0} ; \end{tikzpicture}\]
\[7. \ \begin{tikzpicture}[scale=0.5] \robibluehauttdroite{0}{0} ; \end{tikzpicture}, 
\begin{tikzpicture}[scale=0.5] \robibluehauttdroite{0}{0} ; \end{tikzpicture}, 
\begin{tikzpicture}[scale=0.5] \robibluehauttdroite{0}{0} ; \end{tikzpicture}
\ \ \ \ \ 8. \ \begin{tikzpicture}[scale=0.5] \robibluebastgauche{0}{0} ; \end{tikzpicture}, 
\begin{tikzpicture}[scale=0.5] \robibluebastgauche{0}{0} ; \end{tikzpicture}, 
\begin{tikzpicture}[scale=0.5] \robibluebastgauche{0}{0} ; \end{tikzpicture}
\]
These triples correspond to the corners of a cube as on Figure~\ref{fig.cube.corners}.
This cube corresponds to the support of 
apparition of a pattern whose restriction on each of the coordinates 
on the corresponding face is a two dimensional order $n$ cell. We call these cubes 
order $n$ \textbf{three dimensional cells}.
Notice that the restriction 
imposed by these rules is 
not trivial, since the number of allowed triples of blue 
corners is $8$ which is less than the 
total number of possibilities, 
which is $4^3$. There is an equivalent restriction on triples of red corners.
\item When on a position there is only one corner, then the couple of other symbols 
is amongst the following types: 
\begin{itemize}
\item (1) Two six arrows symbols or two five arrows symbols, pointing in 
the same direction, and orthogonal (in $\Z^3$) to the corner: this type of triples corresponds 
to the center of the faces of the cubes.
\item (2) Two four arrows symbols or two tree arrows symbols, 
pointing in the same direction and orthogonal (in $\Z^3$) to the corner.
This type corresponds to the edges 
of the cubes and to the edges of connecting 
their corners. 
\item (3) Two six arrows symbols or two five arrows symbols, pointing 
in the opposite directions of the arms of the corner. This type corresponds 
to the centers of the ridges. 
\item (4) Any couple of four or three arrows symbols that 
are orthogonal (in $\Z^3$), and parallel to the corner. This type corresponds 
to the \textbf{internal faces} of the cubes.
\end{itemize}

The intersection 
of the cell with the $\Z^2$-section of $\Z^3$ that cut a three dimensional cell in two equal parts is called 
an \textbf{internal face} of the cell.

\item The triples with only arrows symbols which are two by two orthogonal (in $\Z^3$) 
are forbidden (the other triples 
of arrows symbols correspond 
to the ridges of the internal faces, or 
the inside of these faces).
\item On any translate of 
the subset $\mathbb{U}^3_{2}$ of $\Z^3$, 
there is an admissible blue triple.
\end{itemize} \bigskip

\begin{figure}[ht]
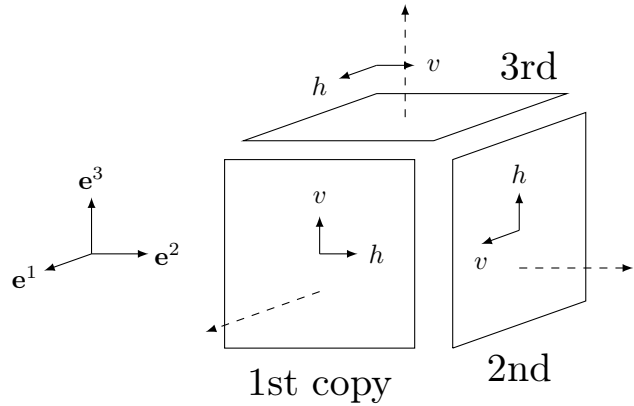

\[
\]
\caption{\label{fig.robinson.copies} Illustration of the first three rules 
of the three dimensional Robinson subshift. The symbols $h$ and $v$ 
indicate respectively 
the horizontal and vertical directions 
in each of the copies of the subshift
$X_{adR}$.}
\end{figure}

\begin{figure}[ht]
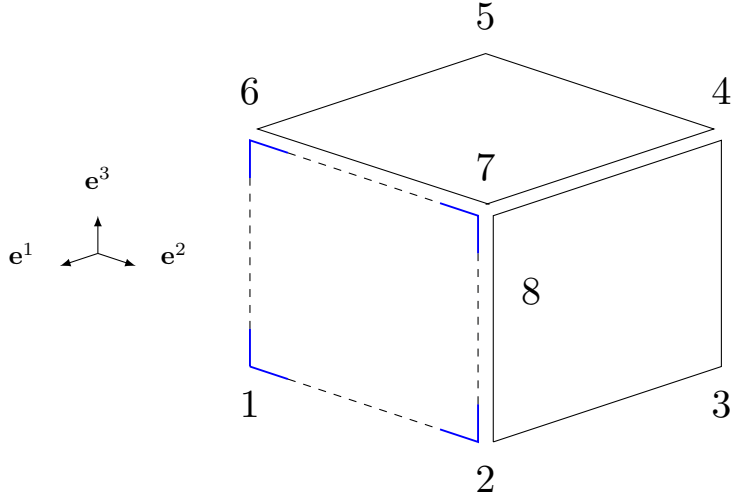

\[%
\]
\caption{\label{fig.cube.corners} Illustration of the coincidence rules of 
the structure layer: the corners of the cube 
represent the possibilities of blue corner triples.}
\end{figure}

\subsubsection{\label{sec.hierarchical.struct} Hierarchical structures}

\begin{enumerate} \item 
\textbf{Finite supertiles:}

In this paragraph, the symbols $sw,se,nw,ne$ mean respectively south west, south east, north west 
and north east orientations of 
the corner symbols in 
the alphabet of the Robinson subshift. 

\begin{enumerate}
\item \textbf{Definition by projection 
on the faces:}

Any block on the language of this 
layer whose projection over a plane parallel to $\vec{e}^2$ and $\vec{e}^3$ (resp $\vec{e}^1$ and $\vec{e}^3$, 
resp. $\vec{e}^1$ and $\vec{e}^2$), 
considering only the first (resp. second, resp. third) coordinate of the triple, 
is an order $n$ two-dimensional 
supertile with orientation $t_1$ (resp. $t_2$, resp. $t_3$) is called a \define{three-dimensional supertile} 
with orientation $t \in \{sw,se,nw,ne\}^3$
and order $n$. 

\item \textbf{Recursive definition:}

\begin{figure}[ht]
{\renewcommand{\arraystretch}{1.2}
	\begin{center}
	\begin{tabular}{|c||c|}
    \hline Position 
		of the translate & Orientation of the supertile  \\
    \hline
\hline 

 (0,0,0) &
$(ne,ne,ne)$ \\ 
\hline
$(2^{n+1},0,0)$ & 
$(ne,se,nw)$ \\
\hline
$(0,2^{n+1},0)$ &
$(nw,ne,sw)$ \\
\hline 
$(2^{n+1},2^{n+1},0)$ & 
$(nw,se,sw)$ \\
\hline 
$(0,0,2^{n+1})$ & 
$(se,nw,ne)$ \\
\hline 
$(2^{n+1},0,2^{n+1})$ &  
$(ne,sw,nw)$\\
\hline 
$(0,2^{n+1},2^{n+1})$ & 
$(sw,nw,se)$\\ 
\hline
$(2^{n+1},2^{n+1},2^{n+1})$ & 
$(sw,sw,sw)$\\
\hline 

  \end{tabular} 
	\end{center}
  }
  \caption{\label{table.restrictions.supertiles}
	Correspondence table for recursive 
	definition of three-dimensional supertiles. 
	The table gives the orientation 
	of order the $n$ cubic supertile 
	superimposed on the $\vec{v}+
	\mathbb{U}^3_{2^{n+2}-1}$, 
	where $\vec{v}$ are the entries of the table.}
\end{figure}

The order $n+1$ three-dimensional supertile with 
orientation $t \in \{sw,se,nw,ne\}^3$ can 
be constructed from  
the order $n$ cubic supertiles as follows. 
The support of the order $n+1$ cubic supertile is $\mathbb{U}^3_{2^{n+2}-1}$. 
Figure~\ref{table.restrictions.supertiles} 
shows positions in the support. On each 
of the translate of $\mathbb{U}^3_{2^{n+1}-1}$
corresponding to these positions, we put the 
an order $n$ cubic supertile whose 
orientation is given by the table.

In order to complete the construction we put on
position $(2^{n+2},2^{n+2},2^{n+2})$ 
a triple of red corners with orientation $t$. 
The three planes separating the order 
$n$ supertiles are filled with 
three arrows symbols or four arrows 
symbols induced by the corners. 
On the three lines of these 
planes intersecting the center 
position $(2^{n+2},2^{n+2},2^{n+2})$
there are triples with a unique corner.
Because the corners of the center triple have 
compatible orientations, these triple 
are admissible. The reason is 
that the arrows 
symbols are orthogonal to the corner. 

One can see that in particular, the 
faces of the cubic 
supertiles constructed this 
way verify the recurrence relation used 
to construct the 
supertiles of the Robinson subshift.

\item \textbf{Admissibility:}

One can check that the order $1$ cubic supertiles are locally admissible patterns.
As a consequence, by a 
recurrence argument, \textbf{the three-dimensional
supertiles are locally admissible patterns}
of this subshift. This means that the supertiles 
do not admit any forbidden pattern 
as a sub-pattern). \end{enumerate}

\item \textbf{Three-dimensional cells:}

For all $n \ge 0$, we call order $n$ 
\textbf{three-dimensional cell} 
any subset of $\Z^3$ which is 
the support of a pattern 
whose projection over one of the coordinates on the corresponding face (the face 
which is parallel to the copy of the Robinson subshift) is an order $n$ two dimensional cell.
The order $n$ \textbf{three-dimensional cells} 
have size $4^n+1$: this property comes 
directly from the properties of the two dimensional cells.

\item \textbf{Infinite supertiles:} 

\begin{enumerate}
\item \textbf{Definition:}

Any order $n$ three-dimensional supertile forces the presence of an order $n+1$ three-dimensional supertile 
in the direction of its orientation. This 
comes from the properties of the rigid 
version of the Robinson subshift $X_{adR}$ 
listed in 
Section~\ref{sec.hierarchical.structures}. For 
any configuration $x$ of this structure layer, 
we denote $\sim_x$ the equivalence 
relation on $\Z^3$
defined by $\vec{i} \sim_x \vec{j}$ if there is a 
supertile in $x$ which contains 
$\vec{i}$ and $\vec{j}$.
This is indeed an equivalence relation, because 
two supertiles can not intersect 
but when one of 
these two supertiles 
is a sub-pattern of the other. 
This comes directly from the same fact 
verified by the Robinson subshift, 
by projection on the faces of the three-dimensional 
blocks.
An \define{infinite order} three-dimensional
supertile is an infinite pattern over an equivalence class of this relation. 

\item \textbf{Types of configurations 
according to the number of infinite supertiles:}

Each configuration is amongst the 
following types: 
\begin{itemize}
\item[(i)] A unique infinite order supertile which covers $\Z^3$.
\item[(ii)] Two infinite order supertiles separated by a plane.
In this case, on the plane, two of 
the coordinates of the triple 
are constant and equal to orthogonal three 
or four arrows symbols (this means 
that the long arrows of 
these symbols are orthogonal), 
whose arrows directions 
are in the plane. 
The last coordinates form a configuration 
of the Robinson subshift.
This is a degenerated non-centered face or 
internal face of a three-dimensional cell. 
See 
Figure~\ref{fig.internal.face} for an illustration, where the internal 
faces are colored gray, 
and the localization of non centered part of one 
of these faces specified by a red square.

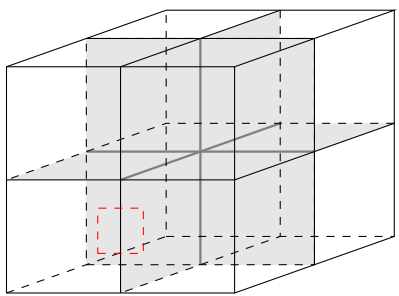
\begin{figure}[ht]
\[\begin{tikzpicture}[scale=0.3] 
\fill[gray!20] (3.5,1.25) -- (13.5,1.25) -- (13.5,11.25) -- (3.5,11.25) -- (3.5,1.25);
\fill[gray!20] (0,5) -- (10,5) -- (17,7.5) -- (7,7.5) -- (0,5);
\fill[gray!20] (5,0) -- (5,10) -- (12,12.5) -- (12,2.5) -- (5,0);
\draw[line width =0.3mm,color=gray] (3.5,6.25) -- (13.5,6.25);
\draw[line width =0.3mm,color=gray] (5,5) -- (12,7.5);
\draw[line width =0.3mm,color=gray] (8.5,1.25) -- (8.5,11.25);
\draw (0,0) -- (10,0) -- (10,10) -- (0,10) -- (0,0);
\draw (0,10) -- (7,12.5) -- (17,12.5) -- (10,10); 
\draw (10,0) -- (17,2.5) -- (17,12.5);
\draw[dashed] (0,0) -- (7,2.5) -- (7,12.5);
\draw[dashed] (7,2.5) -- (17,2.5);
\draw[dashed] (3.5,1.25) -- (13.5,1.25) -- (13.5,11.25) -- (3.5,11.25) -- (3.5,1.25);
\draw (3.5,11.25) -- (13.5,11.25) -- (13.5,1.25);
\draw (0,5) -- (10,5) -- (17,7.5);
\draw (5,0) -- (5,10) -- (12,12.5);
\draw[dashed] (12,12.5) -- (12,2.5) -- (5,0);
\draw[dashed] (17,7.5) -- (7,7.5) -- (0,5);
\draw[dashed,color=red] (4,1.75) rectangle (6,3.75);
\end{tikzpicture}\]
\caption{\label{fig.internal.face} Illustration of the internal faces.}
\end{figure}

\item[(iii)] Four infinite order supertiles separated by two orthogonal planes. 
On the intersection of the two planes, one coordinate is 
constant and consists in some (red) corner. The other coordinates can be as follows: 
\begin{itemize} 
\item on the positions of the intersection, there are 
two arrows symbols 
with either three or four arrows. 
This corresponds to a degenerated face 
superimposed with a configuration of the Robinson with 
four infinite supertiles.
\item on some position of the intersection, there are two 
six arrows or two five arrows or two three arrows 
symbols pointing in the opposite direction of 
the arms of the corner (and the symbols of the other
positions are determined). 
This corresponds to a degenerated 
centered edge of a three-dimensional cell, 
illustrated by point $1$ on 
Figure~\ref{fig.infinite.supertiles2}
\item 
all the positions of the intersection have two four arrows symbols or two three arrows 
symbols pointing to the same 
direction, orthogonal to the corner. 
This case corresponds to a degenerated 
non-centered edge, illustrated by point 
$2$ on Figure~\ref{fig.infinite.supertiles2}.
\begin{figure}[ht]
\[\begin{tikzpicture}[scale=0.3] 
\draw (0,0) -- (10,0) -- (10,10) -- (0,10) -- (0,0);
\draw (0,10) -- (7,12.5) -- (17,12.5) -- (10,10); 
\draw (10,0) -- (17,2.5) -- (17,12.5);
\draw[dashed] (0,0) -- (7,2.5) -- (7,12.5);
\draw[dashed] (7,2.5) -- (17,2.5);
\draw[dashed] (3.5,1.25) -- (13.5,1.25) -- (13.5,11.25) -- (3.5,11.25) -- (3.5,1.25);
\draw (3.5,11.25) -- (13.5,11.25) -- (13.5,1.25);
\draw (0,5) -- (10,5) -- (17,7.5);
\draw (5,0) -- (5,10) -- (12,12.5);
\draw[dashed] (12,12.5) -- (12,2.5) -- (5,0);
\draw[dashed] (17,7.5) -- (7,7.5) -- (0,5);
\draw[dashed,color=red,line width=0.3mm] (5,4) rectangle (5,6);
\draw[dashed,color=red,line width=0.3mm] (4,5) rectangle (6,5);
\node at (6,6) {$4$};
\draw[dashed,color=red,line width=0.3mm] (0,0) rectangle (0,1);
\draw[dashed,color=red,line width=0.3mm] (0,0) rectangle (1,0);
\node at (-1,-1) {$3$};
\draw[color=red,line width=0.3mm] (0,4) rectangle (0,6);
\draw[color=red,line width=0.3mm] (0,5) rectangle (1,5);
\node at (-1,5) {$1$};
\draw[color=red,line width=0.3mm] (6.5,0) -- (8.5,0);
\node at (7.5,-1) {$2$};
\end{tikzpicture}\]
\caption{\label{fig.infinite.supertiles2} Illustration of 
some locations corresponding to 
degenerated supertiles.}\end{figure}
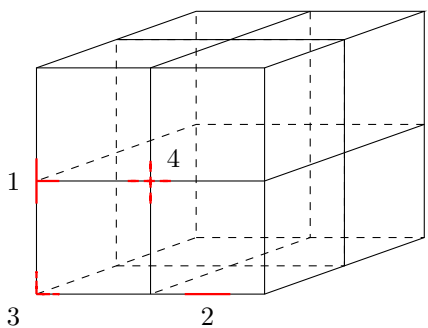
\end{itemize}
\item[(iv)] Eight infinite order supertiles, separated by three orthogonal planes. 
The intersection of the three planes contains some corner symbol.
Then there are two cases: 
\begin{itemize}
\item on this intersection there is a triple of red corners with compatible orientations. 
This corresponds to a degenerated corner of a cube, illustrated by point 3 on 
Figure~\ref{fig.infinite.supertiles2}.
\item there is only one corner, and there are two six arrows symbols or two five arrows 
symbols pointing in the opposite direction of the arms of the corner. 
This case corresponds a degenerated centered face of a cube, illustrated by point 
4 on Figure~\ref{fig.infinite.supertiles2}.
\end{itemize}
\end{itemize} 

\item \textbf{Proof of the 
exhaustiveness of this classification:}

Let us prove that there is no other possibility.

\begin{enumerate}
\item \textbf{Evaluation of the 
space separating infinite supertiles:}

First, the set of 
positions that are outside 
any infinite supertile does not contain 
any translate of $\mathbb{U}^3_2$, because 
it would imply that it contains some 
triple of blue corners, and there would 
be an infinite supertile which does not 
intersects the others (impossible) or 
intersects non-trivially 
another infinite supertile 
(impossible, because they are equivalence classes). 

\item \textbf{Possible supports and 
combination of them:}

As for the Robinson subshift, 
the supports of infinite supertiles are some 
translates of $(\epsilon_1 \N) \times (\epsilon_2 \N)
\times (\epsilon_3 \N)$ with $\epsilon_i \in \{-1,1\}$ (1/8 of $\Z^3$), 
$(\Z) \times (\epsilon_1 \N)
\times (\epsilon_2 \N)$, $(\epsilon_1 \N) \times (\epsilon_2 \N)
\times (\Z)$, $(\epsilon_1 \N) \times (\Z)
\times (\epsilon_2 \N)$ with $\epsilon_i \in \{-1,1\}$ (1/4 of $\Z^3$), 
some $(\Z) \times (\Z)
\times (\epsilon \N)$, $(\Z) \times (\epsilon \N)
\times (\Z)$, $(\epsilon \N) \times (\Z)
\times (\Z)$ with $\epsilon \in \{-1,1\}$ (half $\Z^3$) or $\Z^3$.
Because the set of positions that are not in any infinite supertile does not contain 
any translate of $\mathbb{U}^3_2$, the possibilities correspond 
to the type $(i)$ to $(iv)$ types listed above. 

\item \textbf{Reduction of eight infinite 
supertiles configurations:}

In the case where there are eight supertiles separated by three planes (type $(iv)$), 
if there are three corners, then they have compatible orientations. 
If not, then there is one corner. Indeed, if there were no corner, the intersection 
of the three planes would be composed by three orthogonal arrows symbols. This 
is impossible (see the coincidence rules).
Then the triple has to be of type $(1)$ of the third coincidence rule.
Hence, if it was of type $(2)$ for instance, it would mean that, projecting 
on the copies of the Robinson that do not correspond to corner, the configurations 
have two infinite supertiles separated by an infinite line, and the plane 
generated by translating this line is included in the two orthogonal 
planes generated by translation of the separating cross of the first copy of 
the Robinson subshift. That there are only four infinite supertiles. 

\item \textbf{Four infinite supertiles:}

In the case of four infinite supertiles, if one position 
of the intersection has a corner, then 
this corner is present on 
the whole line, and the triples 
on this line are of type $(2),(3),(4)$ 
types of the coincidence rule, and to 
the type $(iii)$ of the above description. 

\item \textbf{Two infinite supertiles:}

When there are two infinite supertiles, 
by projecting on the copies of the Robinson, we get that two of the copies 
have two infinite supertiles with a separating line.
Moreover, the plane supports a configuration of the Robinson which has a unique infinite supertile, 
for if it was not the case, there would be two separating planes. 
\end{enumerate}
\end{enumerate}
\end{enumerate}

\subsubsection{Properties of this layer}

\begin{figure}[ht]
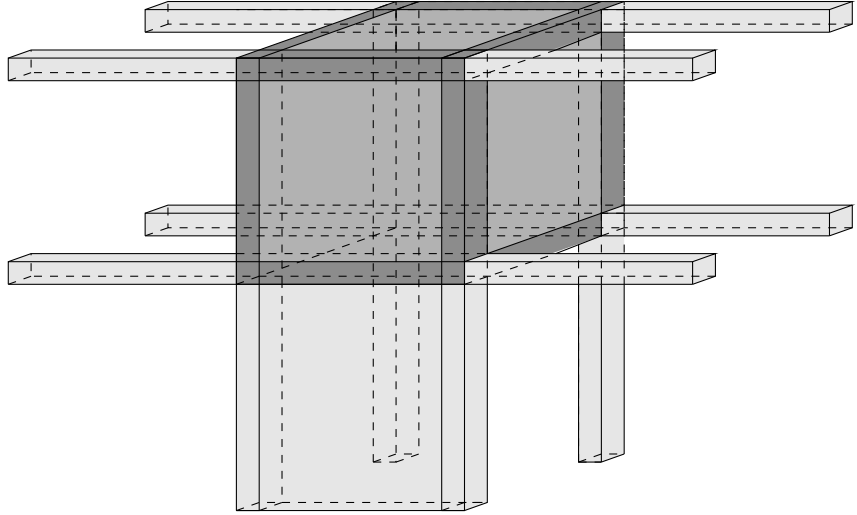

\[
\]
\caption{\label{fig.coloration} 
Partial 
representation of the surroundings of a 
three-dimensional cell - omitting 
the other ones.}
\end{figure}

The structure layer has the following properties: 

\begin{enumerate}
\item \textbf{Non-emptiness:}

A configuration: 

\begin{itemize}
\item 
whose projections 
on the copies of $X_{adR}$ are all 
of type (iii) and centered on a red corner
\item  
and such that the 
triple of these centers has compatible orientations
\end{itemize}
is an element of this subshift, 
hence it is \textbf{not empty}. 

\item \textbf{Repetition of the supertiles:}

The order $m$ three-dimensional supertile 
appear periodically in any order $n \ge m$ 
three-dimensional supertile with period $2^{m+2}$, 
horizontally and 
vertically. This comes directly from the similar property 
of the two-dimensional supertiles. 
This is also true inside an infinite supertile. 
Because we use the rigid 
version of the Robinson subshift, 
this is also true for the whole configuration, in any configuration of the subshift. 

\item \textbf{Completion result:}

\begin{proposition}
Any $n$-block in this layer can be completed into an order $k$ cubic supertile, 
with $k \ge \Bigl \lceil \log_2 (n) 
\Bigr \rceil + 4$. Moreover it can be complete
in an order $k$ three-dimensional cell, with 
\[k \ge \Bigl \lceil \frac{\lfloor 
\log_2 (n)}{2} \Bigr \rceil +2.\]
\end{proposition}

We don't write the proof of this proposition, 
since this is similar to the 
two-dimensional version of the subshift. Moreover, 
this version is also minimal:

\begin{corollary}
This three-dimensional version of 
the Robinson subshift is minimal.
\end{corollary}

\end{enumerate}

\subsubsection{Coloration}

We use colors in order to simplify 
the representations of the 
configurations of this layer. The positions 
in the edges of the cubes 
are characterized by having three petal symbols 
with $0,1$-counter equal to $1$ (recall 
that we call value of the $0,1$-counter the 
symbols in $\{0,1\}$ on corner symbols 
of the Robinson subshift),
and we represent this by the symbol 
$\begin{tikzpicture}[scale=0.3]
\fill[gray!90] (0,0) rectangle (1,1);
\draw (0,0) rectangle (1,1);
\end{tikzpicture}$. The faces positions 
have exactly 
two such symbols, and are represented by 
$\begin{tikzpicture}[scale=0.3]
\fill[gray!60] (0,0) rectangle (1,1);
\draw (0,0) rectangle (1,1);
\end{tikzpicture}$.
The other faces connecting the edges of the cubes are 
colored
with $\begin{tikzpicture}[scale=0.3]
\fill[gray!20] (0,0) rectangle (1,1);
\draw (0,0) rectangle (1,1);
\end{tikzpicture}$, and 
are characterized by having 
a unique petal symbol 
with $0,1$-counter value 
equal to $1$. See on Figure~\ref{fig.coloration} 
the representation of the surroundings 
of a three-dimensional cells with these 
colorations.

\subsection{\label{sec.func.areas} Functional areas}

In this section, we describe how to draw 
functional areas on these faces.
This means that we attribute local
functions to the positions of these faces 
realizing the global functions 
of the counters and machine computations.
These local functions are the execution 
of one step of computation (including the incrementation
step for the counter), and horizontal and 
vertical transmission of information.

Moreover, we use an addressing mechanism 
that allows the selection of sparse sub-areas 
of this functional areas, so that 
the global functions do not 
contribute to entropy dimension.

\textbf{\textit{In this section, each 
sublayer is presented as the
superimposition, 
on the colored faces, of symbols 
in a finite set $\mathcal{A}$.
As a consequence, the positions 
in the intersection of two faces 
are superimposed with a couple 
of these symbols, 
and the other positions 
in the faces with just one of 
these symbols. In order to keep 
the descriptions as simple 
as possible, each 
sublayer is presented as having 
alphabet $\mathcal{A}$, 
while the real alphabet has to include the 
elements of $\mathcal{A}^2$, 
and the real set of rules corresponds to 
this alphabet. These rules can 
be deduced easily from the descriptions 
above.}}

Moreover, in each sublayer, the non-blank symbols 
are superimposed on and only on petals of the 
copy of the Robinson parallel to this face - we will simply refer to these as 
petals of this face.

We recall that the petals 
having $0,1$-counter value equal to $0$
are called \textbf{transmission petals}, 
and 
the other ones are called \textbf{support petals}.

\subsubsection{\label{sec.orientation} Orientation in the hierarchy}

The purpose of this first sublayer is to give access, to 
the support petal of each colored face,
to the orientation of this petal relatively to the 
support petal just above in the hierarchy. \bigskip

\begin{figure}[ht]
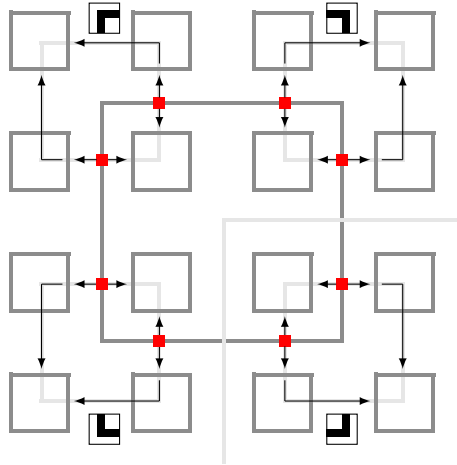

\[
};

\fill[red] (15.5,39.5) rectangle (17,41);
\fill[red] (23,47) rectangle (24.5,48.5);

\fill[red] (15.5+31.5,39.5) rectangle (17+31.5,41);
\fill[red] (23+16.5,47) rectangle (24.5+16.5,48.5);

\fill[red] (15.5,39.5-16.5) rectangle (17,41-16.5);
\fill[red] (15.5+31.5,39.5-16.5) rectangle (17+31.5,41-16.5);

\fill[red] (23,47-31.5) rectangle (24.5,48.5-31.5);
\fill[red] (23+16.5,47-31.5) rectangle (24.5+16.5,48.5-31.5);

%nordouest

\draw[-latex] (23.75,48.5) -- (23.75,51.5);
\draw[-latex] (23.75,47) -- (23.75,44.5);
\draw (23.75,53) -- (23.75,55.75);
\draw[-latex] (23.75,55.75) -- (12.5,55.75);
\draw[-latex] (15.5,40.25) -- (12.5,40.25);
\draw[-latex] (17,40.25) -- (19.5,40.25);

\draw (11,40.25) -- (8.25,40.25);
\draw[-latex] (8.25,40.25) -- (8.25,51.5);
\node at (16.5,59) {
};
\end{tikzpicture}\]
\caption{\label{fig.orientation} Schematic illustration of the
orientation rules, showing a support petal and the 
support petals just under this one in 
the hierarchy, 
all of them colored dark 
gray. The transmission 
petals connecting 
them are colored light gray.
Transformation positions are colored with a red square. The arrows give the natural interpretation 
of the propagation direction 
of the signal transmitting 
the orientation information.}
\end{figure}

\noindent \textbf{\textit{Symbols:}} \bigskip

The symbols are 
elements of \[\left\{%
 \right\},\] and a blank symbol. \bigskip

\noindent \textbf{\textit{Local rules:}}

\begin{itemize}
\item \textbf{Localization:} the non blank symbols are superimposed 
on and only on positions with petal symbols of the gray faces.
The symbol is transmitted through the petals, except on 
transformation positions, defined just below.
\item \textbf{Transformation positions:} the 
\textbf{transformation positions} are the 
positions where the transformation rule occurs 
(meaning that the signal is tranformed). 
These positions depend on the (sub)layer.
In this sublayer, these are the positions 
where a support petal intersects 
a transmission petal just 
under in the hierarchy. 
On these positions is superimposed a \textbf{couple 
of symbols}, in 
\[\left\{ \right\},\] 
while the other petal positions are superimposed 
with a \textbf{simple symbol}, in 
\[\left\{%
 \right\}.\]
On the transmission positions, 
the first symbol corresponds 
to outgoing arrows (this is 
the direction from which the signal 
comes, that is to say the support petal), 
and the second one to incoming arrows 
(this is the direction where the signal is 
transmitted, after transformation).
\item \textbf{Transformation:} on a transformation
position, the second bit of the
couple is $%
$ if the six arrows symbol is 
\[%
.\]
These positions correspond to the 
intersection of an 
order $n+1$ support petal  
with the \textbf{north west} order $n$ 
transmission petal just under in the hierarchy. 
\textit{There are similar rules for the other orientations}.
\item \textbf{Border rule:} on the border of a 
gray face, the symbol is blank if not on a transmission 
position. On a transmission position, the
first symbol of the couple is blank.
\end{itemize}

\noindent \textbf{\textit{Global behavior:}} \bigskip

This layer supports a signal that propagates
through the petal 
hierarchy on the colored faces.
This signal is transmitted through 
the petals except on the intersections 
of a support petal and a 
transmission petal just above 
in the hierarchy. On these positions, 
the symbol transmitted by the 
signal is transformed into a symbol 
representing the orientation of the transmission 
petal with respect to the support petal. 

As a consequence, the support petals 
just under the transmission petal 
are colored with this orientation symbol.
See the schema on Figure~\ref{fig.orientation}.

\subsubsection{\label{sec.functional.areas.min} 
Functional areas}

\begin{figure}[ht]
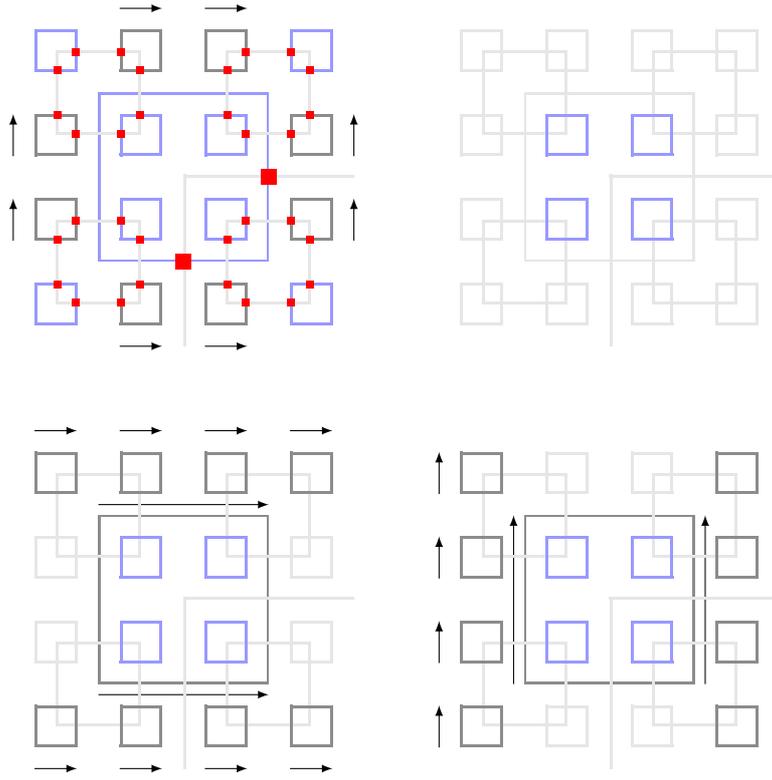

\[
\]
\caption{\label{fig.trans.comp.area} 
Schematic illustration of 
the transmission rules of the functional areas sublayer.
The transformation positions are marked with red squares
on the first schema.}
\end{figure}

\noindent \textbf{\textit{Symbols:}} \bigskip

$%
,\uparrow \right)\right\}^2$. \bigskip

\noindent \textbf{\textit{Local rules:}} \bigskip
 
\begin{itemize}
\item \textbf{Localization:} 
the non blank symbols are superimposed on
and only on petal positions of the colored 
faces. 
\item the couples of symbols 
are superimposed 
on six arrows symbols positions 
where the border of a cell 
intersects the petal just 
above in the hierarchy (\textbf{transformation 
positions}).
\item the symbols are transmitted through through 
the petals except on transformation 
positions.
\item {\bf{Transformation through hierarchy}}. 
On the transformation positions, 
if the first symbol is
, then the second one is equal. 
For the other cases, if the symbol is 
\[%
,\]
then second color is according to the following rules. The condition 
on the symbol above corresponds
to positions 
where a support 
petal intersects 
the transmission petal 
just above in the hierarchy 
and being oriented 
\textbf{north west} 
with respect to this 
transmission petal 
(examples of 
such positions 
are represented
with large squares on 
Figure~\ref{fig.trans.comp.area}).

\begin{enumerate}
\item if the orientation symbol is $%
$, the second symbol is equal to the first one. 
\item if the orientation symbol is   $%
$, the second symbol 
is a function of the first one: 
\begin{itemize}
\item if the first symbol of the couple is %
,  then the second symbol is (%
,$\rightarrow$). 
\item if the first symbol is (%
,$\rightarrow$), then the second one is equal.
\item if the first symbol is $(%
, \uparrow)$, then the second symbol is 
%
. 
\end{itemize}
\item if the orientation symbol is  
$%
$, the second color is 
%
.
\item if the orientation symbol is 
$%
$, the second symbol 
is a function of the first one: 
\begin{itemize}
\item if the first symbol of the couple is or %
,  then the second symbol is (%
,$\rightarrow$). 
\item if the first symbol is (%
,$\rightarrow$), then the second one is %
. 
\item if the first symbol is $(%
, \uparrow)$, then the second symbol is equal.
\end{itemize}
\end{enumerate}

For the other orientations 
of the support 
petal with respect 
to the transmission 
petal just above, 
the rules are obtained 
by rotation.

See Figure~\ref{fig.trans.comp.area} for an illustration of these rules.
\item \textbf{Border rule:} on the border of a face of a three-dimensional cell, 
the symbol is blue.
\item \textbf{Coherence rules:} these rules 
enable a retroaction of the colors on infinite petals 
over the order $0$ petals that are nearby, ensuring a coherence 
amongst the hierarchy (they are similar to the the property $Q$ 
in Mozes' construction).
\begin{enumerate}
\item when near a \textbf{corner} 
of a three-dimensional cell, on a pattern
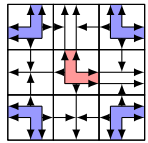
\[\begin{tikzpicture}[scale=0.3]
\robibluehautgauchek{4}{0}
\robibluehautdroitek{0}{0}
\robibluebasdroitek{0}{4}
\robibluebasgauchek{4}{4}
\robiredbasgauche{2}{2}
\robithreehaut{2}{4}
\robithreedroite{4}{2}
\robionehaut{2}{0}
\robionegauche{0}{2}
\draw[step=2] (0,0) grid (6,6);
\end{tikzpicture}\]
in one of the copies of the Robinson subshift, the corresponding pattern in this sublayer 
has to be 
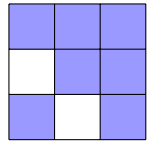
\[\begin{tikzpicture}[scale=0.3]
\fill[blue!40] (0,4) rectangle (2,6);
\fill[blue!40] (4,0) rectangle (6,2);
\fill[blue!40] (0,0) rectangle (2,2);
\fill[blue!40] (2,2) rectangle (6,6);
\draw[step=2] (0,0) grid (6,6);
\end{tikzpicture}.\]
There are similar rules for the other corners. This allows 
degenerated behaviors 
which do not happen around finite cells to be avoided. The other rules in this list 
have the same function. Their analysis is important to understand 
how it is possible to complete patterns of this subshift into 
three-dimensional cells.
\item near an \textbf{edge} of 
a three-dimensional cell, 
the symbols on the corresponding 
blue corners 
on the two faces connected by 
this edge have to match. 
Moreover, on the middle of an edge, 
the two blue corners inside each 
one of the two faces adjacent to the 
edge that are on the two sides of the middle, 
are colored blue. For instance, on the pattern
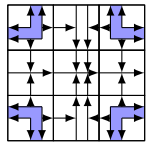
\[
\]
in any of the copies of the Robinson subshift,
the pattern in this layer has to be \[%
.\]
When on the quarter of the edge, the two symbols have to be all 
$(%
,\uparrow)$, the 
direction of the arrows 
corresponding 
to the orientation of the edge. 
For instance, on the pattern 
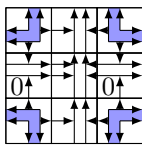
\[%
,\]
the pattern has to be: 
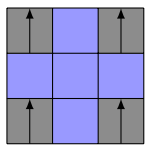
\[%
.\]
\item Inside a colored face, 
the degenerated behaviors 
correspond to the degenerated behaviors of 
the two-dimensional version of the Robinson 
subshift. When near a corner of a support 
petal inside the 
face, the symbols on the blue 
corners nearby have to be colored according 
to the color of the corner. 
For instance, on \[%
\]
in the copy of the Robinson 
subshift parallel to the face, 
the pattern in this layer has to be amongst the following 
ones: \[%
\]
We impose similar restrictions for the 
other orientations of the corner.
On the positions where  
two petals intersect, the two symbols inside the
two-dimensional cell 
have to be colored with blue, and the two other ones have to be marked with 
an arrow pointing in the direction parallel to 
the cell border. For instance, on the pattern 
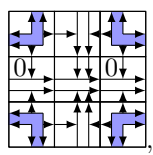
\[%
,\]
the pattern has to be: 
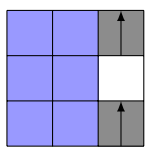
\[%
.\]

When near the border of a cell not 
intersecting another petal, 
then the only restriction is to have 
arrows marks outside, in the direction of the 
border.
\end{enumerate}
\end{itemize} \bigskip

\noindent \textbf{
\textit{Global behavior:}} \bigskip

Like the first sublayer presented in 
Section~\ref{sec.orientation}, 
the global behavior of 
this sublayer consists in coding 
a substitution. This coding 
uses a signal 
that propagates through the petal 
hierarchy, on any colored face.

This process is similar as the one 
used by Robinson~\cite{R71} in order to 
create areas supporting 
the computations of Turing machines 
in a $\Z^2$-SFT. However, we localize 
it here on faces of cubes in a three-dimensional 
SFT. \bigskip

The result of this process is 
that order zero two 
dimensional cell borders of a 
colored face are colored with 
a symbol 
which represent a \textbf{function} of 
the blue corners positions - called 
\textbf{functional positions} - 
included in the 
zero order petal just under this petal 
in the petal hierarchy.
These symbols and functions are as follows: 

\begin{itemize}
\item \textbf{blue} if the set of columns and 
the set of lines in which it is included do not
intersect larger order two-dimensional 
cells. The associated function
is to \textbf{support a step of computation} 
(which can be just to transfer the 
information in the case when 
the face support a counter),  
and the corresponding 
positions are called \textbf{computation 
positions}. 
\item an \textbf{horizontal} (resp. 
vertical)
\textbf{arrow} directed to the right (resp. to 
the top)
when the set of columns (resp. lines) containing 
this petal intersects larger order two-dimensional cells but not the set of lines (resp. 
columns) containing it. The associated 
function is to \textbf{transfer information} 
in the direction of the arrow (this 
information can be trivial in the case 
when the face support a counter. This means  
that the symbol transmitted is the blank 
symbol), and the corresponding positions 
are called \textbf{information transfer 
positions}.
\item when 
the two sets intersect larger order 
cells, the petal is colored \textbf{light gray}.
These positions have no function.
\end{itemize}

See on Figure~\ref{fig.func.area2} 
a schema of the functional positions over 
the faces of an order 
$3$ three-dimensional cell.

\begin{figure}[ht]
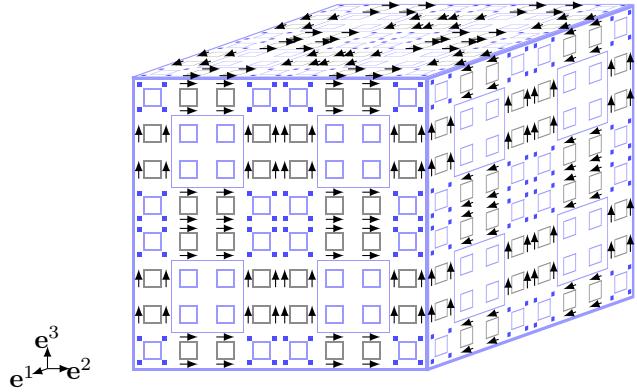

\[
\]
\caption{\label{fig.func.area2} Schema of the 
functional positions on the a face
of an order $3$ three-dimensional cell. The 
arrows are oriented according to the
fixed orientations of the face.}
\end{figure}

The aim of the following sections, 
Section~\ref{sec.p.counter}, Section~\ref{sec.p.addressing}, 
and Section~\ref{sec.active.functional.areas} is 
to add symbols on the colored faces in order 
to specify thin sub-areas of the functional 
areas. These sub-areas will 
support the computations of 
the machines and the 
linear counters. They are 
thin enough so that the machines 
and counters do not contribute to 
the entropy dimension. This 
entropy dimension is thus generated 
by the values of the hierarchical counter.

The first sublayer of this set 
encodes a counter, called $p$-counter. 
It specifies, on each support petal, the 
arithmetical class of $n$ 
modulo $p$, where $n$ is the order of the corresponding 
two-dimensional cell. 
The second sublayer, based on this $p$-counter, 
gives an address of the columns (resp. lines) 
amongst the \textbf{computation columns} (resp. lines), 
defined by intersecting computation positions in 
the face. In the third sublayer, we describe 
a process which selects a thin 
subset of these columns (resp. lines) according 
to the address.

\subsubsection{\label{sec.p.counter} The $p$-counter}

We recall that $p=2^m-1$ is an integer, fixed at 
the beginning of the construction. \bigskip

\noindent \textbf{\textit{Symbols:}} \bigskip

Elements of $\Z/p\Z$ 
and of $\left(\Z/p\Z \right)^2$
and a blank symbol. \bigskip

\noindent \textbf{\textit{Local rules:}} \bigskip

\begin{itemize}
\item \textbf{Localization:} 
\begin{itemize} 
\item the non blank symbols are superimposed on and only 
on gray faces, on petals of the copy of the Robinson subshift parallel 
to this face. 
\item The blue petals are superimposed with $\overline{0}$.
\item the couples are superimposed on \textbf{transformation positions} 
defined as the intersection positions of a support petal
and the transmission petal just above in the hierarchy.
\item the other positions have a simple symbol and 
this symbol is transmitted through these positions.
\end{itemize}
\item \textbf{Transformation rule:} on transformation positions, 
if the first bit is $\overline{i}$, then the second bit is 
$\overline{i}+\overline{1}$.
\item \textbf{Coherence rule:} on the edges connecting two faces, the
two values of the $p$-counter are equal. On the corners, the three values 
are equal.
\end{itemize}

\noindent \textbf{\textit{Global behavior:}} \bigskip

Each of the support petals on the colored faces 
is attached with some element of $\Z/p\Z$. This 
symbol is transmitted 
through the petals except on transformation 
positions. These positions 
are defined to be the intersections 
of the support petals with the transmission petal 
just above in the hierarchy, where 
the transmission petal has value $\overline{k}+\overline{1}$
and $\overline{k}$ is the value of the support petal.
As the blue petals are marked with $\overline{0}$, this 
imposes that the border petals of the order $n$ 
two-dimensional cells on the faces are marked 
with $\overline{n}$.

\subsubsection{\label{sec.p.addressing} The $p$-addressing}

This sublayer has two subsublayers, 
one for vertical addresses, and another 
for horizontal addresses.

\begin{figure}[ht]
\[
$};
\node at (56,56) {$w1$};

\fill[red] (11,7.5) rectangle (12.5,9);
\fill[red] (7.5,11) rectangle (9,12.5);

\fill[red] (19.5,7.5) 
rectangle (21,9);
\fill[red] (23,11) 
rectangle (24.5,12.5);

\fill[red] (43,7.5) rectangle (44.5,9);
\fill[red] (39.5,11) rectangle (41,12.5);

\fill[red] (51.5,7.5) 
rectangle (53,9);
\fill[red] (55,11) 
rectangle (56.5,12.5);

\fill[red] (11,39.5) rectangle (12.5,41);
\fill[red] (7.5,43) rectangle (9,44.5);

\fill[red] (19.5,39.5) 
rectangle (21,41);
\fill[red] (23,43) 
rectangle (24.5,44.5);

\fill[red] (43,39.5) rectangle (44.5,41);
\fill[red] (39.5,43) rectangle (41,44.5);

\fill[red] (51.5,39.5) 
rectangle (53,41);
\fill[red] (55,43) 
rectangle (56.5,44.5);

%%%%%%%%%%%%%%%%%%%%%%%

\fill[red] (11,23) rectangle (12.5,24.5);
\fill[red] (7.5,19.5) rectangle (9,21);

\fill[red] (19.5,23) rectangle (21,24.5);
\fill[red] (23,19.5) rectangle (24.5,21);

\fill[red] (43,23) rectangle (44.5,24.5);
\fill[red] (39.5,19.5) rectangle (41,21);

\fill[red] (51.5,23) rectangle (53,24.5);
\fill[red] (55,19.5) rectangle (56.5,21);

\fill[red] (11,55) rectangle (12.5,56.5);
\fill[red] (7.5,51.5) rectangle (9,53);

\fill[red] (19.5,55) rectangle (21,56.5);
\fill[red] (23,51.5) rectangle (24.5,53);

\fill[red] (43,55) rectangle (44.5,56.5);
\fill[red] (39.5,51.5) rectangle (41,53);

\fill[red] (51.5,55) rectangle (53,56.5);
\fill[red] (55,51.5) rectangle (56.5,53);

\fill[red] (30.5,14.5) rectangle 
(33.5,17.5);

\fill[red] (46.5,30.5) rectangle 
(49.5,33.5);

\end{scope}

\begin{scope}[xshift=84cm]
\fill[gray!90] (16,16) rectangle (48,48); 
\fill[white] (16.5,16.5) rectangle (47.5,47.5);
\fill[gray!20] (32,0) rectangle (32.5,32.5);
\fill[gray!20] (32,32.5) rectangle (64,32);

\fill[gray!20] (8,8) rectangle (8.5,24); 
\fill[gray!20] (8,8) rectangle (24,8.5);
\fill[gray!20] (8,23.5) rectangle (24,24);  
\fill[gray!20] (23.5,8) rectangle (24,24);

\fill[gray!20] (40,8) rectangle (40.5,24); 
\fill[gray!20] (40,8) rectangle (56,8.5);
\fill[gray!20] (40,23.5) rectangle (56,24);  
\fill[gray!20] (55.5,8) rectangle (56,24);

\fill[gray!20] (8,40) rectangle (8.5,56); 
\fill[gray!20] (8,40) rectangle (24,40.5);
\fill[gray!20] (8,55.5) rectangle (24,56);  
\fill[gray!20] (23.5,40) rectangle (24,56);

\fill[gray!20] (40,40) rectangle (40.5,56); 
\fill[gray!20] (40,40) rectangle (56,40.5);
\fill[gray!20] (40,55.5) rectangle (56,56);  
\fill[gray!20] (55.5,40) rectangle (56,56);

\fill[gray!90] (4,4) rectangle (4.5,12); 
\fill[gray!90] (4,4) rectangle (12,4.5); 
\fill[gray!90] (4.5,11.5) rectangle (12,12); 
\fill[gray!90] (11.5,4.5) rectangle (12,12); 

\fill[gray!90] (4,20) rectangle (4.5,28); 
\fill[gray!90] (4,20) rectangle (12,20.5); 
\fill[gray!90] (4.5,27.5) rectangle (12,28); 
\fill[gray!90] (11.5,20.5) rectangle (12,28); 

\fill[gray!90] (4,36) rectangle (4.5,44); 
\fill[gray!90] (4,36) rectangle (12,36.5); 
\fill[gray!90] (4.5,43.5) rectangle (12,44); 
\fill[gray!90] (11.5,36.5) rectangle (12,44);

\fill[gray!90] (4,52) rectangle (4.5,60); 
\fill[gray!90] (4,52) rectangle (12,52.5); 
\fill[gray!90] (4.5,59.5) rectangle (12,60); 
\fill[gray!90] (11.5,52.5) rectangle (12,60);

\fill[gray!90] (20,4) rectangle (20.5,12); 
\fill[gray!90] (20,4) rectangle (28,4.5); 
\fill[gray!90] (20.5,11.5) rectangle (28,12); 
\fill[gray!90] (27.5,4.5) rectangle (28,12); 

\fill[gray!90] (20,20) rectangle (20.5,28); 
\fill[gray!90] (20,20) rectangle (28,20.5); 
\fill[gray!90] (20.5,27.5) rectangle (28,28); 
\fill[gray!90] (27.5,20.5) rectangle (28,28); 

\fill[gray!90] (20,36) rectangle (20.5,44); 
\fill[gray!90] (20,36) rectangle (28,36.5); 
\fill[gray!90] (20.5,43.5) rectangle (28,44); 
\fill[gray!90] (27.5,36.5) rectangle (28,44);

\fill[gray!90] (20,52) rectangle (20.5,60); 
\fill[gray!90] (20,52) rectangle (28,52.5); 
\fill[gray!90] (20.5,59.5) rectangle (28,60); 
\fill[gray!90] (27.5,52.5) rectangle (28,60);

\fill[gray!90] (36,4) rectangle (36.5,12); 
\fill[gray!90] (36,4) rectangle (44,4.5); 
\fill[gray!90] (36.5,11.5) rectangle (44,12); 
\fill[gray!90] (43.5,4.5) rectangle (44,12); 

\fill[gray!90] (36,20) rectangle (36.5,28); 
\fill[gray!90] (36,20) rectangle (44,20.5); 
\fill[gray!90] (36.5,27.5) rectangle (44,28); 
\fill[gray!90] (43.5,20.5) rectangle (44,28); 

\fill[gray!90] (36,36) rectangle (36.5,44); 
\fill[gray!90] (36,36) rectangle (44,36.5); 
\fill[gray!90] (36.5,43.5) rectangle (44,44); 
\fill[gray!90] (43.5,36.5) rectangle (44,44);

\fill[gray!90] (36,52) rectangle (36.5,60); 
\fill[gray!90] (36,52) rectangle (44,52.5); 
\fill[gray!90] (36.5,59.5) rectangle (44,60); 
\fill[gray!90] (43.5,52.5) rectangle (44,60);

\fill[gray!90] (52,4) rectangle (52.5,12); 
\fill[gray!90] (52,4) rectangle (60,4.5); 
\fill[gray!90] (52.5,11.5) rectangle (60,12); 
\fill[gray!90] (59.5,4.5) rectangle (60,12); 

\fill[gray!90] (52,20) rectangle (52.5,28); 
\fill[gray!90] (52,20) rectangle (60,20.5); 
\fill[gray!90] (52.5,27.5) rectangle (60,28); 
\fill[gray!90] (59.5,20.5) rectangle (60,28); 

\fill[gray!90] (52,36) rectangle (52.5,44); 
\fill[gray!90] (52,36) rectangle (60,36.5); 
\fill[gray!90] (52.5,43.5) rectangle (60,44); 
\fill[gray!90] (59.5,36.5) rectangle (60,44);

\fill[gray!90] (52,52) rectangle (52.5,60); 
\fill[gray!90] (52,52) rectangle (60,52.5); 
\fill[gray!90] (52.5,59.5) rectangle (60,60); 
\fill[gray!90] (59.5,52.5) rectangle (60,60);

\node at (8,8) {$0$};
\node at (24,8) {$
\begin{tikzpicture}[scale=0.3]
\fill[gray!90] (0,0) rectangle (1,1);
\draw (0,0) rectangle (1,1);
\end{tikzpicture}$};
\node at (8,24) {$0$};
\node at (24,24) {$0$};

\node at (8,40) {$0$};
\node at (24,40) {$0$};
\node at (8,56) {$0$};
\node at (24,56) {$
\begin{tikzpicture}[scale=0.3]
\fill[gray!90] (0,0) rectangle (1,1);
\draw (0,0) rectangle (1,1);
\end{tikzpicture}$};

\draw[-latex] (4,32) -- (16,32);
\node at (-8,32) {$|w|=p$};

\node at (40,8) {$
\begin{tikzpicture}[scale=0.3]
\fill[gray!90] (0,0) rectangle (1,1);
\draw (0,0) rectangle (1,1);
\end{tikzpicture}$};
\node at (56,8) {$1$};
\node at (40,24) {$1$};
\node at (56,24) {$1$};

\node at (40,40) {$1$};
\node at (56,40) {$1$};
\node at (40,56) {$
\begin{tikzpicture}[scale=0.3]
\fill[gray!90] (0,0) rectangle (1,1);
\draw (0,0) rectangle (1,1);
\end{tikzpicture}$};
\node at (56,56) {$1$};

\end{scope}

\begin{scope}[yshift=-70cm]
\fill[gray!90] (16,16) rectangle (48,48); 
\fill[white] (16.5,16.5) rectangle (47.5,47.5);
\fill[gray!20] (32,0) rectangle (32.5,32.5);
\fill[gray!20] (32,32.5) rectangle (64,32);

\fill[gray!20] (8,8) rectangle (8.5,24); 
\fill[gray!20] (8,8) rectangle (24,8.5);
\fill[gray!20] (8,23.5) rectangle (24,24);  
\fill[gray!20] (23.5,8) rectangle (24,24);

\fill[gray!20] (40,8) rectangle (40.5,24); 
\fill[gray!20] (40,8) rectangle (56,8.5);
\fill[gray!20] (40,23.5) rectangle (56,24);  
\fill[gray!20] (55.5,8) rectangle (56,24);

\fill[gray!20] (8,40) rectangle (8.5,56); 
\fill[gray!20] (8,40) rectangle (24,40.5);
\fill[gray!20] (8,55.5) rectangle (24,56);  
\fill[gray!20] (23.5,40) rectangle (24,56);

\fill[gray!20] (40,40) rectangle (40.5,56); 
\fill[gray!20] (40,40) rectangle (56,40.5);
\fill[gray!20] (40,55.5) rectangle (56,56);  
\fill[gray!20] (55.5,40) rectangle (56,56);

\fill[gray!90] (4,4) rectangle (4.5,12); 
\fill[gray!90] (4,4) rectangle (12,4.5); 
\fill[gray!90] (4.5,11.5) rectangle (12,12); 
\fill[gray!90] (11.5,4.5) rectangle (12,12); 

\fill[gray!90] (4,20) rectangle (4.5,28); 
\fill[gray!90] (4,20) rectangle (12,20.5); 
\fill[gray!90] (4.5,27.5) rectangle (12,28); 
\fill[gray!90] (11.5,20.5) rectangle (12,28); 

\fill[gray!90] (4,36) rectangle (4.5,44); 
\fill[gray!90] (4,36) rectangle (12,36.5); 
\fill[gray!90] (4.5,43.5) rectangle (12,44); 
\fill[gray!90] (11.5,36.5) rectangle (12,44);

\fill[gray!90] (4,52) rectangle (4.5,60); 
\fill[gray!90] (4,52) rectangle (12,52.5); 
\fill[gray!90] (4.5,59.5) rectangle (12,60); 
\fill[gray!90] (11.5,52.5) rectangle (12,60);

\fill[gray!90] (20,4) rectangle (20.5,12); 
\fill[gray!90] (20,4) rectangle (28,4.5); 
\fill[gray!90] (20.5,11.5) rectangle (28,12); 
\fill[gray!90] (27.5,4.5) rectangle (28,12); 

\fill[gray!90] (20,20) rectangle (20.5,28); 
\fill[gray!90] (20,20) rectangle (28,20.5); 
\fill[gray!90] (20.5,27.5) rectangle (28,28); 
\fill[gray!90] (27.5,20.5) rectangle (28,28); 

\fill[gray!90] (20,36) rectangle (20.5,44); 
\fill[gray!90] (20,36) rectangle (28,36.5); 
\fill[gray!90] (20.5,43.5) rectangle (28,44); 
\fill[gray!90] (27.5,36.5) rectangle (28,44);

\fill[gray!90] (20,52) rectangle (20.5,60); 
\fill[gray!90] (20,52) rectangle (28,52.5); 
\fill[gray!90] (20.5,59.5) rectangle (28,60); 
\fill[gray!90] (27.5,52.5) rectangle (28,60);

\fill[gray!90] (36,4) rectangle (36.5,12); 
\fill[gray!90] (36,4) rectangle (44,4.5); 
\fill[gray!90] (36.5,11.5) rectangle (44,12); 
\fill[gray!90] (43.5,4.5) rectangle (44,12); 

\fill[gray!90] (36,20) rectangle (36.5,28); 
\fill[gray!90] (36,20) rectangle (44,20.5); 
\fill[gray!90] (36.5,27.5) rectangle (44,28); 
\fill[gray!90] (43.5,20.5) rectangle (44,28); 

\fill[gray!90] (36,36) rectangle (36.5,44); 
\fill[gray!90] (36,36) rectangle (44,36.5); 
\fill[gray!90] (36.5,43.5) rectangle (44,44); 
\fill[gray!90] (43.5,36.5) rectangle (44,44);

\fill[gray!90] (36,52) rectangle (36.5,60); 
\fill[gray!90] (36,52) rectangle (44,52.5); 
\fill[gray!90] (36.5,59.5) rectangle (44,60); 
\fill[gray!90] (43.5,52.5) rectangle (44,60);

\fill[gray!90] (52,4) rectangle (52.5,12); 
\fill[gray!90] (52,4) rectangle (60,4.5); 
\fill[gray!90] (52.5,11.5) rectangle (60,12); 
\fill[gray!90] (59.5,4.5) rectangle (60,12); 

\fill[gray!90] (52,20) rectangle (52.5,28); 
\fill[gray!90] (52,20) rectangle (60,20.5); 
\fill[gray!90] (52.5,27.5) rectangle (60,28); 
\fill[gray!90] (59.5,20.5) rectangle (60,28); 

\fill[gray!90] (52,36) rectangle (52.5,44); 
\fill[gray!90] (52,36) rectangle (60,36.5); 
\fill[gray!90] (52.5,43.5) rectangle (60,44); 
\fill[gray!90] (59.5,36.5) rectangle (60,44);

\fill[gray!90] (52,52) rectangle (52.5,60); 
\fill[gray!90] (52,52) rectangle (60,52.5); 
\fill[gray!90] (52.5,59.5) rectangle (60,60); 
\fill[gray!90] (59.5,52.5) rectangle (60,60);

\node at (8,8) {$
$};

\end{scope}

\begin{scope}[xshift=84cm,yshift=-70cm]
\fill[gray!90] (16,16) rectangle (48,48); 
\fill[white] (16.5,16.5) rectangle (47.5,47.5);
\fill[gray!20] (32,0) rectangle (32.5,32.5);
\fill[gray!20] (32,32.5) rectangle (64,32);

\fill[gray!20] (8,8) rectangle (8.5,24); 
\fill[gray!20] (8,8) rectangle (24,8.5);
\fill[gray!20] (8,23.5) rectangle (24,24);  
\fill[gray!20] (23.5,8) rectangle (24,24);

\fill[gray!20] (40,8) rectangle (40.5,24); 
\fill[gray!20] (40,8) rectangle (56,8.5);
\fill[gray!20] (40,23.5) rectangle (56,24);  
\fill[gray!20] (55.5,8) rectangle (56,24);

\fill[gray!20] (8,40) rectangle (8.5,56); 
\fill[gray!20] (8,40) rectangle (24,40.5);
\fill[gray!20] (8,55.5) rectangle (24,56);  
\fill[gray!20] (23.5,40) rectangle (24,56);

\fill[gray!20] (40,40) rectangle (40.5,56); 
\fill[gray!20] (40,40) rectangle (56,40.5);
\fill[gray!20] (40,55.5) rectangle (56,56);  
\fill[gray!20] (55.5,40) rectangle (56,56);

\fill[gray!90] (4,4) rectangle (4.5,12); 
\fill[gray!90] (4,4) rectangle (12,4.5); 
\fill[gray!90] (4.5,11.5) rectangle (12,12); 
\fill[gray!90] (11.5,4.5) rectangle (12,12); 

\fill[gray!90] (4,20) rectangle (4.5,28); 
\fill[gray!90] (4,20) rectangle (12,20.5); 
\fill[gray!90] (4.5,27.5) rectangle (12,28); 
\fill[gray!90] (11.5,20.5) rectangle (12,28); 

\fill[gray!90] (4,36) rectangle (4.5,44); 
\fill[gray!90] (4,36) rectangle (12,36.5); 
\fill[gray!90] (4.5,43.5) rectangle (12,44); 
\fill[gray!90] (11.5,36.5) rectangle (12,44);

\fill[gray!90] (4,52) rectangle (4.5,60); 
\fill[gray!90] (4,52) rectangle (12,52.5); 
\fill[gray!90] (4.5,59.5) rectangle (12,60); 
\fill[gray!90] (11.5,52.5) rectangle (12,60);

\fill[gray!90] (20,4) rectangle (20.5,12); 
\fill[gray!90] (20,4) rectangle (28,4.5); 
\fill[gray!90] (20.5,11.5) rectangle (28,12); 
\fill[gray!90] (27.5,4.5) rectangle (28,12); 

\fill[gray!90] (20,20) rectangle (20.5,28); 
\fill[gray!90] (20,20) rectangle (28,20.5); 
\fill[gray!90] (20.5,27.5) rectangle (28,28); 
\fill[gray!90] (27.5,20.5) rectangle (28,28); 

\fill[gray!90] (20,36) rectangle (20.5,44); 
\fill[gray!90] (20,36) rectangle (28,36.5); 
\fill[gray!90] (20.5,43.5) rectangle (28,44); 
\fill[gray!90] (27.5,36.5) rectangle (28,44);

\fill[gray!90] (20,52) rectangle (20.5,60); 
\fill[gray!90] (20,52) rectangle (28,52.5); 
\fill[gray!90] (20.5,59.5) rectangle (28,60); 
\fill[gray!90] (27.5,52.5) rectangle (28,60);

\fill[gray!90] (36,4) rectangle (36.5,12); 
\fill[gray!90] (36,4) rectangle (44,4.5); 
\fill[gray!90] (36.5,11.5) rectangle (44,12); 
\fill[gray!90] (43.5,4.5) rectangle (44,12); 

\fill[gray!90] (36,20) rectangle (36.5,28); 
\fill[gray!90] (36,20) rectangle (44,20.5); 
\fill[gray!90] (36.5,27.5) rectangle (44,28); 
\fill[gray!90] (43.5,20.5) rectangle (44,28); 

\fill[gray!90] (36,36) rectangle (36.5,44); 
\fill[gray!90] (36,36) rectangle (44,36.5); 
\fill[gray!90] (36.5,43.5) rectangle (44,44); 
\fill[gray!90] (43.5,36.5) rectangle (44,44);

\fill[gray!90] (36,52) rectangle (36.5,60); 
\fill[gray!90] (36,52) rectangle (44,52.5); 
\fill[gray!90] (36.5,59.5) rectangle (44,60); 
\fill[gray!90] (43.5,52.5) rectangle (44,60);

\fill[gray!90] (52,4) rectangle (52.5,12); 
\fill[gray!90] (52,4) rectangle (60,4.5); 
\fill[gray!90] (52.5,11.5) rectangle (60,12); 
\fill[gray!90] (59.5,4.5) rectangle (60,12); 

\fill[gray!90] (52,20) rectangle (52.5,28); 
\fill[gray!90] (52,20) rectangle (60,20.5); 
\fill[gray!90] (52.5,27.5) rectangle (60,28); 
\fill[gray!90] (59.5,20.5) rectangle (60,28); 

\fill[gray!90] (52,36) rectangle (52.5,44); 
\fill[gray!90] (52,36) rectangle (60,36.5); 
\fill[gray!90] (52.5,43.5) rectangle (60,44); 
\fill[gray!90] (59.5,36.5) rectangle (60,44);

\fill[gray!90] (52,52) rectangle (52.5,60); 
\fill[gray!90] (52,52) rectangle (60,52.5); 
\fill[gray!90] (52.5,59.5) rectangle (60,60); 
\fill[gray!90] (59.5,52.5) rectangle (60,60);

\node at (8,8) {$
\begin{tikzpicture}[scale=0.3]
\fill[gray!90] (0,0) rectangle (1,1);
\draw (0,0) rectangle (1,1);
\end{tikzpicture}$};
\node at (24,8) {$
\begin{tikzpicture}[scale=0.3]
\fill[gray!90] (0,0) rectangle (1,1);
\draw (0,0) rectangle (1,1);
\end{tikzpicture}$};
\node at (8,24) {$
\begin{tikzpicture}[scale=0.3]
\fill[gray!90] (0,0) rectangle (1,1);
\draw (0,0) rectangle (1,1);
\end{tikzpicture}$};
\node at (24,24) {$0$};

\node at (8,40) {$
\begin{tikzpicture}[scale=0.3]
\fill[gray!90] (0,0) rectangle (1,1);
\draw (0,0) rectangle (1,1);
\end{tikzpicture}$};
\node at (24,40) {$0$};
\node at (8,56) {$
\begin{tikzpicture}[scale=0.3]
\fill[gray!90] (0,0) rectangle (1,1);
\draw (0,0) rectangle (1,1);
\end{tikzpicture}$};
\node at (24,56) {$
\begin{tikzpicture}[scale=0.3]
\fill[gray!90] (0,0) rectangle (1,1);
\draw (0,0) rectangle (1,1);
\end{tikzpicture}$};

\draw[-latex] (4,32) -- (16,32);
\node at (-8,32) {$
\begin{tikzpicture}[scale=0.3]
\fill[gray!90] (0,0) rectangle (1,1);
\draw (0,0) rectangle (1,1);
\end{tikzpicture}$};

\node at (40,8) {$
\begin{tikzpicture}[scale=0.3]
\fill[gray!90] (0,0) rectangle (1,1);
\draw (0,0) rectangle (1,1);
\end{tikzpicture}$};
\node at (56,8) {$
\begin{tikzpicture}[scale=0.3]
\fill[gray!90] (0,0) rectangle (1,1);
\draw (0,0) rectangle (1,1);
\end{tikzpicture}$};
\node at (40,24) {$1$};
\node at (56,24) {$
\begin{tikzpicture}[scale=0.3]
\fill[gray!90] (0,0) rectangle (1,1);
\draw (0,0) rectangle (1,1);
\end{tikzpicture}$};

\node at (40,40) {$1$};
\node at (56,40) {$
\begin{tikzpicture}[scale=0.3]
\fill[gray!90] (0,0) rectangle (1,1);
\draw (0,0) rectangle (1,1);
\end{tikzpicture}$};
\node at (40,56) {$
\begin{tikzpicture}[scale=0.3]
\fill[gray!90] (0,0) rectangle (1,1);
\draw (0,0) rectangle (1,1);
\end{tikzpicture}$};
\node at (56,56) {$
\begin{tikzpicture}[scale=0.3]
\fill[gray!90] (0,0) rectangle (1,1);
\draw (0,0) rectangle (1,1);
\end{tikzpicture}$};

\end{scope}

\end{tikzpicture}\]
\caption{\label{fig.trans.selec.comp.area} 
Schemata 
of the transformation rules of the vertical addressing.
On the two schemata on the left, the 
central petal has $p$-counter value not 
equal to $\overline{0}$. On the two 
schemata on the right, this value is 
$\overline{0}$. The transformation 
positions are symbolized by red squares.
The symbols inside the little petals is 
the symbol attached to them in this sublayer.}
\end{figure}
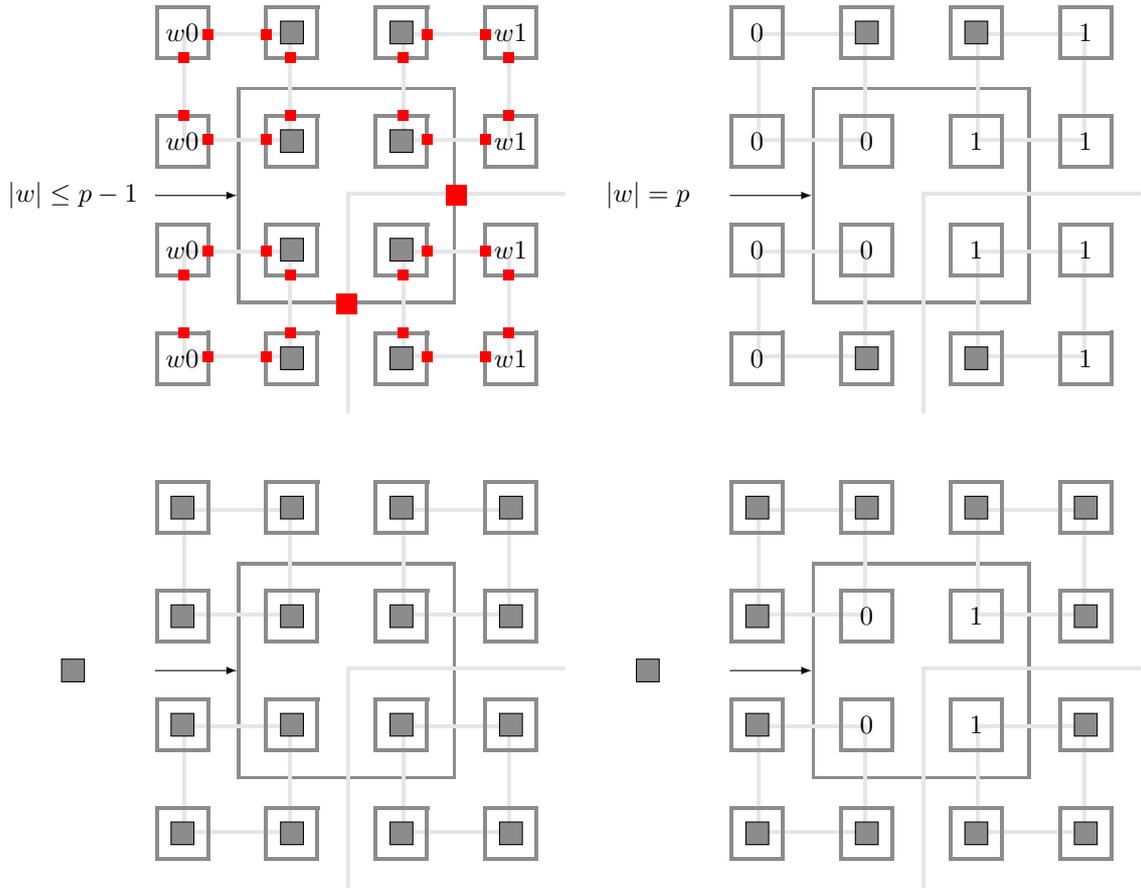

\paragraph{Vertical addressing:} \bigskip

\noindent \textbf{\textit{Symbols:}} \bigskip

The symbols of this subsublayer are the 
following ones. 

\begin{itemize}
\item Length $k$ words on the alphabet 
$\{0,1\}$, for $k=1 ... p$: these 
symbols are called the \textbf{vertical 
address} of support petals 
relatively to the next support
petal above in the hierarchy. 
\item The couples $(w,w')$ 
such that $w$ and $w'$ are words on 
alphabet $\{0,1\}$ with respective lengths 
$k$ and $k+1$ with $k < p$, and 
$(w,
\begin{tikzpicture}[scale=0.3]
\fill[gray!90] (0,0) rectangle (1,1);
\draw (0,0) rectangle (1,1);
\end{tikzpicture})$ (they represent 
possible \textbf{transformations} of simple symbols
when going through a support petal corresponding 
to a two-dimensional cell having order different 
from any $mp$, $m \ge 0$).
\item The couples $(w,0)$ and $(w,1)$, where $w$ 
is a length $p$ word on $\{0,1\}$ (\textbf{transformations} 
of simple symbols when going through 
a petal corresponding to order $mp$ cells).
\item The symbols $\begin{tikzpicture}[scale=0.3]
\fill[gray!90] (0,0) rectangle (1,1);
\draw (0,0) rectangle (1,1);
\end{tikzpicture}$, $\left(\begin{tikzpicture}[scale=0.3]
\fill[gray!90] (0,0) rectangle (1,1);
\draw (0,0) rectangle (1,1);
\end{tikzpicture},\begin{tikzpicture}[scale=0.3]
\fill[gray!90] (0,0) rectangle (1,1);
\draw (0,0) rectangle (1,1);
\end{tikzpicture}\right)$, $\left(w,\begin{tikzpicture}[scale=0.3]
\fill[gray!90] (0,0) rectangle (1,1);
\draw (0,0) rectangle (1,1);
\end{tikzpicture}\right)$, $\left(\begin{tikzpicture}[scale=0.3]
\fill[gray!90] (0,0) rectangle (1,1);
\draw (0,0) rectangle (1,1);
\end{tikzpicture},0\right)$ with 
$w$ a word on $\{0,1\}$ having 
length between $1$ and $p$, 
$\left(\begin{tikzpicture}[scale=0.3]
\fill[gray!90] (0,0) rectangle (1,1);
\draw (0,0) rectangle (1,1);
\end{tikzpicture},1\right)$,
\item A blank symbol. 
\end{itemize}

\noindent \textbf{\textit{Local rules:}} 

\begin{itemize}
\item \textbf{Localization:} 
\begin{itemize}
\item the non-blank symbols are 
located on the petals of the colored faces.
\item the symbols are transmitted through the petals, 
except on \textbf{transformation positions}.
These positions are defined 
as the intersection of a support petal 
and the transmission petal just above 
in the hierarchy.
\item the transformation positions 
are written with a couple, and the other 
petal positions with a simple symbol.
\item the non-transmission 
positions in counter $\overline{k}$ support petals 
with $k \le p-1$ are superimposed 
with a length $p-k$ word on $\{0,1\}$ or 
the symbol
.
\end{itemize}
\item \textbf{Border rule:} on the border of a gray face, all the positions 
are colored with %
.
\item \textbf{Transformation through hierarchy:} on the intersection 
of a petal support with the 
transmission petal just above in the hierarchy 
(\textbf{transformation positions}):

if the symbol is 
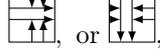
\[%
,\]
corresponding to positions symbolized
by a large red square on 
Figure~\ref{fig.trans.selec.comp.area},
then second symbol is set according to the following rules: 
\begin{enumerate}
\item if the orientation symbol is $%
$, and the first symbol 
is some length $k<p$ word $w$, 
then the second symbol is $w0$. If the first symbol 
is some length $p$ word, then the second is $0$. If the first symbol 
is %
, then the second symbol is equal.
\item if the orientation symbol is   $%
$, then the second symbol is 
%
.
\item if the orientation symbol is  
$%
$,  the second symbol is $%
$ 
if the counter is not $\overline{p-1}$, and $1$ else.
\item if the orientation symbol is 
$%
$, and the first symbol 
is some length $k$ word $w$ for $k<p$, 
then the second symbol is $w0$. If the first symbol 
is some length $p$ word $w$, then the second is $0$. If the first symbol is 
%
, then is the second.
\end{enumerate}

There are similar rules when the Robinson 
symbol
corresponds to another orientation of the 
support petal with respect to the 
support petal just above in the hierarchy.

See Figure~\ref{fig.trans.selec.comp.area} for an illustration of these rules.
\item \textbf{Coherence rules:} we impose 
coherence rules in a similar way 
as in Section~\ref{sec.functional.areas.min}.
These rules impose that the order $1$ petals 
inside a two-dimensional cell on a gray face 
are superimposed with %
 if the value 
of the $p$-counter of the cell 
border is not $\overline{0}$ and 
is a length $p$ word on $\{0,1\}$ if it is 
$\overline{0}$.

We don't describe these rules, in order 
to keep the exposition as simple as possible, 
since they are similar as in 
Section~\ref{sec.functional.areas.min}.
\end{itemize}

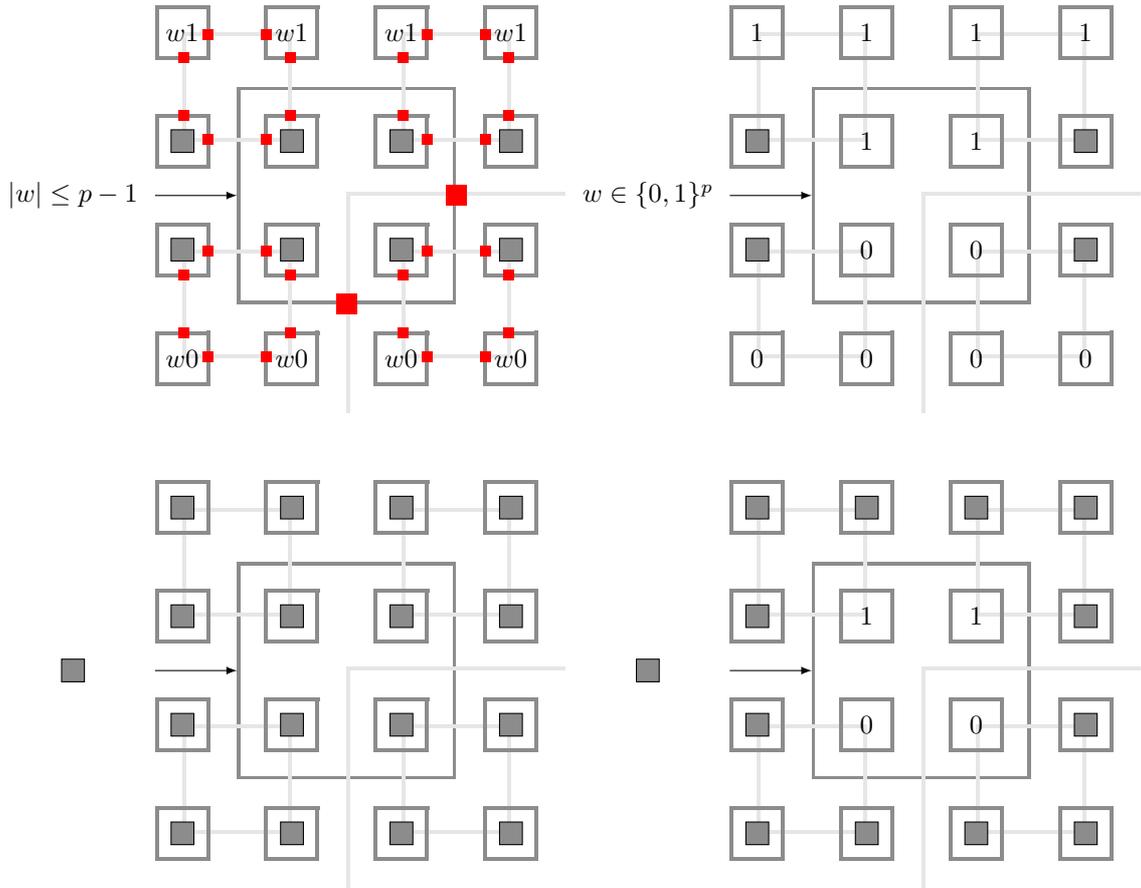
\begin{figure}[ht]
\[%
$};
\node at (40,56) {$w1$};
\node at (56,56) {$w1$};

\fill[red] (11,7.5) rectangle (12.5,9);
\fill[red] (7.5,11) rectangle (9,12.5);

\fill[red] (19.5,7.5) 
rectangle (21,9);
\fill[red] (23,11) 
rectangle (24.5,12.5);

\fill[red] (43,7.5) rectangle (44.5,9);
\fill[red] (39.5,11) rectangle (41,12.5);

\fill[red] (51.5,7.5) 
rectangle (53,9);
\fill[red] (55,11) 
rectangle (56.5,12.5);

\fill[red] (11,39.5) rectangle (12.5,41);
\fill[red] (7.5,43) rectangle (9,44.5);

\fill[red] (19.5,39.5) 
rectangle (21,41);
\fill[red] (23,43) 
rectangle (24.5,44.5);

\fill[red] (43,39.5) rectangle (44.5,41);
\fill[red] (39.5,43) rectangle (41,44.5);

\fill[red] (51.5,39.5) 
rectangle (53,41);
\fill[red] (55,43) 
rectangle (56.5,44.5);

%%%%%%%%%%%%%%%%%%%%%%%

\fill[red] (11,23) rectangle (12.5,24.5);
\fill[red] (7.5,19.5) rectangle (9,21);

\fill[red] (19.5,23) rectangle (21,24.5);
\fill[red] (23,19.5) rectangle (24.5,21);

\fill[red] (43,23) rectangle (44.5,24.5);
\fill[red] (39.5,19.5) rectangle (41,21);

\fill[red] (51.5,23) rectangle (53,24.5);
\fill[red] (55,19.5) rectangle (56.5,21);

\fill[red] (11,55) rectangle (12.5,56.5);
\fill[red] (7.5,51.5) rectangle (9,53);

\fill[red] (19.5,55) rectangle (21,56.5);
\fill[red] (23,51.5) rectangle (24.5,53);

\fill[red] (43,55) rectangle (44.5,56.5);
\fill[red] (39.5,51.5) rectangle (41,53);

\fill[red] (51.5,55) rectangle (53,56.5);
\fill[red] (55,51.5) rectangle (56.5,53);

\fill[red] (30.5,14.5) rectangle 
(33.5,17.5);

\fill[red] (46.5,30.5) rectangle 
(49.5,33.5);

\end{scope}

\begin{scope}[xshift=84cm]
\fill[gray!90] (16,16) rectangle (48,48); 
\fill[white] (16.5,16.5) rectangle (47.5,47.5);
\fill[gray!20] (32,0) rectangle (32.5,32.5);
\fill[gray!20] (32,32.5) rectangle (64,32);

\fill[gray!20] (8,8) rectangle (8.5,24); 
\fill[gray!20] (8,8) rectangle (24,8.5);
\fill[gray!20] (8,23.5) rectangle (24,24);  
\fill[gray!20] (23.5,8) rectangle (24,24);

\fill[gray!20] (40,8) rectangle (40.5,24); 
\fill[gray!20] (40,8) rectangle (56,8.5);
\fill[gray!20] (40,23.5) rectangle (56,24);  
\fill[gray!20] (55.5,8) rectangle (56,24);

\fill[gray!20] (8,40) rectangle (8.5,56); 
\fill[gray!20] (8,40) rectangle (24,40.5);
\fill[gray!20] (8,55.5) rectangle (24,56);  
\fill[gray!20] (23.5,40) rectangle (24,56);

\fill[gray!20] (40,40) rectangle (40.5,56); 
\fill[gray!20] (40,40) rectangle (56,40.5);
\fill[gray!20] (40,55.5) rectangle (56,56);  
\fill[gray!20] (55.5,40) rectangle (56,56);

\fill[gray!90] (4,4) rectangle (4.5,12); 
\fill[gray!90] (4,4) rectangle (12,4.5); 
\fill[gray!90] (4.5,11.5) rectangle (12,12); 
\fill[gray!90] (11.5,4.5) rectangle (12,12); 

\fill[gray!90] (4,20) rectangle (4.5,28); 
\fill[gray!90] (4,20) rectangle (12,20.5); 
\fill[gray!90] (4.5,27.5) rectangle (12,28); 
\fill[gray!90] (11.5,20.5) rectangle (12,28); 

\fill[gray!90] (4,36) rectangle (4.5,44); 
\fill[gray!90] (4,36) rectangle (12,36.5); 
\fill[gray!90] (4.5,43.5) rectangle (12,44); 
\fill[gray!90] (11.5,36.5) rectangle (12,44);

\fill[gray!90] (4,52) rectangle (4.5,60); 
\fill[gray!90] (4,52) rectangle (12,52.5); 
\fill[gray!90] (4.5,59.5) rectangle (12,60); 
\fill[gray!90] (11.5,52.5) rectangle (12,60);

\fill[gray!90] (20,4) rectangle (20.5,12); 
\fill[gray!90] (20,4) rectangle (28,4.5); 
\fill[gray!90] (20.5,11.5) rectangle (28,12); 
\fill[gray!90] (27.5,4.5) rectangle (28,12); 

\fill[gray!90] (20,20) rectangle (20.5,28); 
\fill[gray!90] (20,20) rectangle (28,20.5); 
\fill[gray!90] (20.5,27.5) rectangle (28,28); 
\fill[gray!90] (27.5,20.5) rectangle (28,28); 

\fill[gray!90] (20,36) rectangle (20.5,44); 
\fill[gray!90] (20,36) rectangle (28,36.5); 
\fill[gray!90] (20.5,43.5) rectangle (28,44); 
\fill[gray!90] (27.5,36.5) rectangle (28,44);

\fill[gray!90] (20,52) rectangle (20.5,60); 
\fill[gray!90] (20,52) rectangle (28,52.5); 
\fill[gray!90] (20.5,59.5) rectangle (28,60); 
\fill[gray!90] (27.5,52.5) rectangle (28,60);

\fill[gray!90] (36,4) rectangle (36.5,12); 
\fill[gray!90] (36,4) rectangle (44,4.5); 
\fill[gray!90] (36.5,11.5) rectangle (44,12); 
\fill[gray!90] (43.5,4.5) rectangle (44,12); 

\fill[gray!90] (36,20) rectangle (36.5,28); 
\fill[gray!90] (36,20) rectangle (44,20.5); 
\fill[gray!90] (36.5,27.5) rectangle (44,28); 
\fill[gray!90] (43.5,20.5) rectangle (44,28); 

\fill[gray!90] (36,36) rectangle (36.5,44); 
\fill[gray!90] (36,36) rectangle (44,36.5); 
\fill[gray!90] (36.5,43.5) rectangle (44,44); 
\fill[gray!90] (43.5,36.5) rectangle (44,44);

\fill[gray!90] (36,52) rectangle (36.5,60); 
\fill[gray!90] (36,52) rectangle (44,52.5); 
\fill[gray!90] (36.5,59.5) rectangle (44,60); 
\fill[gray!90] (43.5,52.5) rectangle (44,60);

\fill[gray!90] (52,4) rectangle (52.5,12); 
\fill[gray!90] (52,4) rectangle (60,4.5); 
\fill[gray!90] (52.5,11.5) rectangle (60,12); 
\fill[gray!90] (59.5,4.5) rectangle (60,12); 

\fill[gray!90] (52,20) rectangle (52.5,28); 
\fill[gray!90] (52,20) rectangle (60,20.5); 
\fill[gray!90] (52.5,27.5) rectangle (60,28); 
\fill[gray!90] (59.5,20.5) rectangle (60,28); 

\fill[gray!90] (52,36) rectangle (52.5,44); 
\fill[gray!90] (52,36) rectangle (60,36.5); 
\fill[gray!90] (52.5,43.5) rectangle (60,44); 
\fill[gray!90] (59.5,36.5) rectangle (60,44);

\fill[gray!90] (52,52) rectangle (52.5,60); 
\fill[gray!90] (52,52) rectangle (60,52.5); 
\fill[gray!90] (52.5,59.5) rectangle (60,60); 
\fill[gray!90] (59.5,52.5) rectangle (60,60);

\node at (8,8) {$0$};
\node at (24,8) {$0$};
\node at (8,24) {$
\begin{tikzpicture}[scale=0.3]
\fill[gray!90] (0,0) rectangle (1,1);
\draw (0,0) rectangle (1,1);
\end{tikzpicture}$};
\node at (24,24) {$0$};

\node at (8,40) {$
\begin{tikzpicture}[scale=0.3]
\fill[gray!90] (0,0) rectangle (1,1);
\draw (0,0) rectangle (1,1);
\end{tikzpicture}$};
\node at (24,40) {$1$};
\node at (8,56) {$1$};
\node at (24,56) {$1$};

\draw[-latex] (4,32) -- (16,32);
\node at (-8,32) {$w \in \{0,1\}^p$};

\node at (40,8) {$0$};
\node at (56,8) {$0$};
\node at (40,24) {$0$};
\node at (56,24) {$
\begin{tikzpicture}[scale=0.3]
\fill[gray!90] (0,0) rectangle (1,1);
\draw (0,0) rectangle (1,1);
\end{tikzpicture}$};

\node at (40,40) {$1$};
\node at (56,40) {$
\begin{tikzpicture}[scale=0.3]
\fill[gray!90] (0,0) rectangle (1,1);
\draw (0,0) rectangle (1,1);
\end{tikzpicture}$};
\node at (40,56) {$1$};
\node at (56,56) {$1$};

\end{scope}

\begin{scope}[yshift=-70cm]
\fill[gray!90] (16,16) rectangle (48,48); 
\fill[white] (16.5,16.5) rectangle (47.5,47.5);
\fill[gray!20] (32,0) rectangle (32.5,32.5);
\fill[gray!20] (32,32.5) rectangle (64,32);

\fill[gray!20] (8,8) rectangle (8.5,24); 
\fill[gray!20] (8,8) rectangle (24,8.5);
\fill[gray!20] (8,23.5) rectangle (24,24);  
\fill[gray!20] (23.5,8) rectangle (24,24);

\fill[gray!20] (40,8) rectangle (40.5,24); 
\fill[gray!20] (40,8) rectangle (56,8.5);
\fill[gray!20] (40,23.5) rectangle (56,24);  
\fill[gray!20] (55.5,8) rectangle (56,24);

\fill[gray!20] (8,40) rectangle (8.5,56); 
\fill[gray!20] (8,40) rectangle (24,40.5);
\fill[gray!20] (8,55.5) rectangle (24,56);  
\fill[gray!20] (23.5,40) rectangle (24,56);

\fill[gray!20] (40,40) rectangle (40.5,56); 
\fill[gray!20] (40,40) rectangle (56,40.5);
\fill[gray!20] (40,55.5) rectangle (56,56);  
\fill[gray!20] (55.5,40) rectangle (56,56);

\fill[gray!90] (4,4) rectangle (4.5,12); 
\fill[gray!90] (4,4) rectangle (12,4.5); 
\fill[gray!90] (4.5,11.5) rectangle (12,12); 
\fill[gray!90] (11.5,4.5) rectangle (12,12); 

\fill[gray!90] (4,20) rectangle (4.5,28); 
\fill[gray!90] (4,20) rectangle (12,20.5); 
\fill[gray!90] (4.5,27.5) rectangle (12,28); 
\fill[gray!90] (11.5,20.5) rectangle (12,28); 

\fill[gray!90] (4,36) rectangle (4.5,44); 
\fill[gray!90] (4,36) rectangle (12,36.5); 
\fill[gray!90] (4.5,43.5) rectangle (12,44); 
\fill[gray!90] (11.5,36.5) rectangle (12,44);

\fill[gray!90] (4,52) rectangle (4.5,60); 
\fill[gray!90] (4,52) rectangle (12,52.5); 
\fill[gray!90] (4.5,59.5) rectangle (12,60); 
\fill[gray!90] (11.5,52.5) rectangle (12,60);

\fill[gray!90] (20,4) rectangle (20.5,12); 
\fill[gray!90] (20,4) rectangle (28,4.5); 
\fill[gray!90] (20.5,11.5) rectangle (28,12); 
\fill[gray!90] (27.5,4.5) rectangle (28,12); 

\fill[gray!90] (20,20) rectangle (20.5,28); 
\fill[gray!90] (20,20) rectangle (28,20.5); 
\fill[gray!90] (20.5,27.5) rectangle (28,28); 
\fill[gray!90] (27.5,20.5) rectangle (28,28); 

\fill[gray!90] (20,36) rectangle (20.5,44); 
\fill[gray!90] (20,36) rectangle (28,36.5); 
\fill[gray!90] (20.5,43.5) rectangle (28,44); 
\fill[gray!90] (27.5,36.5) rectangle (28,44);

\fill[gray!90] (20,52) rectangle (20.5,60); 
\fill[gray!90] (20,52) rectangle (28,52.5); 
\fill[gray!90] (20.5,59.5) rectangle (28,60); 
\fill[gray!90] (27.5,52.5) rectangle (28,60);

\fill[gray!90] (36,4) rectangle (36.5,12); 
\fill[gray!90] (36,4) rectangle (44,4.5); 
\fill[gray!90] (36.5,11.5) rectangle (44,12); 
\fill[gray!90] (43.5,4.5) rectangle (44,12); 

\fill[gray!90] (36,20) rectangle (36.5,28); 
\fill[gray!90] (36,20) rectangle (44,20.5); 
\fill[gray!90] (36.5,27.5) rectangle (44,28); 
\fill[gray!90] (43.5,20.5) rectangle (44,28); 

\fill[gray!90] (36,36) rectangle (36.5,44); 
\fill[gray!90] (36,36) rectangle (44,36.5); 
\fill[gray!90] (36.5,43.5) rectangle (44,44); 
\fill[gray!90] (43.5,36.5) rectangle (44,44);

\fill[gray!90] (36,52) rectangle (36.5,60); 
\fill[gray!90] (36,52) rectangle (44,52.5); 
\fill[gray!90] (36.5,59.5) rectangle (44,60); 
\fill[gray!90] (43.5,52.5) rectangle (44,60);

\fill[gray!90] (52,4) rectangle (52.5,12); 
\fill[gray!90] (52,4) rectangle (60,4.5); 
\fill[gray!90] (52.5,11.5) rectangle (60,12); 
\fill[gray!90] (59.5,4.5) rectangle (60,12); 

\fill[gray!90] (52,20) rectangle (52.5,28); 
\fill[gray!90] (52,20) rectangle (60,20.5); 
\fill[gray!90] (52.5,27.5) rectangle (60,28); 
\fill[gray!90] (59.5,20.5) rectangle (60,28); 

\fill[gray!90] (52,36) rectangle (52.5,44); 
\fill[gray!90] (52,36) rectangle (60,36.5); 
\fill[gray!90] (52.5,43.5) rectangle (60,44); 
\fill[gray!90] (59.5,36.5) rectangle (60,44);

\fill[gray!90] (52,52) rectangle (52.5,60); 
\fill[gray!90] (52,52) rectangle (60,52.5); 
\fill[gray!90] (52.5,59.5) rectangle (60,60); 
\fill[gray!90] (59.5,52.5) rectangle (60,60);

\node at (8,8) {$
$};

\end{scope}

\begin{scope}[xshift=84cm,yshift=-70cm]
\fill[gray!90] (16,16) rectangle (48,48); 
\fill[white] (16.5,16.5) rectangle (47.5,47.5);
\fill[gray!20] (32,0) rectangle (32.5,32.5);
\fill[gray!20] (32,32.5) rectangle (64,32);

\fill[gray!20] (8,8) rectangle (8.5,24); 
\fill[gray!20] (8,8) rectangle (24,8.5);
\fill[gray!20] (8,23.5) rectangle (24,24);  
\fill[gray!20] (23.5,8) rectangle (24,24);

\fill[gray!20] (40,8) rectangle (40.5,24); 
\fill[gray!20] (40,8) rectangle (56,8.5);
\fill[gray!20] (40,23.5) rectangle (56,24);  
\fill[gray!20] (55.5,8) rectangle (56,24);

\fill[gray!20] (8,40) rectangle (8.5,56); 
\fill[gray!20] (8,40) rectangle (24,40.5);
\fill[gray!20] (8,55.5) rectangle (24,56);  
\fill[gray!20] (23.5,40) rectangle (24,56);

\fill[gray!20] (40,40) rectangle (40.5,56); 
\fill[gray!20] (40,40) rectangle (56,40.5);
\fill[gray!20] (40,55.5) rectangle (56,56);  
\fill[gray!20] (55.5,40) rectangle (56,56);

\fill[gray!90] (4,4) rectangle (4.5,12); 
\fill[gray!90] (4,4) rectangle (12,4.5); 
\fill[gray!90] (4.5,11.5) rectangle (12,12); 
\fill[gray!90] (11.5,4.5) rectangle (12,12); 

\fill[gray!90] (4,20) rectangle (4.5,28); 
\fill[gray!90] (4,20) rectangle (12,20.5); 
\fill[gray!90] (4.5,27.5) rectangle (12,28); 
\fill[gray!90] (11.5,20.5) rectangle (12,28); 

\fill[gray!90] (4,36) rectangle (4.5,44); 
\fill[gray!90] (4,36) rectangle (12,36.5); 
\fill[gray!90] (4.5,43.5) rectangle (12,44); 
\fill[gray!90] (11.5,36.5) rectangle (12,44);

\fill[gray!90] (4,52) rectangle (4.5,60); 
\fill[gray!90] (4,52) rectangle (12,52.5); 
\fill[gray!90] (4.5,59.5) rectangle (12,60); 
\fill[gray!90] (11.5,52.5) rectangle (12,60);

\fill[gray!90] (20,4) rectangle (20.5,12); 
\fill[gray!90] (20,4) rectangle (28,4.5); 
\fill[gray!90] (20.5,11.5) rectangle (28,12); 
\fill[gray!90] (27.5,4.5) rectangle (28,12); 

\fill[gray!90] (20,20) rectangle (20.5,28); 
\fill[gray!90] (20,20) rectangle (28,20.5); 
\fill[gray!90] (20.5,27.5) rectangle (28,28); 
\fill[gray!90] (27.5,20.5) rectangle (28,28); 

\fill[gray!90] (20,36) rectangle (20.5,44); 
\fill[gray!90] (20,36) rectangle (28,36.5); 
\fill[gray!90] (20.5,43.5) rectangle (28,44); 
\fill[gray!90] (27.5,36.5) rectangle (28,44);

\fill[gray!90] (20,52) rectangle (20.5,60); 
\fill[gray!90] (20,52) rectangle (28,52.5); 
\fill[gray!90] (20.5,59.5) rectangle (28,60); 
\fill[gray!90] (27.5,52.5) rectangle (28,60);

\fill[gray!90] (36,4) rectangle (36.5,12); 
\fill[gray!90] (36,4) rectangle (44,4.5); 
\fill[gray!90] (36.5,11.5) rectangle (44,12); 
\fill[gray!90] (43.5,4.5) rectangle (44,12); 

\fill[gray!90] (36,20) rectangle (36.5,28); 
\fill[gray!90] (36,20) rectangle (44,20.5); 
\fill[gray!90] (36.5,27.5) rectangle (44,28); 
\fill[gray!90] (43.5,20.5) rectangle (44,28); 

\fill[gray!90] (36,36) rectangle (36.5,44); 
\fill[gray!90] (36,36) rectangle (44,36.5); 
\fill[gray!90] (36.5,43.5) rectangle (44,44); 
\fill[gray!90] (43.5,36.5) rectangle (44,44);

\fill[gray!90] (36,52) rectangle (36.5,60); 
\fill[gray!90] (36,52) rectangle (44,52.5); 
\fill[gray!90] (36.5,59.5) rectangle (44,60); 
\fill[gray!90] (43.5,52.5) rectangle (44,60);

\fill[gray!90] (52,4) rectangle (52.5,12); 
\fill[gray!90] (52,4) rectangle (60,4.5); 
\fill[gray!90] (52.5,11.5) rectangle (60,12); 
\fill[gray!90] (59.5,4.5) rectangle (60,12); 

\fill[gray!90] (52,20) rectangle (52.5,28); 
\fill[gray!90] (52,20) rectangle (60,20.5); 
\fill[gray!90] (52.5,27.5) rectangle (60,28); 
\fill[gray!90] (59.5,20.5) rectangle (60,28); 

\fill[gray!90] (52,36) rectangle (52.5,44); 
\fill[gray!90] (52,36) rectangle (60,36.5); 
\fill[gray!90] (52.5,43.5) rectangle (60,44); 
\fill[gray!90] (59.5,36.5) rectangle (60,44);

\fill[gray!90] (52,52) rectangle (52.5,60); 
\fill[gray!90] (52,52) rectangle (60,52.5); 
\fill[gray!90] (52.5,59.5) rectangle (60,60); 
\fill[gray!90] (59.5,52.5) rectangle (60,60);

\node at (8,8) {$
\begin{tikzpicture}[scale=0.3]
\fill[gray!90] (0,0) rectangle (1,1);
\draw (0,0) rectangle (1,1);
\end{tikzpicture}$};
\node at (24,8) {$
\begin{tikzpicture}[scale=0.3]
\fill[gray!90] (0,0) rectangle (1,1);
\draw (0,0) rectangle (1,1);
\end{tikzpicture}$};
\node at (8,24) {$
\begin{tikzpicture}[scale=0.3]
\fill[gray!90] (0,0) rectangle (1,1);
\draw (0,0) rectangle (1,1);
\end{tikzpicture}$};
\node at (24,24) {$0$};

\node at (8,40) {$
\begin{tikzpicture}[scale=0.3]
\fill[gray!90] (0,0) rectangle (1,1);
\draw (0,0) rectangle (1,1);
\end{tikzpicture}$};
\node at (24,40) {$1$};
\node at (8,56) {$
\begin{tikzpicture}[scale=0.3]
\fill[gray!90] (0,0) rectangle (1,1);
\draw (0,0) rectangle (1,1);
\end{tikzpicture}$};
\node at (24,56) {$
\begin{tikzpicture}[scale=0.3]
\fill[gray!90] (0,0) rectangle (1,1);
\draw (0,0) rectangle (1,1);
\end{tikzpicture}$};

\draw[-latex] (4,32) -- (16,32);
\node at (-8,32) {$
\begin{tikzpicture}[scale=0.3]
\fill[gray!90] (0,0) rectangle (1,1);
\draw (0,0) rectangle (1,1);
\end{tikzpicture}$};

\node at (40,8) {$
\begin{tikzpicture}[scale=0.3]
\fill[gray!90] (0,0) rectangle (1,1);
\draw (0,0) rectangle (1,1);
\end{tikzpicture}$};
\node at (56,8) {$
\begin{tikzpicture}[scale=0.3]
\fill[gray!90] (0,0) rectangle (1,1);
\draw (0,0) rectangle (1,1);
\end{tikzpicture}$};
\node at (40,24) {$0$};
\node at (56,24) {$
\begin{tikzpicture}[scale=0.3]
\fill[gray!90] (0,0) rectangle (1,1);
\draw (0,0) rectangle (1,1);
\end{tikzpicture}$};

\node at (40,40) {$1$};
\node at (56,40) {$
\begin{tikzpicture}[scale=0.3]
\fill[gray!90] (0,0) rectangle (1,1);
\draw (0,0) rectangle (1,1);
\end{tikzpicture}$};
\node at (40,56) {$
\begin{tikzpicture}[scale=0.3]
\fill[gray!90] (0,0) rectangle (1,1);
\draw (0,0) rectangle (1,1);
\end{tikzpicture}$};
\node at (56,56) {$
\begin{tikzpicture}[scale=0.3]
\fill[gray!90] (0,0) rectangle (1,1);
\draw (0,0) rectangle (1,1);
\end{tikzpicture}$};

\end{scope}

\end{tikzpicture}\]
\caption{\label{fig.trans.selec2.comp.area} Schemata 
of the transmission rules of the horizontal addressing.
The central petals on the two schemata on 
the left have $p$-counter value not 
equal to $\overline{0}$. 
The one on the schemata on the right have 
$p$-counter value equal to $\overline{0}$.}
\end{figure}

\paragraph{Horizontal addressing:} \bigskip 

We add another subsublayer with the same symbols and similar 
rules that are abstracted as on the 
Figure~\ref{fig.trans.selec2.comp.area}. \bigskip

\noindent \textbf{\textit{Global behavior:}} \bigskip

In the \textit{vertical addressing} 
subsublayer, then on each colored face, the 
petals support the propagation of a signal 
which is transformed on the intersections 
of a support petal with the transmission 
petal just above in the hierarchy. 
Except on the transformation 
positions, the petals are colored with 
a non-empty word in $\{0,1\}$ 
or the symbol 
\begin{tikzpicture}[scale=0.3]
\fill[gray!90] (0,0) rectangle (1,1);
\draw (0,0) rectangle (1,1);
\end{tikzpicture}. 
Considering an order $n$ colored face, 
and three integers $k,m$ and $j$ 
such that $mp \le j < (m+1)p \le kp \le n$, 
the support petals that: 
\begin{enumerate}
\item are the border of an order $j$ two-dimensional 
cell, 
\item are included in an order $kp$ cell, 
\item such that the columns in the cells 
intersecting the order $j$ cell 
do not intersect any order $i$ cell 
with $ j< i \le kp$,
\end{enumerate}
are marked with a length $(m+1)p-j$ 
word on $\{0,1\}$ which codes for 
the column position of this petal 
in the order $kp$ cell 
with respect to the next order $(m+1)p$ 
cell in the hierarchy. 

All the other support petals are 
colored with \begin{tikzpicture}[scale=0.3]
\fill[gray!90] (0,0) rectangle (1,1);
\draw (0,0) rectangle (1,1);
\end{tikzpicture}.

The possible length $l$ 
addresses, with $1 \le l \le p$ are 
all the words in $\{0,1\}^l$, and 
are arranged 
in the alphabetic order, $0^l$ being 
the address of the leftmost petals and 
$1^l$ the address of the rightmost ones.

See an example of addresses arrangement 
on Figure~\ref{fig.sparse}.

As a consequence, in an order $kp$ two-dimensional 
cell, with $k \ge 0$, the columns containing 
order $0$ cells and not intersecting order $i$ cells 
such that $0<i\le kp$, are virtually 
addressed with a word in $\{0,1\}^{kp}$, obtained 
as the concatenation of $k$ length $p$ addresses. \bigskip

The global behavior of the horizontal addressing 
sublayer is similar, replacing columns 
by rows. \bigskip

The aim of the following section is to select 
a subset of these columns by selecting 
a subset of the length $p$ addresses.

\subsubsection{\label{sec.active.functional.areas} Active functional areas}

In this section, we describe how to select, for all the two dimensional cells 
on the colored faces, 
a subset of the functional areas of these cells 
which is sparse enough to not contribute the entropy dimension. This subset is 
although large enough so that 
the machines have access to all the data 
it has to check 
(namely, the frequency bits, and 
grouping bits, presented later).

Let $\Delta$ be the set of words $w$ in $\{0,1\}^p$ (amongst vertical or 
horizontal $p$-addresses) such that there exists $k \in \llbracket 0,p\rrbracket$ 
such that $w= 0^k 1^{p-k}$.
For instance, when $p=3$, $\Delta=\{000,011,001,111\}$. 
\bigskip

\begin{figure}[ht]
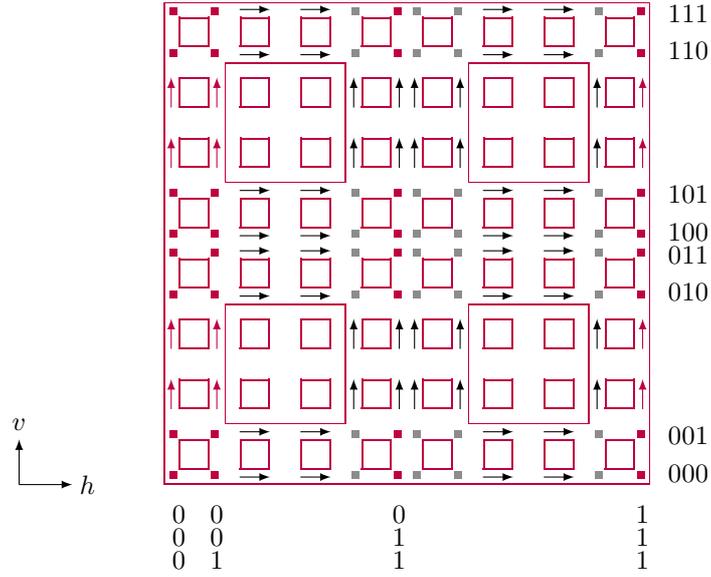

\[
\]
\caption{\label{fig.sparse} Schema of the activation signal 
in a some $(m+1)p$ order two dimensional 
cell, when $p=3$, 
for the $mp$ order cells contained in it, 
for any $m \ge 0$. The addresses 
of the active columns are 
$000$, $001$, $011$ and $111$, 
and are written under these columns.
On the right side are written 
the addresses of all the lines - not 
only the active ones.}
\end{figure}

\noindent \textbf{\textit{Symbols:}} \bigskip

The elements of
$\left\{%
\right\}^2$, and a blank symbol. \bigskip

\noindent \textbf{\textit{Local rules:}} 

\begin{itemize}
\item \textbf{Localization:} 
\begin{itemize}
\item the non-blank symbols are located on the petals 
on gray faces.
\item The \textbf{transformation positions} are 
the positions where a 
support petal intersects 
a transmission petal immediately under 
in the hierarchy. The transformation positions are marked with a couple of couples, and the other 
positions with a unique couple.
\end{itemize}
\item \textbf{Border rule:} on the border of the faces of the three-dimensional cells, 
the possible symbols are
\right)$.
\item the symbols are transmitted through petals 
and intersections of petals 
that are not transformation positions.
\item {\bf{Transformation through hierarchy}}. 
On transformation positions 
we have the following rules: 
\begin{itemize} 
\item if the counter is not $\overline{0}$, 
then the two elements of the couple are the same. 
\item if the counter is $\overline{0}$, then the couple is changed from the $n+1$ order 
petal to the $n$ order one as follows: 
\begin{enumerate}
\item In the following cases, the couples become $(%
)$: 
\begin{itemize}
\item the two addresses are in $\Delta$,
\item the Robinson symbol is %
,
and the orientation symbol is  
$%
$,  
\item the Robinson symbol is  
%
, and the orientation symbol is  
$%
$,
\item the Robinson symbol is 
%
, and the orientation symbol is 
$%
$.
\item the Robinson symbol is 
%
, and the orientation symbol is 
 $%
$
\end{itemize} 
\item if the couple of addresses is in $\Delta \times \Delta ^{c}$, 
the couple of colors is changed to
$\left(%
\right)$: the first address corresponds 
to first coordinate, and the second one to the second coordinate.
\item if the couple of addresses is in $\Delta^{c} \times \Delta$,
the couple of colors is changed to 
$\left(%
\right)$.
\item in the other cases, the couple of symbols is changed to 
$\left( %
 \right)$
\end{enumerate}
\end{itemize}
See an illustration on Figure~\ref{fig.sparse} when $p=3$, meaning $m=2$.
\item \textbf{Coherence rules:} we consider, in the following list, 
patterns that intersect three-dimensional cells having order greater or equal to $p$ (on order $\le p$ 
we can impose coherence with a finite set of rules).
\begin{enumerate}
\item when near a \textbf{corner} 
of a three-dimensional cell, on a pattern
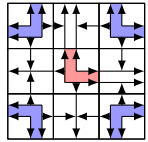
\[%
\]
in one of the copies of the Robinson subshift, the corresponding patterns in this sublayer 
has to be (respectively first and second coordinates):
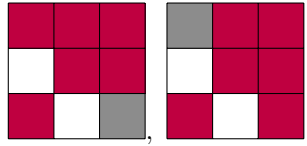
\[%
\]
with similar rules for the other corners. 
\item near a \textbf{edge}, for the 
corresponding blue corners on the 
two faces, the symbols have to be equal. 
On the middle of a edge, 
the blue corners
 are colored with 
%
 in the coordinate corresponding 
to the direction of the edge. The other 
coordinates are such that when 
the $p$-counter value on the 
edge is $\overline{0}$,  
the bottom blue corners in the case of 
a vertical edge, and 
the left ones in the case of an horizontal edge, 
are colored with purple, and the other two with gray. 
When the $p$-counter value on the edge 
is not $\overline{0}$, 
then all the blue corners are colored purple.
For instance, on the following pattern with counter $\overline{0}$ on the edge
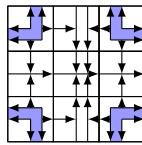
\[%
\]
in any of the copies of the Robinson subshift,
the patterns in this layer has to be \[%
.\]
When on the quarter of the edge, the blue corners
 are colored with 
%
 in the coordinate corresponding 
to the direction of the edge, and the other coordinate is 
%
. For instance, on the pattern 
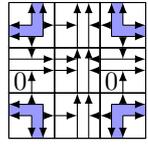
\[%
,\]
the patterns has to be: 
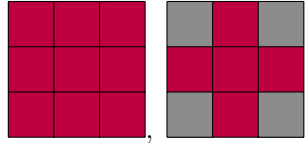
\[%
\]
\item Inside a face of a three-dimensional cell, the degenerated behaviors 
correspond to the degenerated behaviors of the $2d$ version of the Robinson 
subshift. We consider 
only two-dimensional cells having order greater or equal to $p$. 
When near a corner of a support petal 
inside the area, the symbols on the blue corners nearby have to be colored according 
to the color of the corner. On \[%
\]
in the copy of the Robinson subshift parallel to the face, 
the patterns (respectively 
first and second coordinate) in this layer have to be  
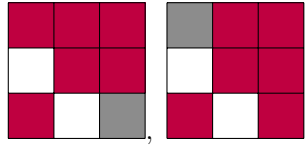
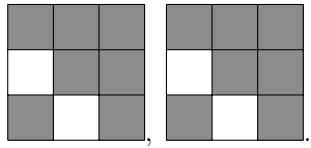
\[%
.\]
We have similar restrictions 
for the other orientations of the corner.

On the intersection positions of 
two petals, the possibilities are the same 
as on the border of the three-dimensional 
cell, except is added 
the possibility that the two patterns of the couple 
are all gray. On the same example: 
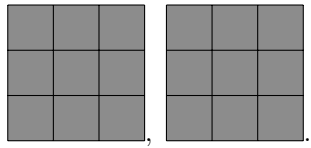
\[%
.\]
When near the border of a cell not 
intersecting another petal, 
or a line with three arrows symbols connecting corners,
the symbols corresponding to the direction of the line 
are purple on the left and gray on 
the right when it is vertical, 
and purple on the bottom and gray on the top when 
it is horizontal. 
The symbols of the other coordinates on both sides 
have to match.
\item \textbf{Coherence with the functional 
areas:} the symbol $\left(%
 \right)$ in this sublayer can be superimposed only 
on %
 in the functional areas sublayer. 
A symbol with a vertical arrow in the functional areas sublayer can be superimposed only 
with else $\left( %
\right)$ in this sublayer.
A symbol with a horizontal (resp. vertical) arrow is superimposed with else
 $\left( %
\right)$) in this sublayer.
When the symbol in the functional areas sublayer is 
%
, then in this layer it is 
(%
).
\end{enumerate}
\end{itemize}

\noindent \textbf{\textit{Global behavior:}} \bigskip

In this sublayer, on each of the 
colored faces, the order $0$ cells borders 
are colored with else %
. The ones 
colored with %
 are 
exactly those that: 
\begin{enumerate} 
\item are contained in an order $kp$, $k \ge 0$, 
two-dimensional cell, 
\item contained in columns 
of this cell that do not intersect 
any intermediate order cell, 
\item and have virtual address written 
$w_1 .. w_k$, where the words $w_i$ 
are length $p$ words on $\{0,1\}$ in the set
$$\Delta = \{0^j 1 ^{p-j} : 0 \le j \le p\}.$$
\end{enumerate}
Such columns are called \textbf{active columns} 
when the order $0$ cells border 
they intersect are colored 
%
.
\bigskip

These petals are colored 
with a second color, with similar rules, replacing 
columns by lines. The lines colored 
%
 are called \textbf{active lines}.

\begin{lemma} \label{lem.number.active}
\begin{enumerate}
\item On any order $n$ colored face, 
the number of active 
columns (resp. active lines) that are not 
in an order $jq < n$ two-dimensional cell
is $2.2^r .(p+1)^{q}$, where $n=qp+r$, with $r<p$. 
\item The number of active columns (resp. active lines) in 
an order $2k+2$ supertile (centered on an order $k$ two-dimensional cell) 
on a gray face is less than $2^{p+1} (p+1)^q$, where $k=qp+r$,
with $r<p$.
\end{enumerate}
\end{lemma}

\begin{proof} 
\begin{enumerate} 
\item \begin{enumerate}
\item The factor $2^r$ corresponds to the 
order $qp$ cells marked with 
$\begin{tikzpicture}[scale=0.3]
\fill[gray!90] (0,0) rectangle (1,1); 
\draw (0,0) rectangle (1,1);
\end{tikzpicture}$ in the addresses sublayer: 
there are $2^r$ columns (resp. lines) of them. 
\item Then the process of this layer selects $p+1$ 
addresses of order $p(q-1)$ cells relatively 
to the $pq$ cells. 
For these cells, it selects $p$ addresses 
of order $p(q-2)$ cells 
relatively to these order $p(q-1)$ cells, etc. 
Hence the factor $(p+1)^q$ in 
the formula. 
\item The factor $2$ comes from the fact 
that there are 
two active columns for all order $0$ cells.
\end{enumerate}
\item There are two possible cases for the color 
of the maximal order petal in the supertile: 
\begin{tikzpicture}[scale=0.3]
\fill[gray!90] (0,0) rectangle (1,1); 
\draw (0,0) rectangle (1,1);
\end{tikzpicture} or \begin{tikzpicture}[scale=0.3]
\fill[purple] (0,0) rectangle (1,1); 
\draw (0,0) rectangle (1,1);
\end{tikzpicture}. In the first case, there are no active columns or lines. In the second one the argument is similar to the first point.
\end{enumerate} 
\end{proof}

\begin{corollary} 
\label{cor.nb.active}
The number of active functional columns in a cell of 
order $qp$, for $q$ integer, is $2^{mq}$. 
\end{corollary}

\begin{lemma}
For all $k <n$, on the left of the first $k$ order cell in 
an order $n$ cell 
(from left to right), 
the column is active. 
Moreover, denote $p_k$ the 
number such that this column is the $p_k$th 
column amongst the active ones
from left to right.
There is an algorithm that 
computes $p_k$ given as input $k$.
\end{lemma}

\begin{proof}
\begin{enumerate}
\item \textbf{If $k$ is between $pq$ and $pq+r$,}
all the order $pq$ cells 
are selected by the process (colored purple). We pick the one which is just on the right 
of the leftmost $k$ order cell. The order $p(q-1)$ cell whose 
address is $1...1$ relatively 
to this order $pq$ cell is selected. 
We repeat this argument, considering a sequence of $p(q-k)$ order cells, the last one being the order $1$ cell 
just on the right of the $k$ order cell. The column on the right of 
the order $1$ cell is an active column 
which is just on the left 
of the $k$ order cell. 

\item \textbf{Other values of $k$:}

For the others $k$, the argument is similar.
Consider some $q'$ such that $q'p\le k < (q'+1)p$.
Pick an order $(q'+1)p$ order cell 
which is selected by the process, 
and then the order $q'p$ order cell 
whose address relatively to the 
first one represents the 
number $2^{k-(q'p+1)-2}-1$.
Then choose successively the corresponding 
order $(q'-1)p$, ... ,$0$ rightmost cells. 
The column on the right of the 
last one is on the left of the picked 
order $k$ cell.
\end{enumerate}
\end{proof}

\begin{remark}
The reason of the choice of 
$p$ this way is the 
factor $(p+1)^q$ in the 
first point of Lemma~\ref{lem.number.active}: 
we need that 
this number is a power of two, 
so that the linear counters 
have as period a Fermat number.
\end{remark}

\subsubsection{Propagation of the border information 
inside the colored faces}

The point of this section is to give 
access to all the functional positions in 
a face of a three-dimensional cell 
(specified by \begin{tikzpicture}[scale=0.3]
\fill[gray!65] (0,0) rectangle (1,1); 
\draw (0,0) rectangle (1,1);
\end{tikzpicture}) to the value of the $p$-counter 
of its border, 
and a border bit that specifies if a functional position 
is proper to the face as a two-dimensional cell. \bigskip

\noindent \textbf{\textit{Symbols:}} \bigskip

The elements of $\Z/p\Z \times \{0,1\}$ 
and a blank symbol. \bigskip

\noindent \textbf{\textit{Local rules:}} 

\begin{itemize}
\item non-blank symbols are superimposed on the faces of the 
three-dimensional cells. On these faces, 
they are superimposed
on and only on positions 
having a Robinson symbol (for 
the copy which is parallel to the face) which is 
not a two-dimensional cell 
border symbol. 
\item a symbol propagates 
to any position in the neighborhood 
positions if this position is not in 
the border of a two-dimensional cell. 
\item when on a position near the border 
of a two-dimensional cell from 
inside (the inside and 
outside are characterized by the position of the arrows: left or right for vertical 
arrows, and top or bottom for horizontal ones), 
then the symbol is equal to 
the value of the $p$-counter on this border.
The bit in $\{0,1\}$ (called \textbf{border bit}) is $1$ when the border is 
colored with \begin{tikzpicture}[scale=0.3]
\fill[gray!90] (0,0) rectangle (1,1); 
\draw (0,0) rectangle (1,1);
\end{tikzpicture} and $1$ if colored with 
\begin{tikzpicture}[scale=0.3]
\fill[gray!65] (0,0) rectangle (1,1); 
\draw (0,0) rectangle (1,1);
\end{tikzpicture} or \begin{tikzpicture}[scale=0.3]
\fill[gray!20] (0,0) rectangle (1,1); 
\draw (0,0) rectangle (1,1);
\end{tikzpicture}.
\end{itemize}

\noindent \textbf{\textit{Global behavior:}} \bigskip

All the functional positions of a face 
are colored with the value 
of the $p$-counter of its border.
The ones that are proper to this face 
as a two-dimensional cell are 
colored with the border bit $1$. The other 
ones with $0$. This is done using the diffusion 
of a signal on the face. The propagation is stopped 
by the cells walls and whose. The information 
transported by these signals is determined 
on the border of the diffusion 
areas.

\subsection{\label{sec.frqbits} Frequency 
bits layer}

The frequency bits will help to drive 
a process that, through signaling 
in the hierarchy, will generate the 
entropy dimension of the subshift. 
\bigskip

\noindent \textbf{\textit{Symbols:}} \bigskip

The elements of 
$\left(\{0,1\} \cup \{
\begin{tikzpicture}[scale=0.2]
\draw (0,0) rectangle (1,1);
\end{tikzpicture}\}\right)^3$. \bigskip

\noindent \textbf{\textit{Local rules:}}

\begin{itemize}
\item \textbf{First coordinate:} 
non blank symbols for the 
first coordinate of the couple 
are superimposed on and only on 
$np$ cell border positions 
in the copy of the Robinson 
subshift parallel to $\vec{e}^2$
 and $\vec{e}^3$.
\item \textbf{Second coordinate:} 
non blank symbols for the second coordinate of the couple 
are superimposed on and only on order $np$ cell border, $ n \ge 0$ petal positions 
in the copy of the Robinson subshift parallel to $\vec{e}^3$ and $\vec{e}^1$.
\item \textbf{Third coordinate:} non blank symbols for the third coordinate of the couple 
are superimposed on and only on order $np$ cell border positions 
in the copy of the Robinson subshift parallel to $\vec{e}^1$ and $\vec{e}^2$.
\item This means that on a position in $\Z^3$, the number of 
non-blank coordinates correspond to the color of this position in the structure layer: 
\begin{tikzpicture}[scale=0.3]
\fill[gray!90] (0,0) rectangle (1,1);
\draw (0,0) rectangle (1,1);
\end{tikzpicture} means three, 
\begin{tikzpicture}[scale=0.3]
\fill[gray!65] (0,0) rectangle (1,1);
\draw (0,0) rectangle (1,1);
\end{tikzpicture} means two, \begin{tikzpicture}[scale=0.3]
\fill[gray!20] (0,0) rectangle (1,1);
\draw (0,0) rectangle (1,1);
\end{tikzpicture} one and \begin{tikzpicture}[scale=0.3]
\draw (0,0) rectangle (1,1);
\end{tikzpicture} zero.
\item \textbf{Synchronization rule:} on positions 
colored with \begin{tikzpicture}[scale=0.3]
\fill[gray!90] (0,0) rectangle (1,1);
\draw (0,0) rectangle (1,1);
\end{tikzpicture}, the three symbols are equal.
\end{itemize}

\noindent \textbf{\textit{Global behavior:}} \bigskip

For all $n \ge 0$, the 
three dimensional cells having order $np$ are 
attached with the 
same bit called \textbf{frequency bit} and 
denoted $f_{bit} (n)$. Part of the work of 
the machine implemented in the following
will be to impose that $f_{bit} (n) = a_n$ for 
all $n$.

\subsection{\label{sec.grouping.bits} 
Grouping bits}

This layer supports bits, called 
\textbf{grouping bits}, attached to the 
two-dimensional cells of the hierarchy in each copies of the 
Robinson subshift, having order $qp$ for 
some $q$. 
They verify similar rules as the frequency bits, 
except that the Turing machines check that 
the $2^k$th (this corresponds to 
some sparse subset of the 
counter $\overline{0}$ two-dimensional cells)
bit of this sequence is $1$ for all $k$, 
and the other ones are $0$. \bigskip

The grouping bits serve to make 
group of random bits, described later, 
into values of the hierarchical counter. 

\subsection{\label{sec.lin.counter} Linear counter layer}

The construction model of M. Hochman and 
T. Meyerovitch~\cite{Hochman-Meyerovitch-2010} implies 
degenerated behaviors of the Turing machines. 
For this reason, in order to preserve the minimality property, 
we use a counter which alternates all the possible 
behaviors of these machines. We describe this 
counter here. We attribute to each face a type, 
as follows. Each type has 
specific alphabet and rules.

\begin{figure}[h!]
\[
\}$};

\node at (15,-14) {$2,2' : 
\mathcal{A} \times \mathcal{Q} \times 
\mathcal{D} \times \{\texttt{on},
\texttt{off}\}$};

\node at (15, -16) {$3,3' : \{\leftarrow,\rightarrow\} \times \{\texttt{on},
\texttt{off}\}$};

\node at (15,-18) {$5 : \mathcal{A}_c$};
\node at (15,-20) {$6,6',7,7' : 
\mathcal{Q} \times \{\texttt{on},
\texttt{off}\}$};
\end{tikzpicture}\]
\caption{\label{fig.loc.lin.counter} 
Localization of the linear counter 
area (colored yellow). In purple, 
the localization of non-blank symbols in these areas.
The demultiplexing line is represented as a thick 
dark line. The type 
of information supported 
by each face in this layer is listed under 
the picture.}
\end{figure}

\begin{itemize}
\item The top face of three-dimensional 
cells according 
to direction $\vec{e}^3$ 
(that we call \textbf{type $1$ face}) 
supports the following functions: 
\begin{itemize}
\item incrementation of 
the counter value in direction 
$\vec{e}^2$. This value is a word written on the 
intersection of active columns and the 
bottom line of the face 
(Section~\ref{sec.incrementation.linear}). 
\item Extraction of 
some information (Section~\ref{sec.demultiplexing.lines}
and Section~\ref{sec.demultiplexing}) from
this value.
\end{itemize}
\item the other faces serve 
for transporting some informations 
(Section~\ref{sec.information.transfer.lin}) 
to the machine face. This is 
the bottom face of the three-dimensional 
cells according to direction $\vec{e}^3$. 
I serve also for the synchronization 
of the counter values 
on three-dimensional cells having 
the same order and adjacent in directions 
$\vec{e}^1$ and $\vec{e}^3$. The definition 
of these face is as follows:
\begin{itemize}
\item 
The type $4$ (resp. 5) 
faces are defined as connecting 
type $1$ faces in direction $\vec{e}^2$ (resp.
$\vec{e}^1$). 
\item The type $2,3$ faces 
are defined as bottom (resp. top) faces
of three-dimensional cells according to 
direction $\vec{e}^2$.  
\item The type $6,7$ faces 
are defined as top (resp. bottom) faces
of three-dimensional cells according to 
direction $\vec{e}^1$. 
\item The type $2',3',6',7'$ faces 
are defined as connecting two 
type $2,3,6,7$ faces in direction $\vec{e}^3$.
\end{itemize}
\end{itemize}

See Figure~\ref{fig.loc.lin.counter} 
for an illustration. 

We describe the global behavior 
at the end of these sections, and 
the global behavior for the whole layer 
is written in 
Section~\ref{sec.global.behavior.linear}.

In 
Section~\ref{sec.notations.linear} 
we give some notations used in the construction 
of this layer.

\subsubsection{\label{sec.notations.linear} 
Alphabet of the linear counter}

Let $l \ge 1$ be some integer, 
and $\mathcal{A}$, $\mathcal{Q}$ and $\mathcal{D}$ some finite alphabets 
such that $|{\mathcal{A}}|= |{\mathcal{Q}}|
= 2^{2^l}$, and $|\mathcal{D}|=2^{4.2^l-2}$.  
Denote  $\mathcal{A}_c$ the alphabet
$\mathcal{A}  \times \mathcal{Q} ^3
\times \mathcal{D} \times \{\leftarrow,\rightarrow\} \times 
\{\texttt{on},\texttt{off}\}^2$. \bigskip

The alphabet $\mathcal{A}$ will correspond to the alphabet of the working 
tape of the Turing machine after completing it such that is has cardinality 
equal to $2^{2^l}$ (this is possible by adding letters that interact 
trivially with the machine heads, 
and taking $l$ great enough). 
The alphabet $\mathcal{Q}$ will correspond 
to the set of states of the machine (after similar 
completion). The arrows 
will give the direction of the propagation of 
the error signal. The elements of the set 
$\{\texttt{on},\texttt{off}\} ^2$ are a coefficients 
telling which ones of the lines and columns 
are active for computation (which 
has an influence on how each 
computation positions work), 
 and the alphabet 
$\mathcal{D}$ is an artifact so that the 
cardinality of $\mathcal{A}_c$ is $2^{2^{l+3}}$, 
in order for the counter to have a period 
equal to a Fermat number.
\bigskip

Let us fix $s$ some cyclic permutation 
of the set $\mathcal{A}_c$, and 
$\vec{c}_{\texttt{max}}$ some 
element of $\mathcal{A}_c$.

\subsubsection{\label{sec.incrementation.linear}
Incrementation}

\noindent \textbf{\textit{Symbols:}} \bigskip

The elements of 
$({\mathcal{A}_c} \times \{0,1\}) \times \left\{
\right\}$. 
The elements of $({\mathcal{A}_c} \times \{0,1\})$
are thought as the following tiles. 
The first symbol represents the south 
symbol in the tile, and the second one 
representing the west symbol in the tile: 
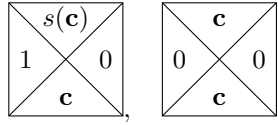
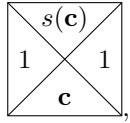
\[%
,\]
for $\vec{c} = \vec{c}_{\texttt{max}}$. \bigskip

The second is called the {\bf{freezing symbol}}.
The other symbols are 
the elements of the following sets: 
$\mathcal{A}_c \times \left\{%
\}$ \bigskip

\noindent \textbf{\textit{Local rules:}} \bigskip

\begin{itemize}
\item \textbf{Localization rules:}

\begin{itemize}
\item in this sublayer, the non-blank symbols
are superimposed on type $1$ faces such 
that the value of the $p$-counter is 
$\overline{0}$ and the border bit is $1$ (we 
recall that any position in the face is 
colored with this information).
\item The elements $\mathcal{A}_c \times \{%
\}$, appear on the active columns, except the bottom 
line (according to direction $\vec{e}^2$).
\item The elements of 
$\{%
\}$, 
appear on all the positions of 
the face except the border and 
the bottom line of the functional area.
\item The 
elements of $\mathcal{A}_c  \times \{0,1\} 
\times \{%
\}$ appear on computation positions 
of the bottom line. 
\item The elements of $\{0,1\} \times \{%
\}$ appear on the other positions of the bottom 
line. See an illustration on Figure~\ref{fig.loc.lin.count.face1}
\end{itemize}

\begin{figure}[ht]
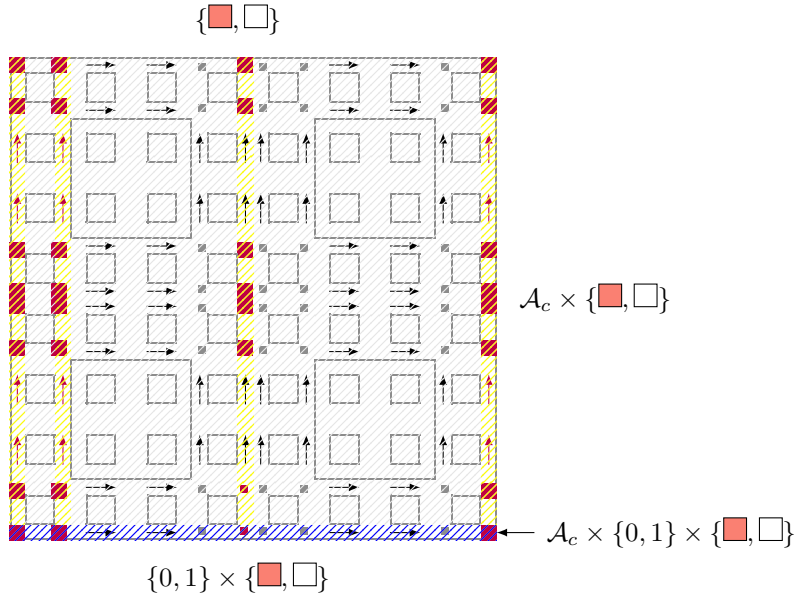

\[%
\]
\caption{\label{fig.loc.lin.count.face1} Localization 
of the linear counter symbols on face 1.}
\end{figure}

A \textbf{value} of the counter is a possible 
sequence of symbols of the 
alphabet $\mathcal{A}_c$ that can appear 
on a row, on the positions 
where this row intersects an active column.

\item \textbf{Freezing signal:}

\begin{itemize}
\item \textbf{On the bottom line:} 

\begin{enumerate}
\item on the leftmost bottommost position, the color 
is %
 if and 
only if the south symbol in 
the tile is 
$\vec{c}_{\texttt{max}}$.
\item this color propagates towards the right while the south symbol 
in the tile is $\vec{c}_{\texttt{max}}$.
When this is not true, the color becomes white. 
\item the white symbol propagates to the left.
\end{enumerate}
\item \textbf{Other positions:} 

\begin{enumerate}
\item the color part of the 
symbol propagates on the gray area 
on Figure~\ref{fig.loc.lin.count.face1}.
\item on the bottom right position of this area (on the top 
of the position on the bottom right corner of the face), 
the color is white if the color of 
the position on bottom is white. 
When this color is salmon (meaning the value of the counter is maximal), 
then, if the top right position 
of the adjacent face $4$ is 
colored salmon, then the considered position is colored white. If not, 
then it is colored salmon. See Figure~\ref{fig.freezing.signal.configurations} 
for an illustration of possible freezing symbol 
configuration on face $1$.
\end{enumerate}

\end{itemize} 

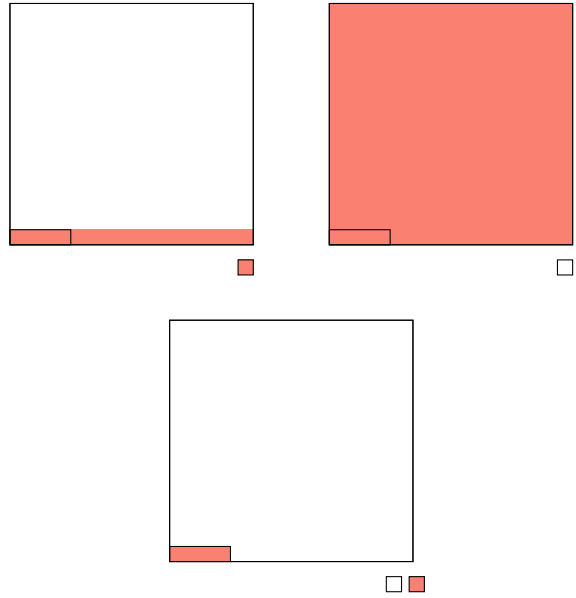
\begin{figure}[ht]
\[\begin{tikzpicture}[scale=0.05]
\begin{scope}
\fill[Salmon] (60,-8) rectangle (64,-4);
\fill[Salmon] (0,0) rectangle (64,4);
\draw (0,0) rectangle (16,4);
\draw (0,0) rectangle (64,64);
\draw (0,0) rectangle (64,64);
\draw (60,-8) rectangle (64,-4);
\end{scope}

\begin{scope}[xshift=84cm]
\fill[Salmon] (0,0) rectangle (64,64);
\draw (0,0) rectangle (16,4);
\draw (0,0) rectangle (64,64);
\draw (0,0) rectangle (64,64);
\draw (60,-8) rectangle (64,-4);
\end{scope}

\begin{scope}[xshift=42cm,yshift=-84cm]
\fill[Salmon] (63,-8) rectangle (67,-4);
\fill[Salmon] (0,0) rectangle (16,4);
\draw (0,0) rectangle (16,4);
\draw (0,0) rectangle (64,64);
\draw (0,0) rectangle (64,64);
\draw (63,-8) rectangle (67,-4);
\draw (57,-8) rectangle (61,-4);
\end{scope}
\end{tikzpicture}\]
\caption{\label{fig.freezing.signal.configurations} Possible 
freezing symbol configurations on face $1$.}
\end{figure}

\item \textbf{Incrementation of the counter:} 

On the bottom line 
of the area: 
\begin{itemize}
\item on the leftmost position of the bottom line, if the freezing signal 
of the top left position of the face $4$ under is white, 
then the east 
symbol in the tile is $1$ (meaning that 
the counter value is incremented in the line). 
Else it is $0$ (meaning this is not incremented).
\item on a computation position of this line, 
the part in $\{0,1\}$ of symbol  
of the position on the right is the east 
symbol of the tile. 
The symbol on the left position is the west 
symbol of the tile. The symbol 
on the top position 
is the north symbol of the tile, 
and the symbol 
of the position in type $4$ face under 
is equal to the south symbol of the tile. 
\item between two computation positions, 
the symbol in $\{0,1\}$ is 
transported.
\end{itemize}

\item \textbf{Transfer of state and letter:}
in the active columns, while not in 
the bottom row, 
the coefficient is transported. 
\end{itemize}

\noindent \textbf{\textit{Global behavior:}} 
\bigskip

On type $1$ faces having order $qp$ for 
some $q \ge 0$, the counter value 
is incremented on the bottom line 
using an adding machine coded with local 
rules, except 
when the freezing signal on face $4$ 
under is 
. Then the value is 
transmitted through the face $4$ above to 
the face $1$ of the adjacent three-dimensional 
cell having 
the same order in direction $\vec{e}^2$. 
As a consequence, the counter value is 
incremented cyclically in direction 
$\vec{e}^2$ each time going through 
a three-dimensional cell.
The freezing signal happens each time 
that the counter reaches 
its maximal value and stops the incrementation 
for one step, since during this step, 
the freezing signal is changed into 
%
.

\subsubsection{\label{sec.demultiplexing.lines}
Demultiplexing lines}

\noindent \textbf{\textit{Symbols:}} \bigskip

%
. \bigskip

\noindent \textbf{\textit{Local rules:}} 

\begin{itemize}
\item \textbf{Localization:} the non-blank symbols 
are superimposed on type $1$ faces.
\item when not on a position with a corner in the copy 
of the Robinson subshift parallel to $\vec{e}^1$ and $\vec{e}^2$, 
the symbol %
 propagates in directions 
$\pm(\vec{e}^1+\vec{e}^2)$.
\item when on the north east (resp.
south west) corner of the face, the position 
has symbol 
%
.
It forces the same symbol in the direction 
$-(\vec{e}^1+\vec{e}^2)$ (resp. $\vec{e}^1+\vec{e}^2$).
\item the symbol %
 can not cross the border of a 
two-dimensional cell inside 
the face, except through south west and north east corners.
\end{itemize}

\noindent \textbf{\textit{Global behavior:}}

On the top face, according to 
direction $\vec{e}^3$, 
of every three-dimensional cells,
the diagonal of this face joining 
the south west corner and the north east one
is colored with the symbol 
%
 (we recall that each face 
has specific orientation).
See an illustration on 
Figure~\ref{fig.loc.lin.counter}.

\subsubsection{\label{sec.demultiplexing}
Demultiplexing}

\noindent \textbf{\textit{Symbols:}} \bigskip

The symbols of this sublayer are 
the elements of $\mathcal{A}_c$ and 
a blank symbol.

\noindent \textbf{\textit{Local rules:}} \bigskip

\begin{itemize}
\item \textbf{Localization rule:}
The non-blank symbols are superimposed on 
the active lines of type $1$ faces where
that the value of the $p$-counter is 
$\overline{0}$ and the border bit is $1$.
\item \textbf{Transmission rule:} 
the symbols are transmitted through the columns 
of the face.
\item \textbf{Demultiplexing rule:}
on symbol \begin{tikzpicture}[scale=0.3]
\draw (0,0) rectangle (1,1);
\draw (0,0) -- (1,1);
\end{tikzpicture} in the demultiplexing lines 
sublayer, the symbol in the present sublayer 
is equal to the symbol in the incrementation 
sublayer, except on the bottom 
line. In this case, the north symbol 
of the tile is copied.
\end{itemize}

\noindent \textbf{\textit{Global 
behavior:}} \bigskip

As a consequence of the local rules, 
the value of the counter is copied 
on the diagonal of type $1$ faces having 
order $qp$ and transmitted 
on the active lines of the faces, so 
that it can be transmitted 
to type $5,6,6',7,7'$
faces.

\subsubsection{\label{sec.information.transfer.lin} 
Information transfers}

\noindent \textbf{\textit{Symbols:}} \bigskip

The symbols of this sublayer are 
the elements of the following sets: 
$\mathcal{A}_c \times 
\{\begin{tikzpicture}[scale=0.3] 
\fill[Salmon] (0,0) rectangle (1,1);
\draw (0,0) rectangle (1,1);
\end{tikzpicture}, \begin{tikzpicture}[scale=0.3] 
\draw (0,0) rectangle (1,1);
\end{tikzpicture}\}, \{\begin{tikzpicture}[scale=0.3] 
\fill[Salmon] (0,0) rectangle (1,1);
\draw (0,0) rectangle (1,1);
\end{tikzpicture},\begin{tikzpicture}[scale=0.3] 
\draw (0,0) rectangle (1,1);
\end{tikzpicture}\}, \mathcal{A}_c, 
\mathcal{A} \times 
\mathcal{Q} \times \mathcal{D} 
\times \{\texttt{on},\texttt{off}\}, 
\{\leftarrow,\rightarrow\} \times 
\{\texttt{on},\texttt{off}\}, \mathcal{Q} \times
\{\texttt{on},\texttt{off}\}
$, and 
a blank symbol.

\bigskip

\noindent \textbf{\textit{Local rules:}} 

\begin{itemize}
\item \textbf{Localization rules:}
The non-blank symbols 
are superimposed on the positions 
of faces such 
that the value of the $p$-counter is 
$\overline{0}$ and the border bit is $1$.
According to the face type, the possible 
non-blank symbols and the location on the 
face are as follows: 
\begin{itemize} 
\item type $2,2'$: elements of 
$\mathcal{A} \times \mathcal{Q} 
\times \mathcal{D} \times 
\{\texttt{on},\texttt{off}\}$, 
appearing on active lines.
\item type $3,3'$: elements of 
$\{\leftarrow,\rightarrow\} \times
\{\texttt{on},\texttt{off}\}$, 
appearing on active lines.
\item type $5$: elements of $\mathcal{A}_c$, 
appearing on active lines.
\item type $6,6',7,7'$: elements of 
$\mathcal{Q} \times 
\{\texttt{on},\texttt{off}\}$, appearing on active columns.
\item type $4$: elements of 
$\mathcal{A}_c \times \{\begin{tikzpicture}[scale=0.3] 
\fill[Salmon] (0,0) rectangle (1,1);
\draw (0,0) rectangle (1,1);
\end{tikzpicture},\begin{tikzpicture}[scale=0.3] 
\draw (0,0) rectangle (1,1);
\end{tikzpicture}\}$, appearing on active 
columns, and elements of 
$\{\begin{tikzpicture}[scale=0.3] 
\fill[Salmon] (0,0) rectangle (1,1);
\draw (0,0) rectangle (1,1);
\end{tikzpicture},\begin{tikzpicture}[scale=0.3] 
\draw (0,0) rectangle (1,1);
\end{tikzpicture}\}$, appearing on the other positions of the face.
\end{itemize}

\item \textbf{Information transfer rules:}

\begin{itemize}
\item On type $2,2',3,3',5$ faces, 
the symbols are transmitted through the rows.
\item On type $6,6',7,7'$, the symbols 
are transmitted through columns.
\item On type $4$ faces, the symbols 
in $\mathcal{A}_c \times \{\begin{tikzpicture}[scale=0.3] 
\fill[Salmon] (0,0) rectangle (1,1);
\draw (0,0) rectangle (1,1);
\end{tikzpicture},\begin{tikzpicture}[scale=0.3] 
\draw (0,0) rectangle (1,1);
\end{tikzpicture}\}$ are transmitted through 
columns, and the ones in $\{\begin{tikzpicture}[scale=0.3] 
\fill[Salmon] (0,0) rectangle (1,1);
\draw (0,0) rectangle (1,1);
\end{tikzpicture},\begin{tikzpicture}[scale=0.3] 
\draw (0,0) rectangle (1,1);
\end{tikzpicture}\}$ are transmitted 
through all the face, and the part 
of the symbols in $\mathcal{A}_c \times \{\begin{tikzpicture}[scale=0.3] 
\fill[Salmon] (0,0) rectangle (1,1);
\draw (0,0) rectangle (1,1);
\end{tikzpicture},\begin{tikzpicture}[scale=0.3] 
\draw (0,0) rectangle (1,1);
\end{tikzpicture}\}$ in the set 
$\{\begin{tikzpicture}[scale=0.3] 
\fill[Salmon] (0,0) rectangle (1,1);
\draw (0,0) rectangle (1,1);
\end{tikzpicture},\begin{tikzpicture}[scale=0.3] 
\draw (0,0) rectangle (1,1);
\end{tikzpicture}\}$ is equal to the symbols 
outside the column.
\end{itemize}

\item \textbf{Connection rules:}

\begin{itemize}
\item Across the line connecting 
face $2,3,6,7$ to 
face $2',3',6',7'$, the symbols are equal.
\item Across the line connecting 
face $4$ to face $1$, on 
the active columns, 
the symbols in $\mathcal{A}_c$ on face $4$ 
is equal to the north symbol of the 
tile on face $1$.
\item Across the line connecting
faces $2$ and $2'$ 
(resp. $3$ and $3'$,$6$ and $6'$,$7$ and 
$7'$,$5$) to face $1$, 
if the symbol on the north of the tile 
on face $1$ in the incrementation 
sublayer (resp. incrementation, 
demultiplexing, demultiplexing, 
and demultiplexing layers) is written 
$$w=(a,q^1,q^2,q^3,d,f,o^1,o^2) 
\in \mathcal{A} 
\times \mathcal{Q}^3 \times \mathcal{D}
\times \{\rightarrow,\leftarrow\}
\times 
\{\texttt{on},\texttt{off}\}^2,$$ then the corresponding 
symbol on face $2$ is equal to $(a,q^1,d,o^1)$ 
(resp. $(f,o^1)$, $(q^2,o^2)$, $(q^3,o^2)$, 
$w$).
\end{itemize}

\end{itemize}

\noindent \textbf{\textit{Global behavior:}} \bigskip

In this sublayer, the information 
of the counter value supported by face $1$ 
having order $qp$, $q \ge 0$,
after its incrementation 
in the first row of this face, is splitted. 
The parts are transmitted to faces $2,3,6,7$ where 
they will be used by the computing machines.
Face $4$ transfers information for the 
incrementation mechanism, and faces $2',3',
6',7',5$ allow the synchronization of the counter 
of adjacent three-dimensional cells having 
the same order in directions $\vec{e}^1$ 
and $\vec{e}^2$.

\subsubsection{\label{sec.global.behavior.linear}Global behavior}

On the top faces of the order $qp$ three-dimensional cells, $q  \ge 0$, 
the counter value represents 
a length $mq$ word on 
alphabet  
$\mathcal{A}_c$. It is incremented 
when going from a three-dimensional cell 
to the adjacent one in direction $\vec{e}^2$, 
except once in a cycle, when the counter 
reaches its maximal value. This time, the 
value is incremented in two steps. 
Since the number of active columns is 
$2^{mq}$ (Corollary~\ref{cor.nb.active}),
the period of the counter (meaning 
the number of time one has to jump 
from a three-dimensional cell to the adjacent 
one in direction $\vec{e}^2$ to see the 
same value) is equal to 

$$(2^{8.2^l})^{2^{mq}} +1 
= 2^{2^{l+3+mq}}+1.$$

Since $m$ and $l$ are fixed, all these 
numbers, for $q \ge 0$, are two by two different.
The important fact about them is 
that there are all co-prime 
(Lemma~\ref{lm.fermat.numbers}).

Moreover, for two adjacent cells in 
directions $\vec{e}^1$ or $\vec{e}^3$, 
their counter values are equal.

\subsection{\label{sec.machine.dim.entropique} 
Machines layer}

In this section, we present the implementation of 
Turing machines.

The support of this layer is the 
bottom face of three-dimensional cells 
having order $qp$ for some $q \ge 0$, 
according to direction $\vec{e}^3$. 

In order 
to preserve minimality, simulate each possible 
degenerated behavior of the machines, we use 
an adaptation of 
the Turing machine model 
as follows.
The bottom line 
of the face is initialized with 
symbols in $\mathcal{A} \times \mathcal{Q}$
(we allow multiple heads). The 
sides of the face are ''initialized'' 
with elements of $\mathcal{Q}$ (we allow 
machine heads to enter at each step on 
the two sides).
As usual in this type of constructions, 
the tape is not connected. 
Between two computation positions, the information 
is transported. In our model, each 
computation position takes as 
input up to four symbols coming 
from bottom and the sides, and outputs 
up to two symbols to the top and sides.
Moreover, we add special states 
to the definition of Turing machine, 
in order to manage the presence 
of multiple machine heads.
We describe this model in 
Section~\ref{subsec.machines.min}, 
and then show how to implement 
it with local rules in Section~\ref{subsec.machines.local.rules}.

In Section~\ref{sec.computation.active.lines} 
we describe signals which 
activate or deactivate lines and columns 
of the computation areas. 
These lines and columns are used by the 
machine if and only if they are active.
These signals are determined by the 
value of the linear counter.

The machine has to take into 
account only computations starting 
on well initialized tape and no machine 
head entering during 
computation. For this purpose, 
we use error signals, described in 
Section~\ref{sec.error.signals.min}.

\subsubsection{\label{sec.computation.active.lines} Computation-active 
lines and columns}

In this section we describe the 
first sublayer.

\noindent \textbf{\textit{Symbols:}} \bigskip

Elements of $\{\texttt{on},\texttt{off}\}^2$, 
of $\{\texttt{on},\texttt{off}\}$ 
and a blank symbol. \bigskip

\noindent \textbf{\textit{Local rules:}} \bigskip

\begin{itemize}
\item \textbf{Localization rules:}

\begin{itemize}
\item 
the non-blank symbols are superimposed 
on active lines and active columns positions 
on a the bottom face according to direction 
$\vec{e}^3$, with $p$-counter equal 
to $\overline{0}$ 
and border bit equal to $1$.
\item the couples are superimposed on 
intersections of an active line 
and an active column, the simple 
symbols are superimposed on the other 
positions.
\end{itemize}
\item \textbf{Transmission rule:}
the symbol is transmitted 
along lines/columns. On 
the intersections the second symbol 
is equal to the symbol on the column. 
The first one is equal 
to the symbol on the line.
\item \textbf{Connection rule:}
Across the line 
connecting type $6,7$ (resp. 
$2,3$) face and the machine 
face, and on positions where the bottom 
line intersects with active columns, 
the symbol in $\{\texttt{on},
\texttt{off}\}$ is equal to the first (resp. second)
element of the couple in 
$\{\texttt{on},
\texttt{off}\}$ in this layer.
\end{itemize}

\noindent \textbf{\textit{Global 
behavior:}} \bigskip

On the machine face of any order $qp$ 
three-dimensional cell, the active 
columns and lines are colored 
with a symbol in $\{\texttt{on},\texttt{off}\}$
which is determined by the value of the 
counter on this cell. We call 
columns (resp.lines) colored 
with $\texttt{on}$ \textbf{computation-active}
columns (resp. lines). 

\subsubsection{\label{subsec.machines.min} Adaptation 
of computing machines 
model to minimality property}

In this section we present the 
way computing machines work in our construction. 
The model we use is adapted in order 
to have the minimality property, 
and is defined as follows: 

\begin{definition}
A \textbf{computing machine} $\mathcal{M}$
is some tuple $=(\mathcal{Q}, \mathcal{A}, 
\delta, q_0, q_e,q_s, \#)$, 
where $\mathcal{Q}$ is the state set, $\mathcal{A}$ 
the alphabet, $q_0$ the initial state, and $\#$ 
is the blank symbol, 
and $$\delta : \mathcal{A} \times \mathcal{Q} 
\rightarrow
\mathcal{A} \times \mathcal{Q}   
\times \{\leftarrow,\rightarrow,\uparrow\}.$$

The other elements $q_e,q_s$ are 
states in $\mathcal{Q}$, such that 
for all $q \in \{q_e,q_s\}$, 
and for all $a$ in $\mathcal{A}$, 
$\delta(a,q) = (a,q,\uparrow)$.

\end{definition}

The special states $q_e,q_s$ in this definition 
have the following meaning: 

\begin{itemize}
\item the error state $q_e$: a machine head 
enters this state when 
it detects an error, 
or when it collides with another 
machine head.

This state is not forbidden 
in the subshift, but this is replaced 
by the sending of an error signal, and forbidding the coexistence of the error 
signal with a well initialized tape. 
The machine stops moving when 
it enters this state.

\item shadow state $q_s$: 
because of multiple heads, we need to specify 
some state which does not act on the 
tape and does not interact with 
the other heads (acting thus as a blank 
symbol). The initial tape will 
have a head in initial state on 
the leftmost position and shadow 
states on the other ones.
\end{itemize}

Any Turing machine can be transformed 
in such a machine by adding some 
state $q_s$ verifying 
the corresponding properties 
listed above. 

Moreover, we add elements to the 
alphabet which interact 
trivially with the machine states. 
This means that for any 
added letter $a$ and any state $q$, 
$\delta(a,q) = (a,q,\uparrow)$,
and then machines states which interact 
trivially with the new alphabet, 
so that the cardinality of 
the state set and the alphabet 
are $2^{2^l}$.

When the machine is well initialized, 
none of these states and letters will be 
reached. Hence the computations 
are the ones of the initial machine.
As a consequence, one can consider that the 
machine we used has these properties.

In this section, we use a machine 
which does the following operations for 
all $n$ 
\begin{itemize}
\item write $1$ on position 
$p_n$ if $n= 2^k$ for some $k$ 
and $0$ if not.
\item write $a^{(n)}_k$ 
on positions $p_k$, $k = 1 ... n$.
\end{itemize}
The sequence $a$ is the $\Pi_1$-computable 
sequence defined at the beginning 
of the construction.
The sequence $(a^{(n)}_k)$ 
is a computable sequence such that 
for all $k$, $a_n = \inf_n a^{(n)}_k$. 
For all $n$, the position $p_n$ is 
defined to be the number of the first a
active
column from left to right which 
is just on the right of an order $n$ 
two dimensional cell on a face, amongst 
active columns.

\subsubsection{\label{subsec.machines.local.rules}Implementation 
of the machines}

In this section, we describe the second 
sublayer of this layer.

\noindent \textbf{\textit{Symbols:}}

The symbols are elements of the sets 
$\mathcal{A} \times \mathcal{Q}$,
in $\mathcal{A}$, 
$\mathcal{Q}^2$, 
and a blank symbol. \bigskip

\noindent \textbf{\textit{Local rules:}} \bigskip

\begin{itemize}
\item \textbf{Localization:}
the non-blank symbols are superimposed on 
the bottom faces of the three-dimensional cells, according to direction $\vec{e}^3$.
On these faces, they are superimposed 
on positions of 
computation-active columns and rows 
with $p$-counter value equal to $\overline{0}$ 
and border bit equal to $1$.
More precisely: 
\begin{itemize}
\item the possible symbols for 
computation active columns are elements 
of the sets $\mathcal{A}$, $\mathcal{A} \times \mathcal{Q}$
and elements of 
$\mathcal{A} \times \mathcal{Q}$ are on 
the intersection with computation-active 
rows.
\item other positions are 
superimposed with an element of 
$\mathcal{Q}^2$. See an illustration on 
Figure~\ref{fig.loc.machine}.
\end{itemize}
\item along the rows and columns, 
the symbol is transmitted while not 
on intersections of computation-active 
columns and rows.

\begin{figure}[ht]
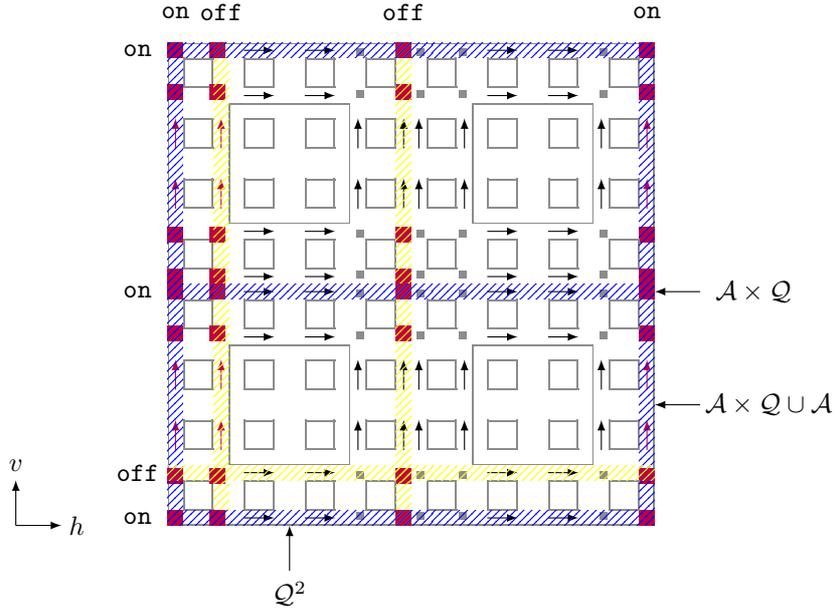

\[
\]
\caption{\label{fig.loc.machine} Localization 
of the machine symbols on the bottom faces 
of the cubes, according to the direction $\vec{e}^3$. Blue columns (resp. rows) symbolize 
computation-active columns (resp. rows).}
\end{figure}

\item 
\textbf{Connection with the counter:} 
On active computation positions that 
are in the bottom
line of this area, 
the symbols 
are equal to the corresponding 
subsymbol in $\mathcal{A} \times 
\mathcal{Q}$ on face 2 of the 
linear counter. On the leftmost 
(resp. rightmost) column of the area that 
are in a computation-active line, the symbol 
is equal to the corresponding 
symbol on face 7 (resp. 6) of the 
linear counter.

\item \textbf{Computation positions rules:}

Consider 
some computation position 
which is the intersection of a 
computation-active row and a 
computation-active column.

For such a position, 
the \textbf{inputs} include: 

\begin{enumerate}
\item the symbols written 
on the south position (or 
on the corresponding position on 
face $2$ when on the bottom line), 
\item 
the first symbol written on the west position 
(or the symbol on the corresponding 
position on face $7$ when on the west border
of the machine face), 
\item and the second symbol 
on the east position (or the 
symbol on the corresponding 
position on face $6$ when on the west border
of the machine face). 
\end{enumerate} 

The \textbf{outputs}
include: 

\begin{enumerate}
\item the symbols written on the 
north position (when not in 
the topmost row), 
\item the second symbol of the west position (when 
not in the leftmost column), 
\item and the first symbol on 
the east position (when not on the 
rightmost column).
\end{enumerate}

Moreover, on the bottom line, the inputs 
from inside the area are always the 
shadow state $q_s$.

See Figure~\ref{fig.schema.inputs.outputs} 
for an illustration.  

\begin{figure}[ht]
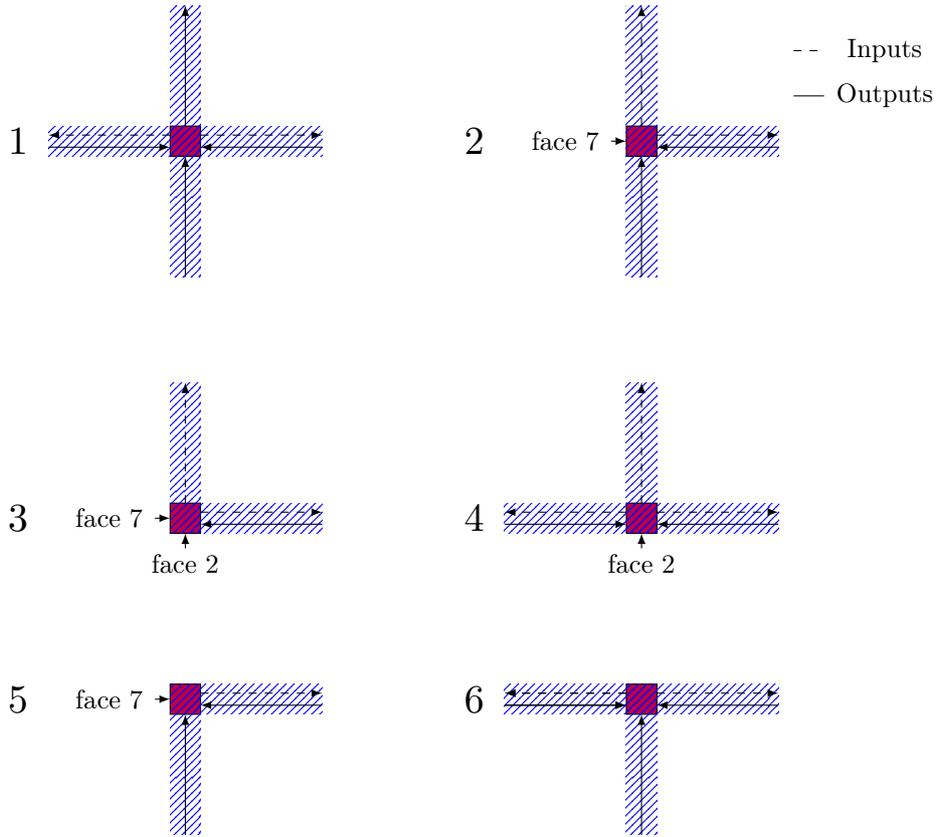

\[
\]
\caption{\label{fig.schema.inputs.outputs}
Schema of the inputs and outputs 
directions when inside the area (1) 
and on the border of the area (2,3,4,5,6).}
\end{figure}

On the first row, all the inputs 
are determined by the counter and 
by the above rule. Then, each 
computation-active row is 
determined from the adjacent one 
on the bottom and the value of the 
linear counter on faces $6$ and $7$, by 
the following rules. These rules determine, on 
each computation position, the 
outputs from the inputs: 

\begin{enumerate}
\item \textbf{Collision between machine heads:} 
if there are at least 
two elements of $\mathcal{Q} \backslash \{q_s\}$ 
in the inputs, then the 
computation position is superimposed 
with $(a,q_e)$. The output on the top 
(when this exists)
is $(a,q_e)$, where $a$ is the letter input below. 
The outputs on the sides are $q_s$.
When there is a unique symbol in 
$\mathcal{Q} \backslash \{q_s\}$ in the inputs, 
this symbol is called the machine head 
state (the symbol $q_s$ is not considered 
as representing a machine head).
\item \textbf{Standard rule:}
\begin{enumerate}
\item when the head input comes from a side, 
then the functional position is superimposed with 
$(a,q)$. The above output is the couple $(a,q)$, 
where $a$ is the letter input under, and $q$ the head input.
The other outputs are $q_s$.
See Figure~\ref{fig.standard.rule.0} for 
an illustration of this rule.

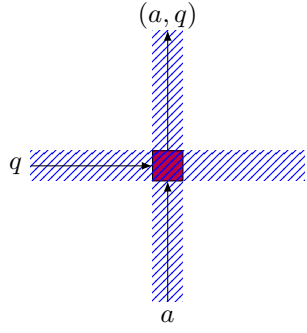
\begin{figure}[ht]
\[\begin{tikzpicture}[scale=0.2]
\fill[purple] (-1,-1) rectangle (1,1);
\draw (-1,-1) rectangle (1,1);
\fill[pattern = north east lines, pattern color = blue] 
(-9,-1) rectangle (9,1);
\fill[pattern = north east lines, pattern color = blue] 
(-1,-9) rectangle (1,9);
\draw[-latex] (-9,0) -- (-1,0);
\draw[-latex] (0,-9) -- (0,-1);
\draw[-latex] (0,1) -- (0,9);
\node at (-10,0) {$q$};
\node at (0,-10) {$a$};
\node at (0,10) {$(a,q)$};
\end{tikzpicture}\]
\caption{\label{fig.standard.rule.0} Illustration 
of the standard rules (1).}
\end{figure}

\item when the head input comes from under, 
the above output is:

\begin{itemize}
\item $\delta_1 (a,q)$ when 
$\delta_3 (a,q)$ is in 
$\{\rightarrow,\leftarrow\}$ 
\item and 
$(\delta_1(a,q),\delta_2 (a,q))$ when 
$\delta_3 (a,q) = \uparrow$. 
\end{itemize}

The head output is in the 
direction of $\delta_3 (a,q)$ (when 
this output direction exists) and equal to 
$\delta_2 (a,q)$
when it is in $\{\rightarrow,\leftarrow\}$. 
The other outputs are $q_s$.
See Figure~\ref{fig.standard.rule.1} 
for an illustration.

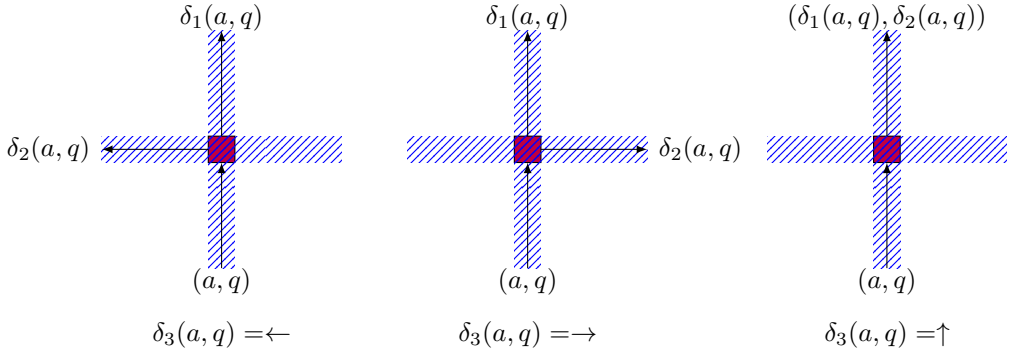
\begin{figure}[ht]
\[\begin{tikzpicture}[scale=0.175]
\begin{scope}
\fill[purple] (-1,-1) rectangle (1,1);
\draw (-1,-1) rectangle (1,1);
\fill[pattern = north east lines, pattern color = blue] 
(-9,-1) rectangle (9,1);
\fill[pattern = north east lines, pattern color = blue] 
(-1,-9) rectangle (1,9);
\draw[-latex] (-1,0) -- (-9,0);
\draw[-latex] (0,-9) -- (0,-1);
\draw[-latex] (0,1) -- (0,9);
\node at (-13,0) {$\delta_2 (a,q)$};
\node at (0,-10) {$(a,q)$};
\node at (0,10) {$\delta_1 (a,q)$};
\node at (0,-14) {$\delta_3 (a,q) = \leftarrow$};
\end{scope}

\begin{scope}[xshift=23cm]
\fill[purple] (-1,-1) rectangle (1,1);
\draw (-1,-1) rectangle (1,1);
\fill[pattern = north east lines, pattern color = blue] 
(-9,-1) rectangle (9,1);
\fill[pattern = north east lines, pattern color = blue] 
(-1,-9) rectangle (1,9);
\draw[-latex] (1,0) -- (9,0);
\draw[-latex] (0,-9) -- (0,-1);
\draw[-latex] (0,1) -- (0,9);
\node at (13,0) {$\delta_2 (a,q)$};
\node at (0,-10) {$(a,q)$};
\node at (0,10) {$\delta_1 (a,q)$};
\node at (0,-14) {$\delta_3 (a,q) = \rightarrow$};
\end{scope}

\begin{scope}[xshift=50cm]
\fill[purple] (-1,-1) rectangle (1,1);
\draw (-1,-1) rectangle (1,1);
\fill[pattern = north east lines, pattern color = blue] 
(-9,-1) rectangle (9,1);
\fill[pattern = north east lines, pattern color = blue] 
(-1,-9) rectangle (1,9);
\draw[-latex] (0,-9) -- (0,-1);
\draw[-latex] (0,1) -- (0,9);
\node at (0,10) {$(\delta_1(a,q),\delta_2 (a,q))$};
\node at (0,-10) {$(a,q)$};
\node at (0,-14) {$\delta_3 (a,q) = \uparrow$};
\end{scope}
\end{tikzpicture}\]
\caption{\label{fig.standard.rule.1}
Illustration of the standard rules (2).}
\end{figure}
\end{enumerate}
\item \textbf{Collision with border:} 
When the output 
direction does not exist, the output 
is $(a,q_e)$ on the top, and the outputs 
on the side is $q_s$.
The computation position is 
superimposed with $(a,q)$.  
\item \textbf{No machine head:} when all 
the inputs in $\mathcal{Q}$ are 
$q_s$, and the above output is 
in $\mathcal{A}$ and equal to its input $a$.
\end{enumerate} 

\end{itemize}

\noindent {\textbf{\textit{Global behavior:}} \bigskip

On the bottom faces according to $\vec{e}^3$ 
of order $qp$ three-dimensional cells, we implemented some computations using our
modified Turing machine model. This model 
allows multiple 
machine heads on the initial tape and 
entering in each row. When there 
is a unique machine head on the leftmost 
position of the bottom line and only blank 
letters on the initial tape, and 
all the lines and columns are computation-active, then the computations 
are as intended. This means that a the 
machine write successively the 
bits $a^{(n)}_k$ 
on the $p_k$th column of its tape (in 
order to impose the value of the frequency bits),
 and moreover 
writes $1$ if $k$ is a power of two, and $0$ 
if not, in order to impose the value of 
the grouping bits. It enters in the error 
state $q_e$ when it detects an error.

When this is not the case, the computations 
are determined by the rules 
giving the outputs on computation positions 
from the inputs. When there is a collision 
of a machine head with the border, 
it enters in state $q_e$. When 
heads collide, they fusion into a unique 
head in state $q_e$.
In Section~\ref{sec.error.signals.min}, 
we describe signal errors that helps us
to take into account the computations 
only when the initial tape is 
empty, all the lines and columns are computation-active, and there is no machine 
head entering 
on the sides of the area.

\subsection{\label{sec.error.signals.min} 
Error signals}

In order to simulate any behavior that happens in infinite areas in 
finite ones, we need error signals. This means that when 
the machine detects an error (enters a halting state), it sends a signal 
to the initialization line to verify it was well initialized: that 
the tape was empty, that 
no machine head enters on the sides, 
and that the machine was initially in the leftmost position 
of the line, and in initialization state. Moreover, for the reason 
that we need to compute precisely the number of possible 
initial tape contents, we allow initialization of multiple heads. 
The first error signal will detect the first position from left 
to right in the top row of the area 
where there is a machine head in error state 
or the active column is $\texttt{off}$.
This position only will trigger an 
error signal (described 
in Section~\ref{subsec.error.signal}, according 
to the direction specified just above when it in the top line of the area 
(the word of arrows specifying the direction is a part of the counter). The empty 
tape signal detects if the 
initial tape was empty, and that there 
was a unique machine head on the leftmost 
position in initialization state $q_0$. 
The empty tape and first error
signals are described in Section~\ref{subsec.empty.tape.first.machine}. 
The empty sides signal, described in 
Section~\ref{subsec.empty.sides} detects if 
there is no machine head entering 
on the sides, and that the $\texttt{on}/
\texttt{off}$ signals on the sides 
are all equal to $\texttt{on}$.
The error signal is taken into account 
(meaning forbidden) when the empty tape, 
and empty sides signals are detecting 
an error.

\subsubsection{ \label{subsec.empty.tape.first.machine}
Empty tape, first error signals}

\noindent \textbf{\textit{Symbols:}} \bigskip

The first sublayer has the 
following symbols: 

$\begin{tikzpicture}[scale=0.3]
\fill[Salmon] (0,0) rectangle (1,1);
\draw (0,0) rectangle (1,1);
\end{tikzpicture}, \begin{tikzpicture}[scale=0.3]
\fill[YellowGreen] (0,0) rectangle (1,1);
\draw (0,0) rectangle (1,1);
\end{tikzpicture}$, symbols in 
$\left\{\begin{tikzpicture}[scale=0.3]
\fill[Salmon] (0,0) rectangle (1,1);
\draw (0,0) rectangle (1,1);
\end{tikzpicture}, \begin{tikzpicture}[scale=0.3]
\fill[YellowGreen] (0,0) rectangle (1,1);
\draw (0,0) rectangle (1,1);
\end{tikzpicture}\right\} ^2 $, and 
a blank symbol $
\begin{tikzpicture}[scale=0.3]
\draw (0,0) rectangle (1,1);
\end{tikzpicture}$. \bigskip

\noindent \textbf{\textit{Local rules:}}  

\begin{itemize}
\item \textbf{Localization:} 
non blank symbols are superimposed on the 
top line and bottom line of the
border of the machine face as a two-dimensional cell.
\item \textbf{First error signal:} this signal 
detects the first error on the top 
of the functional area, from 
the left to the right, 
where an error means a symbol 
$\texttt{off}$ or $q_e$. The rules are: 
\begin{itemize}
\item the topmost leftmost position of the top 
line of the cell is 
marked with $\begin{tikzpicture}[scale=0.3]
\fill[YellowGreen] (0,0) rectangle (1,1);
\draw (0,0) rectangle (1,1);
\end{tikzpicture}$. 
\item the symbol $\begin{tikzpicture}[scale=0.3]
\fill[YellowGreen] (0,0) rectangle (1,1);
\draw (0,0) rectangle (1,1);
\end{tikzpicture}$ propagates the the left, 
and propagates to the right
while the position under is not in error 
state $q_e$ and the symbol in $\{\texttt{on},
\texttt{off}\}$ is $\texttt{on}$.
\item when on position in the top row with an 
error, 
the position on the top right is colored 
$\begin{tikzpicture}[scale=0.3]
\fill[Salmon] (0,0) rectangle (1,1);
\draw (0,0) rectangle (1,1);
\end{tikzpicture}$.
\item the symbol $\begin{tikzpicture}[scale=0.3]
\fill[Salmon] (0,0) rectangle (1,1);
\draw (0,0) rectangle (1,1);
\end{tikzpicture}$ propagates to the right, 
and propagates to the left while 
the positions under is not in error state.
\end{itemize} 
\textbf{Empty tape signal:} this signal
detects if the initial tape of the machine is empty. 
This means that it is 
filled with the symbol $(\#,q_s)$ except 
on the leftmost position 
where it has to be $(\#,q_0)$. The signal detects 
the first symbol which is different from $(\#,q_s)$ or $(\#,q_0)$ 
when on the left, 
from left to right (first color), and from left to right 
(second color).
Concerning the first color: 
\begin{itemize}
\item on the bottom row, the leftmost position is colored with 
$
$
propagates to the right unless when 
on a position under a symbol different from:
\begin{itemize}
\item 
$(\#,q_0)$ when on the leftmost functional position, 
\item $(\#,q_s)$ on another functional position.
\end{itemize}
\item When on these positions, the symbol on the right is 
$%
$ 
propagates to the right.
\end{itemize}
for the second one: 
\begin{itemize}
\item on the bottom row, the rightmost position is colored with 
$%
$
propagates to the left except when
on a position under a symbol different from:
\begin{itemize}
\item  
$(\#,q_0)$ when on the leftmost functional position, 
\item $(\#,q_s)$ on another computation 
position.
\end{itemize}
\item When on these positions, the symbol on the right is 
$%
$ 
propagates to the left.
\end{itemize}
\end{itemize}

\noindent \textbf{\textit{Global behavior:}} \bigskip

The top row is separated in 
two parts: before and after (from left to right) 
the first error. 
The left part is colored $%
$, and the right part 
$%
$. The bottom row is colored 
with a couple of color. The first one 
separates the row in two parts. The limit 
between the two parts is the first occurrence 
\textit{from left to right} 
of a symbol different from $(\#,q_s)$ or $(\#,q_0)$ when on 
the leftmost computation position above. The 
second color of the couple separates 
two similar parts \textit{from 
right to left}.

See Figure~\ref{fig.error.signal.dim.entropique}
for an illustration.

\subsubsection{\label{subsec.empty.sides}
Empty sides signals}

This second sublayer has the same 
symbols as the first sublayer.
The principle of the local rules 
is similar: the leftmost (resp. rightmost) 
column is splitted 
in two parts, the top one colored 
$\begin{tikzpicture}[scale=0.3]
\fill[YellowGreen] (0,0) rectangle (1,1);
\draw (0,0) rectangle (1,1);
\end{tikzpicture}$ and 
the bottom one colored 
$\begin{tikzpicture}[scale=0.3]
\fill[Salmon] (0,0) rectangle (1,1);
\draw (0,0) rectangle (1,1);
\end{tikzpicture}$.
The limit is the first position 
from top to bottom where 
the corresponding symbol across 
the limit with face $7$ (resp. face $6$) 
is $(q_s,\texttt{on})$ (resp. $q_s$).
Moreover, the bottom row
of the machine face is colored with 
a couple of colors. This couple 
is constant over the row. 
The first one of the colors is equal to the color 
at the bottom of the leftmost column. 
The second one is the color at the bottom 
of the rightmost one. 

See an illustration on Figure~\ref{fig.error.signal.dim.entropique}.

\subsubsection{\label{subsec.error.signal} 
Error signals}

\noindent \textbf{\textit{Symbols:}} \bigskip

$
$. \bigskip

\noindent \textbf{\textit{Local rules:}} 

\begin{itemize}
\item  \textbf{Localization:} 
the non-blank symbols are superimposed 
on the right, left and top sides of the border of 
machine faces in three-dimensional cells.
\item \textbf{Propagation:} each of the two symbols 
propagates when inside one of these two areas: 
\begin{itemize}
\item the union of the left side of the face 
and positions colored $%
$ in the top side of the face.
\item the union of the right side of the face 
and positions colored $%
$ in the top side of the face.
\end{itemize}
\item \textbf{Induction:} 
\begin{itemize} 
\item on a position of the top side of the face 
which is colored $%
$ and the position on the right is 
$%
$, if the symbol 
above in the information transfers layer is $\rightarrow$ (resp. 
$\leftarrow$), then 
there is an error signal $%
$ on this position and none on the right (resp. 
there is no error signal on this position and there is one on the right).
\item on the rightmost topmost position of the face, if the 
first machine signal is $%
$, then 
there is no error signal.
\end{itemize} 
\item \textbf{Forbidding wrong configurations:} 

there can not be four 
symbols $%
$ and an error signal $%
$ on the same position.
\end{itemize}

\noindent \textbf{\textit{Global behavior:}} \bigskip

When there is a machine head in 
error state in the top row, the first one 
(from left to right) 
sends an error signal to the bottom row (see 
Figure~\ref{fig.error.signal.dim.entropique}) 
in the direction indicated by the arrow 
on the corresponding position on face $3$.
This signal is forbidden 
if the machine is well initialized. This means 
that the 
working tape of the machine is 
empty in the bottom row. Moreover, there is 
a unique machine in state $q_0$ in the leftmost position of the bottom row, 
all the lines and columns are $\texttt{on}$ 
and there is no machine entering on the sides. 
This means that the error signal is taken 
into account only when the computations 
have the intended behavior.

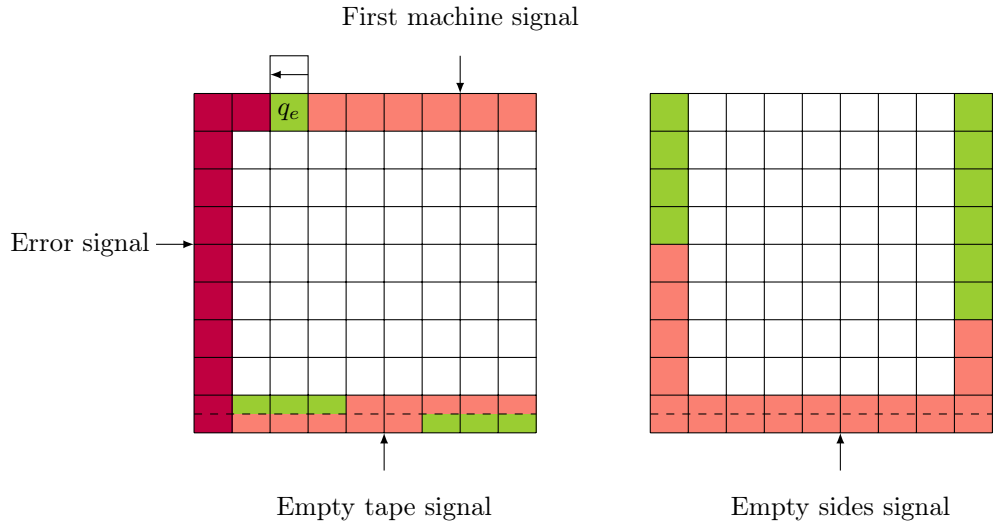
\begin{figure}[ht]
\[\begin{tikzpicture}[scale=0.5]
\begin{scope}
\fill[YellowGreen] (0,8) rectangle (3,9);
\fill[Salmon] (3,8) rectangle (9,9);
\fill[Salmon] (0,0) rectangle (6,0.5);
\fill[YellowGreen] (6,0) rectangle (9,0.5);
\fill[YellowGreen] (0,0.5) rectangle (4,1);
\fill[Salmon] (4,0.5) rectangle (9,1);
\fill[purple] (0,0) rectangle (1,9);
\fill[purple] (0,8) rectangle (2,9);
\draw (2,9) rectangle (3,10);
\draw[-latex] (3,9.5) -- (2,9.5);
\node at (2.5,8.5) {$q_e$};
\draw (0,0) grid (9,9); 
\draw[dashed] (0,0.5) -- (9,0.5);
\draw[-latex] (-1,5) -- (0,5);
\node at (-3,5) {Error signal};
\draw[-latex] (5,-1) -- (5,0);
\node at (5,-2) {Empty tape signal};
\draw[-latex] (7,10) -- (7,9) ;
\node at (7,11) {First machine signal};
\end{scope}

\begin{scope}[xshift=12cm]
\fill[Salmon] (0,0) rectangle (9,1);
\fill[Salmon] (0,0) rectangle (1,5);
\fill[YellowGreen] (0,5) rectangle (1,9);

\fill[Salmon] (8,0) rectangle (9,3);
\fill[YellowGreen] (8,3) rectangle (9,9);
\draw (0,0) grid (9,9); 

\draw[dashed] (0,0.5) -- (9,0.5);

\draw[-latex] (5,-1) -- (5,0);
\node at (5,-2) {Empty sides signal};
\end{scope}
\end{tikzpicture}\]
\caption{\label{fig.error.signal.dim.entropique} Illustration 
of the propagation of an error signal, 
where are represented the empty tape, 
first machine and empty sides signals.}
\end{figure}

Because for any $n$ and any configuration, there 
exists some three-dimensional cell in which 
the machine is well initialized 
(because of the presence of counters). 
For any $k$, there 
exists some $n$ such that 
the machine has enough time to 
check the $k$th frequency bit and grouping 
bit.
This means that in any configuration 
of the subshift, 
the $k$th frequency bit is equal to $a_k$, 
and the $2^k$th grouping bit is $1$, 
and the other ones are $0$.

\subsection{\label{sec.hier.counter} Hierarchical counter layer}

This layer has two sublayers.

\subsubsection{Grouping border bits}

This sublayer supports bits on the top faces of 
the three-dimensional cells, according to $\vec{e}^1$, on positions that 
are not in the border of a two-dimensional cell with grouping bit
bit $1$. The symbol are $0,1$ 
and propagate to neighbors while these neighbors have not grouping bit equal to $1$.
When near the border of a cell with grouping 
bit equal to $1$, the bit is $1$ if it is the border of the face 
of the three-dimensional cell. 
Else, this bit is $0$. 

This means that the only functional positions which have bit $1$ 
are the functional position on a face of an order $2^k p$ 
three-dimensional cell which are not in an order $2^j p$ two-dimensional 
cell with $j<k$. \bigskip

We use a similar mechanism as the 
one presented in Section~\ref{sec.orientation} so 
that \textbf{the order $n$ two dimensional cells have access to the 
grouping bit and the frequency bit of the two-dimensional order $n+1$
cell just above in the hierarchy, as well 
as its $p$-addresses} (for 
the addresses there is not border rule or 
$\epsilon$ symbol).

\subsubsection{Hierarchy bits}

\noindent \textbf{\textit{Symbols:}} \bigskip

\begin{tikzpicture}[scale=0.3]
\fill[purple] (0,0) rectangle (1,1);
\draw (0,0) rectangle (1,1);
\end{tikzpicture}, \begin{tikzpicture}[scale=0.3]
\fill[gray!90] (0,0) rectangle (1,1);
\draw (0,0) rectangle (1,1);
\end{tikzpicture}, elements of $\left\{\begin{tikzpicture}[scale=0.3]
\fill[purple] (0,0) rectangle (1,1);
\draw (0,0) rectangle (1,1);
\end{tikzpicture}, \begin{tikzpicture}[scale=0.3]
\fill[gray!90] (0,0) rectangle (1,1);
\draw (0,0) rectangle (1,1);
\end{tikzpicture}\right\}^2$ and a blank symbol.
\bigskip

\noindent \textbf{\textit{Local rules:}} 

\begin{itemize}
\item \textbf{Localization:} the non-blank symbols 
of this layer are localized on the petal positions 
of the copy of the Robinson subshift which is parallel to $\vec{e}^2$ and $\vec{e}^3$.
\item the couples are superimposed on and only on transformation positions.
\item the symbol is transmitted through arrows 
and intersection of order $n$ and $n+1$ petals such that 
the order $n$ one has counter $0$ or the value of the $p$-counter is not $\overline{p-1}$.
The other intersection positions are \textbf{transformation positions}.
\item \textbf{Transformation:} 
on a transformation position, we have 
the following rules: 
\begin{itemize}
\item if the grouping bit of the above cell 
in the hierarchy is $1$, then we have the following: 
\begin{itemize}
\item in the following cases, the second color is 
: 
\begin{enumerate}
\item the Robinson symbol is %
,
and the orientation symbol is  
$%
$,  
\item the Robinson symbol is  
%
, and the orientation symbol is  
$%
$,
\item the Robinson symbol is 
%
, and the orientation symbol is 
$%
$.
\item the Robinson symbol is 
%
, and the orientation symbol is 
$%
$.
\end{enumerate}
\item in the following cases, the second color is 
equal to the first one: 
\begin{enumerate}
\item the Robinson symbol is %
,
and the orientation symbol is  
$%
$,  
\item the Robinson symbol is  
%
, and the orientation symbol is  
 $%
$,
\item the Robinson symbol is 
%
, and the orientation symbol is 
$%
$.
\item the Robinson symbol is 
%
, and the orientation symbol is 
$%
$.
\end{enumerate}
\item in the other cases, the second color is 
%
.
\end{itemize} 
See an illustration of Figure~\ref{fig.rule.hierarchybit.1}.

\begin{figure}[ht]
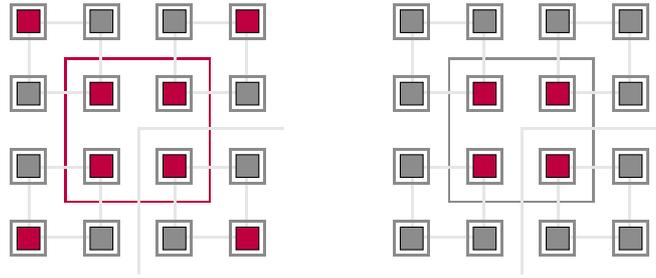

\[%
$};
\end{scope}

\begin{scope}[xshift=84cm]
\fill[gray!90] (16,16) rectangle (48,48); 
\fill[white] (16.5,16.5) rectangle (47.5,47.5);
\fill[gray!20] (32,0) rectangle (32.5,32.5);
\fill[gray!20] (32,32.5) rectangle (64,32);

\fill[gray!20] (8,8) rectangle (8.5,24); 
\fill[gray!20] (8,8) rectangle (24,8.5);
\fill[gray!20] (8,23.5) rectangle (24,24);  
\fill[gray!20] (23.5,8) rectangle (24,24);

\fill[gray!20] (40,8) rectangle (40.5,24); 
\fill[gray!20] (40,8) rectangle (56,8.5);
\fill[gray!20] (40,23.5) rectangle (56,24);  
\fill[gray!20] (55.5,8) rectangle (56,24);

\fill[gray!20] (8,40) rectangle (8.5,56); 
\fill[gray!20] (8,40) rectangle (24,40.5);
\fill[gray!20] (8,55.5) rectangle (24,56);  
\fill[gray!20] (23.5,40) rectangle (24,56);

\fill[gray!20] (40,40) rectangle (40.5,56); 
\fill[gray!20] (40,40) rectangle (56,40.5);
\fill[gray!20] (40,55.5) rectangle (56,56);  
\fill[gray!20] (55.5,40) rectangle (56,56);

\fill[gray!90] (4,4) rectangle (4.5,12); 
\fill[gray!90] (4,4) rectangle (12,4.5); 
\fill[gray!90] (4.5,11.5) rectangle (12,12); 
\fill[gray!90] (11.5,4.5) rectangle (12,12); 

\fill[gray!90] (4,20) rectangle (4.5,28); 
\fill[gray!90] (4,20) rectangle (12,20.5); 
\fill[gray!90] (4.5,27.5) rectangle (12,28); 
\fill[gray!90] (11.5,20.5) rectangle (12,28); 

\fill[gray!90] (4,36) rectangle (4.5,44); 
\fill[gray!90] (4,36) rectangle (12,36.5); 
\fill[gray!90] (4.5,43.5) rectangle (12,44); 
\fill[gray!90] (11.5,36.5) rectangle (12,44);

\fill[gray!90] (4,52) rectangle (4.5,60); 
\fill[gray!90] (4,52) rectangle (12,52.5); 
\fill[gray!90] (4.5,59.5) rectangle (12,60); 
\fill[gray!90] (11.5,52.5) rectangle (12,60);

\fill[gray!90] (20,4) rectangle (20.5,12); 
\fill[gray!90] (20,4) rectangle (28,4.5); 
\fill[gray!90] (20.5,11.5) rectangle (28,12); 
\fill[gray!90] (27.5,4.5) rectangle (28,12); 

\fill[gray!90] (20,20) rectangle (20.5,28); 
\fill[gray!90] (20,20) rectangle (28,20.5); 
\fill[gray!90] (20.5,27.5) rectangle (28,28); 
\fill[gray!90] (27.5,20.5) rectangle (28,28); 

\fill[gray!90] (20,36) rectangle (20.5,44); 
\fill[gray!90] (20,36) rectangle (28,36.5); 
\fill[gray!90] (20.5,43.5) rectangle (28,44); 
\fill[gray!90] (27.5,36.5) rectangle (28,44);

\fill[gray!90] (20,52) rectangle (20.5,60); 
\fill[gray!90] (20,52) rectangle (28,52.5); 
\fill[gray!90] (20.5,59.5) rectangle (28,60); 
\fill[gray!90] (27.5,52.5) rectangle (28,60);

\fill[gray!90] (36,4) rectangle (36.5,12); 
\fill[gray!90] (36,4) rectangle (44,4.5); 
\fill[gray!90] (36.5,11.5) rectangle (44,12); 
\fill[gray!90] (43.5,4.5) rectangle (44,12); 

\fill[gray!90] (36,20) rectangle (36.5,28); 
\fill[gray!90] (36,20) rectangle (44,20.5); 
\fill[gray!90] (36.5,27.5) rectangle (44,28); 
\fill[gray!90] (43.5,20.5) rectangle (44,28); 

\fill[gray!90] (36,36) rectangle (36.5,44); 
\fill[gray!90] (36,36) rectangle (44,36.5); 
\fill[gray!90] (36.5,43.5) rectangle (44,44); 
\fill[gray!90] (43.5,36.5) rectangle (44,44);

\fill[gray!90] (36,52) rectangle (36.5,60); 
\fill[gray!90] (36,52) rectangle (44,52.5); 
\fill[gray!90] (36.5,59.5) rectangle (44,60); 
\fill[gray!90] (43.5,52.5) rectangle (44,60);

\fill[gray!90] (52,4) rectangle (52.5,12); 
\fill[gray!90] (52,4) rectangle (60,4.5); 
\fill[gray!90] (52.5,11.5) rectangle (60,12); 
\fill[gray!90] (59.5,4.5) rectangle (60,12); 

\fill[gray!90] (52,20) rectangle (52.5,28); 
\fill[gray!90] (52,20) rectangle (60,20.5); 
\fill[gray!90] (52.5,27.5) rectangle (60,28); 
\fill[gray!90] (59.5,20.5) rectangle (60,28); 

\fill[gray!90] (52,36) rectangle (52.5,44); 
\fill[gray!90] (52,36) rectangle (60,36.5); 
\fill[gray!90] (52.5,43.5) rectangle (60,44); 
\fill[gray!90] (59.5,36.5) rectangle (60,44);

\fill[gray!90] (52,52) rectangle (52.5,60); 
\fill[gray!90] (52,52) rectangle (60,52.5); 
\fill[gray!90] (52.5,59.5) rectangle (60,60); 
\fill[gray!90] (59.5,52.5) rectangle (60,60);

\node at (8,8) {$
$};
\end{scope}
\end{tikzpicture}\]
\caption{\label{fig.rule.hierarchybit.1} Illustration 
of the transformation rules of the hierarchy bits when the grouping 
bit is $1$.}
\end{figure}

\item when: 
\begin{itemize}
\item  grouping bit of the cell above in the hierarchy is $0$, 
\item 
the value of the $p$-counter is $\overline{0}$,
\item and its frequency bit is $0$,
\end{itemize}
then the transformation rule relies on the addresses.
If the couple of addresses is in 
\[\{(1^p,0^p) , (0^p,1^p) , (0^p,0^p),(1^p,1^p)\},\] 
then the color is transmitted. In the other cases, it becomes 
%

\item when: 
\begin{itemize}
\item the grouping bit of 
the cell above in the hierarchy 
is $0$, 
\item the value of the $p$-counter is $\overline{0}$ and the frequency bit is $1$, 
\item 
or the value of the $p$-counter is 
not $\overline{0}$,
\end{itemize} 
then the second color is equal 
to the first one. 

\item \textbf{Coherence rules:} since the 
hierarchy bits are superimposed 
on the copy of the Robinson subshift parallel to $\vec{e}^2$ and $\vec{e}^3$, 
the degenerated behaviors of this layer correspond to 
the ones of the Robinson subshift. In this list we consider patterns intersecting 
two-dimensional cells having order greater or equal to $p$.
\begin{enumerate}
\item when near a \textbf{corner}, on a pattern
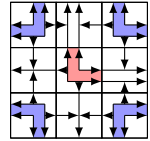
\[\begin{tikzpicture}[scale=0.3]
\robibluehautgauchek{4}{0}
\robibluehautdroitek{0}{0}
\robibluebasdroitek{0}{4}
\robibluebasgauchek{4}{4}
\robiredbasgauche{2}{2}
\robithreehaut{2}{4}
\robithreedroite{4}{2}
\robionehaut{2}{0}
\robionegauche{0}{2}
\draw[step=2] (0,0) grid (6,6);
\end{tikzpicture}\]
the pattern in this layer has to be amongst the following ones: 
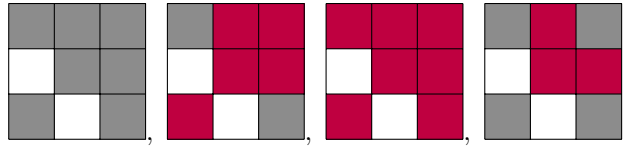
\[\begin{tikzpicture}[scale=0.3]
\fill[gray!90] (0,4) rectangle (2,6);
\fill[gray!90] (4,0) rectangle (6,2);
\fill[gray!90] (0,0) rectangle (2,2);
\fill[gray!90] (2,2) rectangle (6,6);
\draw[step=2] (0,0) grid (6,6);
\end{tikzpicture}, \ \begin{tikzpicture}[scale=0.3]
\fill[gray!90] (0,4) rectangle (2,6);
\fill[gray!90] (4,0) rectangle (6,2);
\fill[purple] (0,0) rectangle (2,2);
\fill[purple] (2,2) rectangle (6,6);
\draw[step=2] (0,0) grid (6,6);
\end{tikzpicture}, \ \begin{tikzpicture}[scale=0.3]
\fill[purple] (0,4) rectangle (2,6);
\fill[purple] (4,0) rectangle (6,2);
\fill[purple] (0,0) rectangle (2,2);
\fill[purple] (2,2) rectangle (6,6);
\draw[step=2] (0,0) grid (6,6);
\end{tikzpicture}, \ \begin{tikzpicture}[scale=0.3]
\fill[gray!90] (0,4) rectangle (2,6);
\fill[gray!90] (4,0) rectangle (6,2);
\fill[gray!90] (0,0) rectangle (2,2);
\fill[purple] (2,2) rectangle (6,6);
\fill[gray!90] (4,4) rectangle (6,6);
\draw[step=2] (0,0) grid (6,6);
\end{tikzpicture}\]
respectively when 
\begin{enumerate}
\item the central corner is colored gray, 
\item this corner is purple and the grouping bit is $1$.
\item this corner is purple, 
the grouping bit is $0$ and 
else the $p$-counter has not value $\overline{0}$, or this value is $\overline{0}$ 
and
the frequency bit is $1$, or the addresses are in 
\[\{(0^p,0^p),(0^p,1^p),(1^p,0^p),(1^p,1^p)\}.\]
\item this corner is purple, the grouping bit is $0$,
the $p$-counter has value $\overline{0}$,
the frequency bit is $0$, and the addresses are not in 
\[\{(0^p,0^p),(0^p,1^p),(1^p,0^p),(1^p,1^p)\}.\]
\end{enumerate}
There are similar rules for the other corners. 
This allows
degenerated behaviors which do not happen around finite cells to be avoided. The other rules in this list 
have the same function. 
\item near a the middle or the quarter of an edge of a cell 
or near a corner of a petal with $0,1$-counter value equal to $0$, the symbols on the 
blue corners are gray if the border of the cell is gray. If this 
is purple, the grouping bit is $0$,
the $p$-counter has value $\overline{0}$ and
the frequency bit is $0$, and the addresses are not in 
$\{(0^p,0^p),(0^p,1^p),(1^p,0^p),(1^p,1^p)\}$, the blue corners are colored gray.
In the other cases, they are colored purple.
For instance, on the pattern 
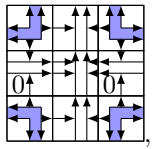
\[
,\]
the pattern has to be amongst the following ones: 
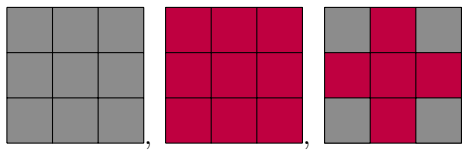
\[%
\]
\end{enumerate}
\end{itemize} 
\end{itemize}

\noindent \textbf{\textit{Global behavior:}}
\bigskip

Over the copy of the Robinson subshift
which is parallel to $\vec{e}^2$ 
and $\vec{e}^3$,
we make a signal propagate and 
transform on intersections 
of support petals and the transmission petal just 
above in the hierarchy. With this signal, 
we color the blue corners
with a hierarchy bit in $\{%
\}$. This process relies 
on the frequency bits in a similar way as 
in the proof of Theorem~\ref{thm.meyerovitch}, 
except that each time the signal enters 
inside an order $2^kp$ two-dimensional cell, 
not matter the color of the cell, the 
support petals just under in the hierarchy 
are colored %
. This leads to the preservation of 
the 
minimality property.
After synchronization, 
Lemma~\ref{lem.count.random.2} count 
the cardinality of these sets. 

\subsubsection{Random bits}

\noindent \textbf{\textit{Symbols:}} \bigskip

Elements
of $\{0,1\}$ (these are the random bits, which 
generate entropy dimension) and a blank symbol. \bigskip

\noindent \textbf{\textit{Local rules:}} 

\begin{itemize}
\item \textbf{Localization:} the non blank symbols are superimposed on and only 
on positions with north west blue corner having a hierarchy bit 
equal to %
.
\item \textbf{Transformation:} for any position $\vec{u} \in \Z^3$ 
which is not on the top (according to $\vec{e}^1$) face of a three-dimensional cell, 
the random bit on position $\vec{u}$ is equal to the random bit on 
position $\vec{u}+\vec{e}^1$. The transformation rule on these faces 
will be stated in Section~\ref{sec.incr.signal}.
\end{itemize}

The next sections describe the set of random bits as a value 
(similarly 
to the linear counter) 
incremented by $1$ going through the top 
face of a three-dimensional cell according to direction $\vec{e}^1$. 

\subsubsection{Convolutions}

In this sublayer is generated a path 
on the top faces of the three-dimensional cells 
according to direction $\vec{e}^1$. \bigskip

\noindent \textbf{\textit{Symbols:}}

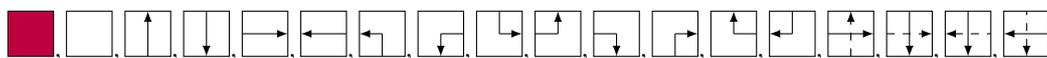
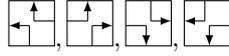
\[
\]

\noindent \textbf{\textit{Local rules:}} 

\begin{itemize}
\item \textbf{Localization:} the non-blank 
symbols are superimposed only on the top faces of the 
three-dimensional cells according to direction $\vec{e}^1$.
\item \textbf{Transmission:} the arrows and blank 
symbols
propagate through five or four arrows symbols of the Robinson copy 
parallel to $\vec{e}^2$ and $\vec{e}^3$. 
\item \textbf{Triggering positions:} 

The position 
on the top or bottom line of the 
border of the face, with 
six arrows symbol in the Robinson copy 
parallel to this face is called \textbf{triggering 
position}. These positions 
are superimposed with 
$\begin{tikzpicture}[scale=0.3]
\fill[purple] (0,0) rectangle (1,1);
\draw (0,0) rectangle (1,1); 
\end{tikzpicture}$, and this symbol 
can be only on this position. 
The presence of this symbol implies a right arrow 
symbol on the position on the right, a left arrow on the left and a blank symbol above
if on the top line. When on the bottom line, the direction of the arrows is 
reversed, and the blank symbol is forced upwards.

\item \textbf{Correspondence 
for other types of positions:}

The following tables give a correspondence 
between other types of positions and 
the symbol superimposed on them in 
this layer. Table~\ref{table.border.rules} 
is related to positions on the border
of a face. In this table, induction 
refers to the fact that the position 
type imposes the presence of a symbol 
on a position nearby. The symbol 
and the position are specified in the table.
Table~\ref{table.inside.rules} refers to 
positions inside a face.

\item \textbf{Determination 
of symbols on the border:}

A cross symbol can not be present on the border.
\end{itemize}

\begin{table}[h]
{\renewcommand{\arraystretch}{1.2}
\begin{tabular}{|c||c||c|c|}
    \hline Type of position & Symbol 
		& Inducted symbols & 
		Inducted positions \\
    \hline
Corners of the area & \begin{tikzpicture}[scale=0.5,baseline=1.5mm]
\draw (0,0) rectangle (1,1); 
\end{tikzpicture} & & \\
\hline 
On the top line, superimposed 
with $\begin{tikzpicture}[scale=0.25,baseline=1.5mm]
\draw (0,0) rectangle (2,2) ;
\draw [-latex] (0,1) -- (2,1) ;
\draw [-latex] (0,1.5) -- (2,1.5) ; 
\draw [-latex] (1,0) -- (1,1) ; 
\draw [-latex] (1,2) -- (1,1.5) ;
\draw [-latex] (1.5,0) -- (1.5,1) ; 
\draw [-latex] (1.5,2) -- (1.5,1.5) ;
\end{tikzpicture}$ & 
$\begin{tikzpicture}[scale=0.5,baseline=1.5mm]
\draw (0,0) rectangle (1,1); 
\draw (0.5,0) -- (0.5,0.5);
\draw[-latex] (0.5,0.5) -- (1,0.5);
\end{tikzpicture}$ & $\begin{tikzpicture}[scale=0.5,baseline=1.5mm]
\draw (0,0) rectangle (1,1); 
\draw[-latex] (0,0.5) -- (1,0.5);
\end{tikzpicture}, \begin{tikzpicture}[scale=0.5,baseline=1.5mm]
\draw (0,0) rectangle (1,1); 
\draw[-latex] (1,0.5) -- (0,0.5);
\end{tikzpicture}$ & right, resp. left 
position \\
\hline 

Top line, $\begin{tikzpicture}[scale=0.25,baseline=1.5mm]
\draw (0,0) rectangle (2,2) ;
\draw [-latex] (2,1) -- (0,1) ;
\draw [-latex] (2,1.5) -- (0,1.5) ; 
\draw [-latex] (1,0) -- (1,1) ; 
\draw [-latex] (1,2) -- (1,1.5) ;
\draw [-latex] (0.5,0) -- (0.5,1) ; 
\draw [-latex] (0.5,2) -- (0.5,1.5) ;
\end{tikzpicture}$ &  
$\begin{tikzpicture}[scale=0.5,baseline=1.5mm]
\draw (0,0) rectangle (1,1); 
\draw[latex-] (0.5,0) -- (0.5,0.5);
\draw (0.5,0.5) -- (0,0.5);
\end{tikzpicture}$ & $\begin{tikzpicture}[scale=0.5,baseline=1.5mm]
\draw (0,0) rectangle (1,1); 
\end{tikzpicture}$,$\begin{tikzpicture}[scale=0.5,baseline=1.5mm]
\draw (0,0) rectangle (1,1); 
\draw[-latex] (0,0.5) -- (1,0.5);
\end{tikzpicture}$ & right, resp. left 
position \\
\hline 

  \end{tabular}
  }
  \caption{\label{table.border.rules}
	Correspondence table for 
	positions on the \textbf{border} of a face.}
\end{table}

\begin{table}[h]
{\renewcommand{\arraystretch}{1.2}
  \begin{tabular}{|c|c||c|c|}
    \hline Type of position & Symbol & Type of 
		position & Symbol \\
    \hline
		
Superimposed 
with $\begin{tikzpicture}[scale=0.25,baseline=1.5mm]
 \robibluebastgauche{0}{0};
\end{tikzpicture}$ &  $\begin{tikzpicture}[scale=0.5,
baseline=1.5mm]
\draw (0,0) rectangle (1,1); 
\draw (0.5,0.5) -- (0.5,1);
\draw[-latex] (0.5,0.5) -- (1,0.5);
\end{tikzpicture}$& $\begin{tikzpicture}[scale=0.25,baseline=1.5mm]
\draw (0,0) rectangle (2,2) ;
\draw [-latex] (1,0) -- (1,2) ;
\draw [latex-] (0.5,2) -- (0.5,0) ; 
\draw [-latex] (0,1) -- (0.5,1) ; 
\draw [-latex] (2,1) -- (1,1) ;
\draw [-latex] (0,0.5) -- (0.5,0.5) ; 
\draw [-latex] (2,0.5) -- (1,0.5) ;
\end{tikzpicture}$ & $\begin{tikzpicture}[scale=0.5,baseline=1.5mm]
\draw (0,0) rectangle (1,1);
\draw (0.45,1) -- (0.45,0.55);
\draw[-latex] (0.45,0.55) -- (0,0.55);
\draw (0.55,0) -- (0.55,0.45);
\draw[-latex] (0.55,0.45) -- (1,0.45);
\end{tikzpicture}$\\

\hline 

$\begin{tikzpicture}[scale=0.25,baseline=1.5mm]
 \robibluehauttdroite{0}{0};
\end{tikzpicture}$ & $\begin{tikzpicture}[scale=0.5,baseline=1.5mm]
\draw (0,0) rectangle (1,1); 
\draw (0.5,0) -- (0.5,0.5);
\draw[-latex] (0.5,0.5) -- (0,0.5);
\end{tikzpicture}$ & $\begin{tikzpicture}[scale=0.25,baseline=1.5mm]
\draw (0,0) rectangle (2,2) ;
\draw [-latex] (0,1) -- (2,1) ;
\draw [-latex] (0,0.5) -- (2,0.5) ; 
\draw [-latex] (1,2) -- (1,1) ; 
\draw [-latex] (1,0) -- (1,0.5) ;
\draw [-latex] (1.5,2) -- (1.5,1) ; 
\draw [-latex] (1.5,0) -- (1.5,0.5) ;
\end{tikzpicture}$ & $\begin{tikzpicture}[scale=0.5,baseline=1.5mm]
\draw (0,0) rectangle (1,1); 
\draw[-latex] (0,0.5) -- (1,0.5);
\draw[dashed,-latex] (0.5,0) -- (0.5,1);
\end{tikzpicture}$ \\
\hline 

  \end{tabular}
  }
  \caption{\label{table.inside.rules}
	Correspondence table for 
	positions \textbf{inside} a face.}
\end{table}

\noindent \textbf{\textit{Global behavior:}} \bigskip

The global behavior is the 
drawing of a curve starting and ending 
at some position at the center of the top or bottom 
row of the border (depending 
on the orientation of the three-dimensional cell) 
of the top face according to $\vec{e}^1$ of three-dimensional cells.
This position is specified by purple color, 
and the path is going through 
every position having a random bit on it. 
This serves as a circuit
for the incrementation signal of the hierarchical 
counter.
See Figure~\ref{fig.circonvolutions}.

\subsubsection{\label{sec.incr.signal} Incrementation signal}

\noindent \textbf{\textit{Symbols:}} \bigskip

This sublayer has symbols in $ \{\begin{tikzpicture}[scale=0.3]
\fill[Salmon] (0,0) rectangle (1,1);
\draw (0,0) rectangle (1,1);
\end{tikzpicture}, 
\begin{tikzpicture}[scale=0.3]
\fill[YellowGreen] (0,0) rectangle (1,1);
\draw (0,0) rectangle (1,1);
\end{tikzpicture}\}$ or 
$\{\begin{tikzpicture}[scale=0.3]
\fill[Salmon] (0,0) rectangle (1,1);
\draw (0,0) rectangle (1,1);
\end{tikzpicture}, 
\begin{tikzpicture}[scale=0.3]
\fill[YellowGreen] (0,0) rectangle (1,1);
\draw (0,0) rectangle (1,1);
\end{tikzpicture}\} ^2$ and a blank symbol. \bigskip

\noindent \textbf{\textit{Local rules:}} 

\begin{itemize}
\item {\bf{Localization :}} 
the non blank symbols are superimposed only on 
non blank convolution symbols when on the top face according to $\vec{e}^1$ of 
a three-dimensional cell. 
\item when outside these faces, 
the non-blank symbols 
are superimposed on the positions that don't have 
a petal symbol, and the symbol on a position $\vec{u} \in \Z^3$ 
is transmitted to next positions $\vec{u} \pm \vec{e}^2$ and $\vec{u} \pm \vec{e}^3$
when these positions are not in the border of a 
two-dimensional cell with grouping bit equal to $1$.
These positions have a simple symbol.
\item the two bits symbols 
can be superimposed only on 
positions with a cross symbol 
in the convolution sub-sublayer, and on top faces of three-dimensional cells according 
to $\vec{e}^1$.
\item \textbf{Initialization:} 
The position with purple symbol in the convolution sublayer is 
superimposed with
is transmitted to the next position in the direction 
of the arrow in the convolution sublayer if: 

\begin{itemize}
\item this next position 
has random bit equal to $1$ or the grouping border bit is 
$0$ (we use this rule so that the counter in a $2^k p$ order two-dimensional cell 
does not interact with the random bits that are in a $2^j p$ order cell, with 
$j<k$), 
\item or the next position 
has the purple symbol in the convolution sublayer.
\end{itemize} 
The symbol 
is changed only in these cases. 
In other cases, the next position is marked with %
 is transmitted through the arrow unless the next position 
has the purple symbol.
\item \textbf{Freezing signal:} on 
the position $\vec{u}$ with purple symbol in the convolution sublayer, 
if the color on the next position with arrow pointing on this position is 
%
, then the symbols on positions 
$\vec{u} - \vec{e}^1$ and $\vec{u}+\vec{e}^1$ are different. 
Else, they are equal. When a position on the face 
have symbol %
, then 
the symbol on positions $\vec{u} - \vec{e}^1$ and $\vec{u}+\vec{e}^1$
is %
.
\item \textbf{Incrementation:} 
on position $\vec{u}$ of a top face according to $\vec{e}^1$, 
if the grouping bit is $1$ and the grouping
border bit is $1$, the symbol is %
. The symbol on position 
$\vec{u}+\vec{e}^1$ is %
, the random bit on position 
$\vec{u}+\vec{e}^1$ is $0$. Else, it is equal to the random bit 
on position $\vec{u}$.
\item \textbf{Coherence rule:} 
we use coherence rules so that on degenerated 
behaviors of the structure layer, the sense 
of propagation of the incrementation signal 
is respected. For instance, 
on a pattern inside the face being
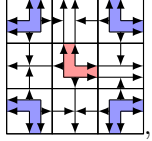
\[%
,\]
the possible colorings are: 
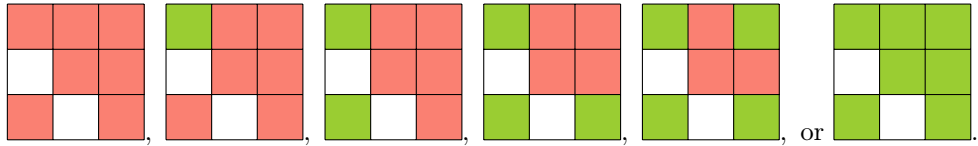
\[%
.\]
\end{itemize}

\noindent \textbf{\textit{Global behavior:}} 
\bigskip

When going through an order $n \neq 2^k p$ 
three-dimensional cell for 
any $k$, the random bits are not changed. 
When going through the top face 
according to $\vec{e}^1$ 
of an order $2^k p$ order three-dimensional cell, 
the random bits that are in the 
corresponding two-dimensional cell and 
are not in a $2^j p$ cell 
with $j<k$ are grouped (using 
the grouping bits and the grouping border bits). 
The sequence of these bits is incremented by $1$ -
in a similar way as is incremented the linear 
counter, using the convolutions -
when the freezing signal is \begin{tikzpicture}[scale=0.3]
\fill[YellowGreen] (0,0) rectangle (1,1);
\draw (0,0) rectangle (1,1);
\end{tikzpicture}. It is not 
incremented when it is \begin{tikzpicture}[scale=0.3]
\fill[Salmon] (0,0) rectangle (1,1);
\draw (0,0) rectangle (1,1);
\end{tikzpicture}. This last case 
happens only once when 
this set of bits is in maximal position. 
Because of this, and the 
formula proved in the 
first point of Lemma~\ref{lem.count.random.2}, 
the period of this counter is $2^{2^{l_k}}+1$, where 
\[l_k = 4[(k+1)(p-1)+1] +2(2^k-1 - c_k )+4p c_{k}\]

The functions $k \mapsto 2^k-1 - c_k$ 
and $k \mapsto c_k$ 
are non-decreasing, and the 
function $k \mapsto (k+1)(p-1)$ 
is strictly increasing. Hence all the 
numbers $l_k$ are different. 
As a consequence, all these numbers are 
different Fermat numbers.

\subsection{\label{sec.synchr.dim.entropique} Synchronization layer}

We use this layer to synchronize 
the hierarchical counters in the orthogonal directions 
of their incrementation direction $\vec{e}^1$
(we recall that from the way we coded the 
linear counters, they are already 
synchronized in the 
orthogonal directions of their 
incrementation direction 
$\vec{e}^2$) for three-dimensional cells having 
the same order in case of the linear counter. 
We synchronize only $2^k p$ cells 
for the same $k \ge 0$ (this is the purpose 
of grouping bits).

This layer has two sub-layers.

\subsubsection{Synchronization areas}

The aim of this first sublayer is to localize places 
where synchronization of hierarchical counters occurs.

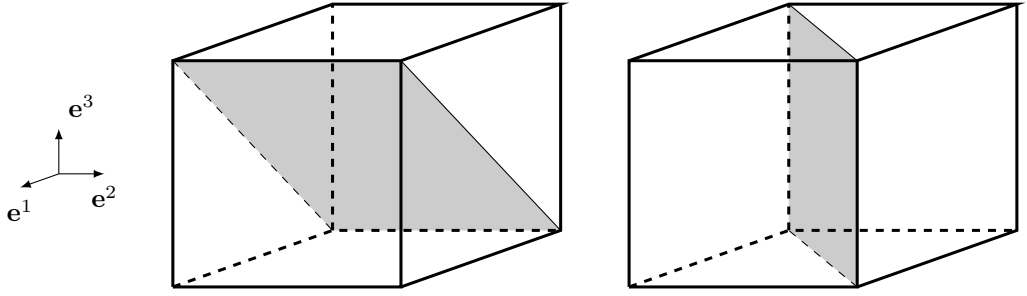
\begin{figure}[ht]
\[\begin{tikzpicture}[scale=0.3]

\draw[-latex] (-5,5) -- (-5,7);
\node at (-4,8) {$\vec{e}^3$};
\draw[-latex] (-5,5) -- (-3,5);
\node at (-3,4) {$\vec{e}^2$};
\draw[-latex] (-5,5) -- (-6.7,5-1.7*5/14);
\node at (-6.7,4-1.7*5/14) {$\vec{e}^1$};
\begin{scope}
\fill[gray!40] (0,10) -- (10,10) -- (17,2.5) -- (7,2.5) -- (0,10);
\draw[line width=0.4mm] (0,0) -- (10,0) -- (10,10) -- (0,10) -- (0,0);
\draw[line width=0.4mm] (0,10) -- (7,12.5) -- (17,12.5) -- (10,10); 
\draw[line width=0.4mm] (17,12.5) -- (17,2.5) -- (10,0);
\draw[line width=0.4mm,dashed] (0,0) -- (7,2.5) -- (17,2.5);
\draw[line width=0.4mm,dashed] (7,2.5) -- (7,12.5) ;
\draw[dashed] (0,10) -- (7,2.5);
\draw (10,10) -- (17,2.5);
\end{scope}

\begin{scope}[xshift=20cm]
\fill[gray!40] (10,0) -- (10,10) -- (7,12.5) -- (7,2.5) -- (10,0);
\draw[line width=0.4mm] (0,0) -- (10,0) -- (10,10) -- (0,10) -- (0,0);
\draw[line width=0.4mm] (0,10) -- (7,12.5) -- (17,12.5) -- (10,10); 
\draw[line width=0.4mm] (17,12.5) -- (17,2.5) -- (10,0);
\draw[line width=0.4mm,dashed] (0,0) -- (7,2.5) -- (17,2.5);
\draw[line width=0.4mm,dashed] (7,2.5) -- (7,12.5) ;
\draw[dashed] (10,0) -- (7,2.5);
\draw (10,10) -- (7,12.5);
\end{scope}
\end{tikzpicture}\]
\caption{\label{fig.synchr.planes} Illustration of the synchronization areas rules.
The plane generated by the first symbol of the couple 
is on the left, and the other one on the right.}
\end{figure}

\noindent \textbf{\textit{Symbols:}} \bigskip

The symbols are 
$\left(\begin{tikzpicture}[scale=0.3]
\draw (0,0) -- (1,1); 
\draw (0,0) rectangle (1,1);
\end{tikzpicture}, \begin{tikzpicture}[scale=0.3]
\draw (0,0) rectangle (1,1);
\end{tikzpicture}\right)^2$. \bigskip

\noindent \textbf{\textit{Local rules:}} 

\begin{itemize}
\item the first symbol of the couple is transmitted in the directions 
$\pm \vec{e}^2$, and 
the second one in the directions $\pm \vec{e}^3$.
\item on a position $\vec{u} \in \Z^3$ which is not superimposed with 
a corner in the copy of the Robinson subshift 
parallel to $\vec{e}^3$ and $\vec{e}^1$, the first symbol 
of the couple is transmitted to positions 
$\vec{u} + \vec{e}^3+\vec{e}^1$ and $\vec{u}-\vec{e}^3-\vec{e}^1$.
\item on a position $\vec{u} \in \Z^3$ which is not superimposed with 
a corner in the copy of the Robinson subshift 
parallel to $\vec{e}^1$ and $\vec{e}^2$, the first symbol 
of the couple is transmitted to positions 
$\vec{u} + \vec{e}^1+\vec{e}^2$ and $\vec{u}-\vec{e}^1-\vec{e}^2$.
\item the symbol as first symbol of the couple 
can not be superimposed on a position 
in the border of of a two-dimensional cell in the Robinson copy 
parallel to $\vec{e}^3$ and $\vec{e}^1$ except a corner.
\item the symbol %
 as second symbol of the couple 
can not be superimposed on a position 
in the border of of a two-dimensional cell in the Robinson copy 
parallel to $\vec{e}^1$ and $\vec{e}^2$ except a corner.
\item on a position $\vec{u}$ with a north east or 
south west corner in the Robinson copy parallel to $\vec{e}^3$ 
and $\vec{e}^1$, the first symbol of the couple is 
%
. According to the orientation of the corner, 
it forces the presence of another symbol %
 on another position as follows: 
\begin{enumerate}
\item if a north east corner, the symbol is forced on position 
$\vec{u} - \vec{e}^3 - \vec{e}^1$
\item if a south west corner, it is forced on 
$\vec{u} + \vec{e}^3 + \vec{e}^1$.
\end{enumerate}
Moreover, when the grouping bit is $0$, it forces the symbol 
 %
 on positions: 
\begin{enumerate}
\item if a north east corner, the symbol is forced on position 
$\vec{u} + \vec{e}^3 + \vec{e}^1$
\item if a south west corner, it is forced on 
$\vec{u} - \vec{e}^3 - \vec{e}^1$.
\end{enumerate}
\item on a position $\vec{u}$ with a north east or 
south west corner in the Robinson copy parallel to $\vec{e}^3$ 
and $\vec{e}^1$, the second symbol of the couple is 
%
. According to the orientation of the corner, 
it forces the presence of another symbol %
 on another position as follows: 
\begin{enumerate}
\item if a north east corner, the symbol is forced on position 
$\vec{u} - \vec{e}^1 - \vec{e}^2$
\item if a south west corner, it is forced on 
$\vec{u} + \vec{e}^1 + \vec{e}^2$.
\end{enumerate}
Moreover, when the grouping bit is $0$, it forces the symbol 
 %
 on positions: 
\begin{enumerate}
\item if a north east corner, the symbol is forced on position 
$\vec{u} + \vec{e}^1 + \vec{e}^2$
\item if a south west corner, it is forced on 
$\vec{u} - \vec{e}^1 - \vec{e}^2$.
\end{enumerate}
\end{itemize}

\noindent \textbf{\textit{Global behavior:}} \bigskip

The global behavior induced by these rules is 
the drawing, for 
any three-dimensional cell having order $2^k p$ for some $k \ge 0$, of two planes crossing it. 
One of them links the north east edge parallel to $\vec{e}^2$ to the south west one. The other one 
links the north east edge parallel to $\vec{e}^3$ to the south west one, as 
on Figure~\ref{fig.synchr.planes}. 
The first plane is drawn by the first symbol 
and generated by the vectors 
 $\vec{e}^3+\vec{e}^1$ and $\vec{e}^2$. 
The other one is drawn by the second symbol 
and is generated by the vectors $\vec{e}^1+\vec{e}^2$ and $\vec{e}^3$.

For all $k$, these planes cross order $j$ cells for all $j<2^k p$ that are inside 
the $qp$ order three-dimensional cell. However, an order $j \neq 2^k p$ cell for any $k$, 
one of these planes can cross this cell if and only if it is on the diagonal 
plane of an order $2^k p$ cell.

\subsubsection{Synchronization of the 
hierarchical counters}

This sub-sublayer consists of two copies of the random bit sublayer, 
the hierarchy bit sublayer and the incrementation signal sublayer, except 
that the first set of copies is parallel to $\vec{e}^1$ and $\vec{e}^2$ 
and the second one to $\vec{e}^1$ and $\vec{e}^3$ (instead of $\vec{e}^2$ and 
$\vec{e}^3$ for the hierarchical counter). Moreover, 
in these copies, there is no incrementation signal. This means that the rules on 
the faces of three-dimensional cells are the same as outside.
Furthermore, when the first symbol of the synchronization areas sublayer is 
\begin{tikzpicture}[scale=0.3]
\draw (0,0) -- (1,1); 
\draw (0,0) rectangle (1,1);
\end{tikzpicture}, then the random bit 
in the first copy is 
equal to the random bit of the hierarchical counter. 
When the second symbol of the 
synchronization areas sublayer 
is \begin{tikzpicture}[scale=0.3]
\draw (0,0) -- (1,1); 
\draw (0,0) rectangle (1,1);
\end{tikzpicture}, then the random 
bit in the second copy is equal to 
the random bit of the hierarchical counter.

The \textit{global behavior} induced is that 
the values of the hierarchical counter 
in two adjacent, in 
else direction $\vec{e}^2$ or 
$\vec{e}^3$, order $2^k p$ 
three-dimensional cells, are equal.  

As a consequence, one can state the following lemma: 

\begin{lemma} \label{lem.count.random.2}
For all $q \neq 2^k$ for any $k$, 
the number $r_q$ of possible colorings by 
hierarchy bits and random bits 
on a $2qp+2$ order supertile 
(centered on an order $qp$ cell) verifies 
the following 
inequalities: 
\[ \left[\alpha (p).q^{\lambda(p)} \right] 4^{c^0_q} . 16^{p c^1_q} 
\ge \log_2(r_q) \ge 4^{c^0_q} . 16^{p c^1_q}\]
where $c_{q,1}$ is the number of $i \le q-1$ such that 
$a_i = 1$ and $c_{q,0}$ is the number of $i \le q-1$ such that $a_i = 0$, 
and $\alpha(p),\lambda(p)>0$ depend only on $p$.
\end{lemma}

\begin{proof}

\begin{itemize}

\item \textbf{A formula on 
the number of proper 
blue corner positions 
in a cell 
having hierarchy bit \begin{tikzpicture}[scale=0.3]
\fill[purple] (0,0) rectangle (1,1);
\draw (0,0) rectangle (1,1);
\end{tikzpicture}:} 

For all $k \ge 0$, in an order $2^k p$ cell, 
the number of blue corners having hierarchy bit 
\begin{tikzpicture}[scale=0.3]
\fill[purple] (0,0) rectangle (1,1);
\draw (0,0) rectangle (1,1);
\end{tikzpicture} that  
are not in an order $2^j p$ cell, $j < k$, is
\[d_k = 4.4^{k}.16^{k(p-1)}.4^{2^k-1 - c_k - k} . 4.16^{p-1}. 16^{p c_{k}},\]
where $c_{k}$ is the number of $i \le 2^k -1$ 
such that the $i$th frequency bit $a_i$ is $1$.

Indeed: 
\begin{enumerate}
\item the process that rules the coloring with hierarchy bits 
of this cell starts on $4$ order $2^k p -1$ cells, hence the factor $4.16^{p-1}$.
\item 
The factor $4^{k}.16^{k(p-1)}$ comes from the 
transitions occurring when the grouping bit is $1$ (there are $k$ such transitions).
\item The factor $4^{2^k-1 -c_k-k}$ comes from the transitions 
occurring when the grouping bit is $0$ and 
the frequency bit is $0$. 
\item The factor $4.16^{pc_k}$ comes from the transitions 
when the grouping bit is $0$ and the frequency bit is $1$. 
\item The factor 
$4$ comes from the fact that there are four blue corners 
for each order $0$ cell. 
\end{enumerate}

\item \textbf{On the hierarchy bits 
of a supertile:}

An order $2qp+2$ supertile is centered on an 
order $qp$ cell. 
There are two possibilities 
for the hierarchy bits,as follows.

\begin{enumerate}
\item \textbf{The border of this cell is colored 
\begin{tikzpicture}[scale=0.3]
\fill[gray!90] (0,0) rectangle (1,1);
\draw (0,0) rectangle (1,1);
\end{tikzpicture}}. 
This means that all the blue corners in 
the supertiles colored 
with \begin{tikzpicture}[scale=0.3]
\fill[gray!90] (0,0) rectangle (1,1);
\draw (0,0) rectangle (1,1);
\end{tikzpicture}, except 
the ones that are in an order $2^j p$ cell 
with $2^j \le q$. This is 
equivalent to $j \le \lfloor \log_2 (q) \rfloor$.
Since the $j$th hierarchical counter for all $j\le \lfloor \log_2 (q) \rfloor$ 
in this supertile
are synchronized, the number of possible colorings 
of this supertile by random bits is given 
by the product for $j \le  \lfloor \log_2 (q) \rfloor$ 
of the numbers of possible colorings 
by random bits of the set of blue corners 
that are in an order $2^j p$ cell and not in an order $2^i p$ cell 
with $i <j$. These numbers are $2^{d_j}$, so the total number 
of random bits displays is $2^{\lambda_1 (q)}$, 
where 
\[\lambda_1 (q) = \sum_{j=0}^{\lfloor \log_2 (q) \rfloor} d_j.\]
\item \textbf{The border is colored \begin{tikzpicture}[scale=0.3]
\fill[purple] (0,0) rectangle (1,1);
\draw (0,0) rectangle (1,1);
\end{tikzpicture}}. 
In this case, we also have to count 
the number of hierarchy bits that 
are in the supertile but 
not in an order $2^j p$ cell with $j\le \lfloor \log_2 (q) \rfloor$. 
This case 
is similar to the first point 
of this proof, and the number of 
random bits displays in this case is 
$2^{\lambda_2 (q)}$,where: 
\[\lambda_2 (q) = \sum_{j=0}^{\lfloor \log_2(q) \rfloor} d_j + 
4. 4^{\lfloor \log_2(q) \rfloor+1}.16^{\lfloor \log_2(q) \rfloor (p-1)} 
.4^{c^0_q - \lfloor \log_2(q) \rfloor} . 16^{pc^1_q}.\]
\end{enumerate}

As a consequence, the total 
number of possibilities for the hierarchy 
bits and random bits is 
\[2^{\lambda_1 (q)} + 2^{\lambda_2 (q)}
\le 2.2^{\lambda_2 (q)},\]
since $\lambda_1 (q) \le \lambda_2 (q)$. 

\item \textbf{Inequalities:}

The lower bound is clear. 
For the upper bound, we have, 
following the last point:
\[\log_2 (r_q) \le \left(
4. 4^{\lfloor \log_2(q) \rfloor+1}.16^{\lfloor \log_2(q) \rfloor (p-1)} 
.4^{c^0_q - \lfloor \log_2(q) \rfloor} . 16^{pc^1_q} + \sum_{j=0}^{\lfloor \log_2(q) \rfloor} d_j + \right)+1\]
Hence from the fact that the functions $k \mapsto c_k$ and $k \mapsto 2^k-1-c_k$ 
are non-decreasing and the inequality 
$\log_2 (q) \ge \lfloor \log_2 (q)\rfloor$: 
\[\begin{array}{ccl}
\log_2 (r_q) &\le & 2(\log_2(q) +1). 
4. 4^{\log_2(q)+1}.16^{\log_2(q) (p-1)} 
.4^{c^0_q - \log_2(q) + 1} . 16^{pc^1_q}\\
& = & 2(\log_2(q) +1). 
4^2 .16^{\log_2(q) (p-1)} 
.4^{c^0_q} . 16^{pc^1_q}
\end{array}.\]
Since $2(\log_2(q) +1). 
4^2 .16^{\log_2(q) (p-1)}$
can be bounded by $\alpha(p) q ^{\lambda(p)}$ 
for some functions $\alpha$ and $\lambda$,
this provides the upper bound.
\end{itemize}
\end{proof}

\section{Properties of the subshifts \texorpdfstring{$X_{z}$}{Xz}:}

\subsection{Pattern completion}

In this section, we prove the following proposition, 
which will serve to prove that $X_z$ is minimal.
We assume the reader has some familiarity with the properties of the 
Robinson subshift.

\begin{proposition} \label{prop.completion}
Let $P$ some $n$-block in the language of $X_z$. Then $p$ 
can be completed into an admissible pattern over a three-dimensional cell.
\end{proposition}

Let $P$ some $n$ block in the language of $X_z$ which appears in some 
configuration $x$.
Let us prove that it can be completed into a three-dimensional supertile.
We follow some order in the layers for the completion. 
First we complete the pattern in the structure layer, 
then in the functional areas layer, the linear counter and machine layers 
and then the hierarchy bits and the hierarchical counter.
After this we prove how this supertile can be completed into 
a three-dimensional cell.

\subsubsection{Completion of the structure}

When the pattern $P$ is a sub-pattern of an infinite supertile of $x$, 
this is clear that the projection of $P$ into 
the structure layer can be completed into a finite supertile in 
the configuration $x$.

When the support of $P$ crosses the separating area 
between the supports of the infinite supertiles, as 
in the proof of Proposition~\ref{prop.complete.rob}, 
we can still complete 
the projection of $P$ over the structure layer into an order $k$ three-dimensional supertile,
with $k$ great enough.

\subsubsection{Completion of the functional areas}

The completion in the functional areas layer is more subtle and depends 
on how the support of the pattern crosses the area between the supports 
of infinite supertiles. We will choose $k$ great enough in each of the cases.
We tell how to complete the border of the greatest three-dimensional cell included 
in the supertile, since the smaller cells can be completed according 
to the configuration $x$.
Here is a list of all the possible cases: 

\begin{itemize}
\item When the pattern intersects the corner of an infinite three-dimensional cell, 
we complete the projection of $P$ on the structure layer
on the inside of each of the faces into the inside of a two-dimensional cell (without 
the border). The size of the two-dimensional cell is chosen 
according to the value of the $p$-counter and the grouping bit 
of the three-dimensional cell corner. The coherence rules 
allow the coloring 
of the petals in this two-dimensional 
cell to be determined.
Indeed, since this is the corner of a three-dimensional cell, 
the border rules imply that the borders of the faces are colored with 
\begin{tikzpicture}[scale=0.3]
\fill[blue!40] (0,0) rectangle (1,1);
\draw (0,0) rectangle (1,1);
\end{tikzpicture} in the functional areas sublayer, and 
\begin{tikzpicture}[scale=0.3]
\fill[purple] (0,0) rectangle (1,1);
\draw (0,0) rectangle (1,1);
\end{tikzpicture} in the active areas sublayer.
The coherence rules allow
the colors of the blue corners 
around the corner of the three-dimensional infinite cell
to be determined. The transformation rules imply that 
there is a unique possibility 
for the coloring of the border of the order $0$ two-dimensional cells around, 
then the order $1$ two-dimensional cells, etc.
This thus determines the color of the other petals in the inside of the two dimensional cells.
See an illustration on Figure~\ref{fig.completion.coherence} 
for the functional areas sublayer.

\begin{figure}[ht]
\[\begin{tikzpicture}[scale=0.4]
\draw[color=blue!40, line width = 0.5mm] (0,0) -- (10,0); 
\draw[color=blue!40, line width = 0.5mm] (0,0) -- (0,10); 
\draw[color=blue!40, line width = 0.5mm] (0.5,0.5) rectangle (1,1);
\draw[color=blue!40, line width = 0.5mm] (1.25,1.25) rectangle (3.25,3.25); 
\draw[color=blue!40, line width = 0.5mm] (3.5,3.5) -- (10,3.5);
\draw[color=blue!40, line width = 0.5mm] (3.5,3.5) -- (3.5,10);
\end{tikzpicture}\]
\caption{\label{fig.completion.coherence} Illustration of the 
sequence of implications of the coherence rule of the functional areas sublayer.}
\end{figure}
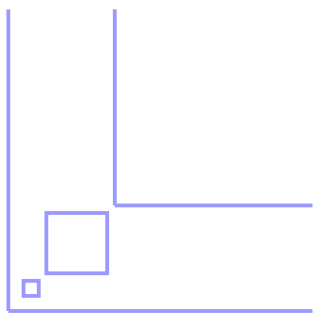

The insides of these two-dimensional cells can be considered as the quarter of a greater cell. 
This allows completing the structure 
of the pattern $P$ into 
a finite three-dimensional cell which is colored in the functional areas layer 
as any three-dimensional cell that appears in the subshift.

\item When the pattern intersects an infinite gray face, 
one can complete the pattern into the structure layer 
such that the projection over the copy of 
$X_{adR}$ parallel to the face 
is a supertile. We use Proposition~\ref{prop.complete.rob} in case that
the patterns intersect the 
area between infinite two-dimensional 
supertiles over 
the face. We chose 
the order of the cell according to the value of the $p$-counter 
on the border of this cell.
As we have coherence rules for the inside of the faces, this supertile  
can be colored as any finite supertile on the faces in the functional areas layer.
There are the following possibilities (in all these cases, 
the symbols on this supertile are determined): 
\begin{itemize}
\item the central corner of the supertile is colored 
 in the functional areas sublayer, and as a consequence is colored 
with the (%
) in the active areas sublayer. The color of the petal intersecting this one, 
and then the other petals in the supertile, 
is thus determined according to the orientation symbol 
written on the central corner.
\item this corner is colored $\left(%
, \rightarrow \right)$ in the functional areas sublayer, 
and else $\left( %
\right)$ in the active areas sublayer. 
\item it is colored or $\left( %
, \rightarrow \right)$ in the functional areas sublayer, 
and else $\left( %
\right)$ in the active areas sublayer. 
\item this corner is colored with %
, and there is no restriction 
in the active areas sublayer.
\end{itemize}
Concerning the functional areas sublayer, all 
these cases can be encountered in any arbitrary large 
colored faces with arbitrary 
value of the counter. Indeed, they 
are encountered on any supertile whose orientations 
with respect to its $n+1$, $n+2$ and $n+3$ supertiles lie 
in a particular set. This set corresponds 
to the rules of the 
functional areas sublayer, and 
to the condition 
that the  $n+3$ order supertile is colored blue. 
This happens for 
all the values of the $p$-counter in 
any two-dimensional cell. 
Because some columns or lines are 
active and others not, all the cases in the active areas sublayer are possible, 
with restrictions corresponding to the coherence rules.
\item The pattern intersects an edge between gray faces. 
This case is similar to the previous one, since the coherence rules impose 
that the symbols have to match on the sides of the edges.
\end{itemize}

\subsection{Completion of the linear counter and machines computations}

In this section, we tell 
how to complete the pattern over this supertile 
in the linear counter and machines layers. 
Since there are rules 
connecting the symbols of the two layers on the edges, 
we only have to tell 
how to complete separately these two layers 
when knowing only one part (near a corner, a edge, 
or inside) of the face corresponding to else the linear counter 
or the machine. 

\paragraph{The linear counter:} 

\begin{enumerate}
\item when knowing a part of face $1,2,3,4$ or $5$ 
which does not intersect the bottom line of face $1$, top line of faces $4,5$ and 
right column for faces $2$ (meaning 
places where there is incrementation of the counter), the only difficulty for completing 
comes from the freezing signal, which is \begin{tikzpicture}[scale=0.3]
\fill[Salmon] (0,0) rectangle (1,1);
\draw (0,0) rectangle (1,1);
\end{tikzpicture} only when all the letters are equal to $c_{\texttt{max}}$.
Thus when the supertile is colored with this signal, there is nothing 
to do but to complete the face adding only letters equal to $c_{\texttt{max}}$.
When the freezing signal is \begin{tikzpicture}[scale=0.3]
\draw (0,0) rectangle (1,1);
\end{tikzpicture}, then we have to add letters that are different from 
$c_{\texttt{max}}$.
\item when knowing some part of the incrementation positions, the completion depends 
on if we known the right 
or left part of the line 
(looking in the direction of face $1$): 
\begin{itemize}
\item when knowing the right part, the freezing signal on the top face and bottom face 
are determined. We complete the left part of the bottom line of face $1$ 
according to the freezing signal on the right part. If at some point 
the freezing signal becomes  \begin{tikzpicture}[scale=0.3]
\draw (0,0) rectangle (1,1);
\end{tikzpicture} after being  \begin{tikzpicture}[scale=0.3]
\fill[Salmon] (0,0) rectangle (1,1);
\draw (0,0) rectangle (1,1);
\end{tikzpicture} (from left to right), this means that we can complete 
the letters on the right equal to $c_{\texttt{max}}$. When the freezing signal is all blank, 
we can complete by adding at least some letter different from $c_{\texttt{max}}$, 
and when it is all \begin{tikzpicture}[scale=0.3]
\fill[Salmon] (0,0) rectangle (1,1);
\draw (0,0) rectangle (1,1);
\end{tikzpicture}, we can complete by adding only letters equal to $c_{\texttt{max}}$.
\end{itemize}
\end{enumerate} 

\paragraph{The machines.} When completing the machine face, 
there are two types of difficulties. 
The first 
one is managing the various signals : machine signals, 
first error, empty tape, empty side 
signals, and error signals. 
The second one is 
managing the space-time diagram 
of the machine. When 
the machine face is all known, there 
is no completion to make. Hence 
we describe the completion 
according to the parts of the 
face that are known (meaning that 
they appear in the initial pattern).
Since there is strictly less difficulty 
to complete knowing only a part 
inside the face than on the border 
(since the difficulties come from 
completing the space-time diagram 
of the machine, in a similar way than 
for the border), we describe 
the completion only when a part 
of the border is known.

\begin{itemize}
\item When knowing the top right corner of the machine face: 

\begin{figure}[ht]
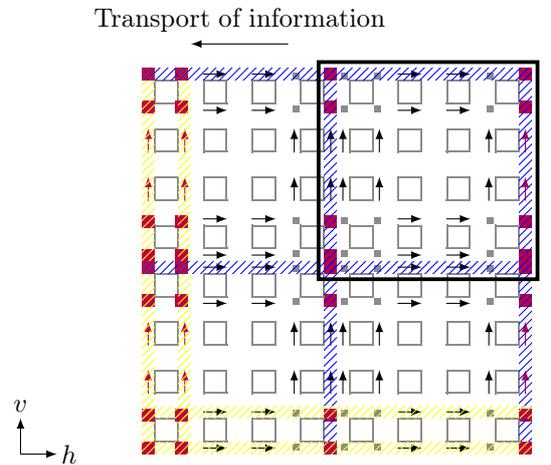

\[
\]
\caption{\label{fig.completing.off} Illustration 
of the completion 
of the $\texttt{on}/\texttt{off}$ 
signals and the space-time diagram 
of the machine. The known part is 
surrounded by a black square.}
\end{figure}

\begin{enumerate}
\item if the signal detects the first machine in error state from left to right 
(becomes %
), then 
we already know the propagation 
direction of 
the error signal. Then we complete 
the first machine signal and the error 
signal according to what is known.
For completing the 
space-time diagram of the machine, 
the difficulty comes from 
the fact that this is possible that when 
completing the trajectory of two machine 
heads according to the local rules, 
they have to collide reversely in time. 
This is not possible in our model.
This is where we use the 
$\texttt{on}/\texttt{off}$ signals.
We complete first the non already determined 
signals using only the symbol $\texttt{off}$, 
as illustrated on 
Figure~\ref{fig.completing.off}.

Then the space-time diagram is completed 
by only transporting the information.

In the end, we completed with any symbols 
in $\mathcal{Q}$ or $\mathcal{A} \times 
\mathcal{Q}$ where the symbols are not 
determined. We can do this since 
they do not interact with the known 
part of the space-time diagram. 
Then we complete empty tape and empty side 
signals according to the determined 
symbols.

\item if the signal is all 
, then one can 
complete the face without 
wondering about the error signal.
\item if the signal is all %
, we complete in the same way as for the first point, and 
the $\texttt{off}$ signals contribute 
to the first machine signal being 
all %
. If 
there is no error signal on the right side, 
one can complete so that 
the first machine in error state has above the arrow $\leftarrow$ indicating 
that the error signal has to propagate to the left. If there is an error signal 
on the right, then we set this arrow as $\rightarrow$. This is illustrated on 
Figure~\ref{fig.completion.signal.direction}.

\begin{figure}[ht]
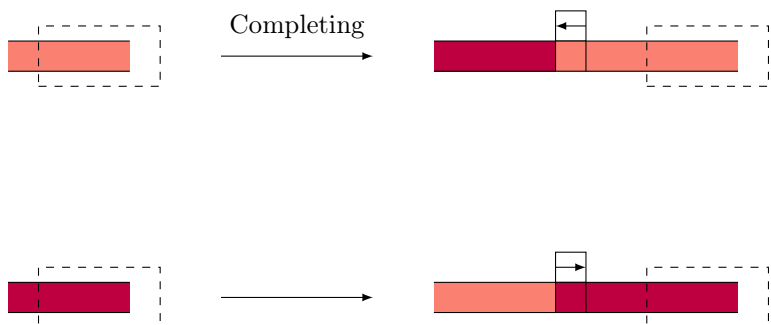

\[%
\]
\caption{\label{fig.completion.signal.direction}
Illustration of the completion 
of the arrows according to the error 
signal in the known part of the area, 
designated by a dashed rectangle.}
\end{figure}

\end{enumerate}
\item when knowing the top left corner, the difficulty comes from the direction of 
propagation of the error signal. This is ruled in a similar way as point 3 
of the last case. 
\item when knowing the bottom right corner or bottom left corner: the completion 
is similar as in the last points, except 
that we have to manage the empty 
tape and sides signals. The 
difficult point comes from the signal, 
whose propagation direction 
is towards the known corner. 
If this signal detects an error before entering 
in the known area, we complete 
so that the added symbols in 
$\{\texttt{on},\texttt{off}\}$
are all $\texttt{off}$: this induces the error. 
When this signal enters without detecting an error, 
we complete all the symbols so that 
they do not introduce any error.
A particular difficulty comes from the case when 
the bottom right corner is known. Indeed, 
when the signal enters without having detecting 
an error, this means we have to complete the 
pattern so that a machine head in initial state 
is initialized in the leftmost position 
of the bottom row. Since the pattern 
is completed in the structure layer into 
a great enough cell, this head can not enter 
in the known pattern.
\item when the pattern crosses only a edge or 
the center of the face, 
the completion is similar (but easier since 
these parts have less information, 
then we need to add less to the pattern).
\end{itemize}

What is left to describe for 
the completion into a cubic supertile 
is the completion of the hierarchical counter layer.

\subsubsection{\label{hierarchy.completion} Completion of the hierarchical counter layer}

We can complete the three copies of the hierarchy bits and random bits
on the completed supertile in a coherent way with the hierarchy bits and random bits 
on the pattern $P$. The colors in this layer are determined by a triple of 
colors in $\left\{\begin{tikzpicture}[scale=0.3]
\fill[gray!90] (0,0) rectangle (1,1);
\draw (0,0) rectangle (1,1);
\end{tikzpicture},\begin{tikzpicture}[scale=0.3]
\fill[purple] (0,0) rectangle (1,1);
\draw (0,0) rectangle (1,1);
\end{tikzpicture}\right\}$, corresponding to petal having maximal order in 
the three supertiles defining the cubic supertile. We have to prove how 
to view a three-dimensional cell having order $n+1$ with these 
colors in a greater three-dimensional cell, since as a consequence we can 
see the supertile around this order $n+1$ cell in the orientation 
corresponding to orientation marked on the supertile.
It is sufficient to view this order $n+1$ three-dimensional cell as a part of an 
order $2^k p -1$ cell with $k$ great enough that is inside a greater cell 
and nearest to one of the corner (the colors of its borders are
). There are three cases: 

\begin{enumerate}
\item When the order $n+1$ cell has its three colors equal to %
, then it can be found near the extremal corners of the 
$2^k p-1$ one (thanks to the transformation rules of the hierarchy bits layer).
\item When there is only one color equal to %
, then consider the set of cells near a edge, 
having its direction orthogonal to the directions of the  
face of the three-dimensional cell having border colored %
. This corresponding two-dimensional cell 
is colored %
 for all these three-dimensional cells. We pick one 
of the cells in this set such that the other faces have border colored 
with %
. This 
is possible because the other two two-dimensional cells 
have the same color for the cells in this set (because the $2^kp-1$ order cell 
has its borders all colored with %
\]
\end{figure}

\item When there are two colors equal to  %
, then we pick some order order $n+1$ two-dimensional cell 
on on the diagonal of the face of the $2^k p-1$ order three-dimensional cell 
parallel to the face of the $n+1$ order cell which
border is colored  %
, such that the line and columns of order $n+1$ two-dimensional cells 
on the face contains two-dimensional cells colored with  %
 (for instance the middle lines/columns of the face).
Hence we can pick a cell with colors ( %
) in the column of order $n+1$ three-dimensional cell under 
the picked two-dimensional cell. 
\end{enumerate}

\subsection{Computation of the entropy dimension}

Let us give upper and lower bounds on the limsup 
and liminf of 
\[\frac{\log_2 (\log_2 (N_n(X_z)))}{\log_2 (n)}.\] 

Let $P$ some $n$-block in the language of $X_z$. 

\subsubsection{Upper bound}

One can complete $P$ into an
an order $2pq_n+2$ supertile, where 
\[q_n = \lceil \log_2 (n) /2p \rceil +1.\]

The number of patterns over 
some $k$ order supertile
is less than $K$ times the number of random bits layout over a $k$ supertile 
times the number of possible values of the linear counter inside a supertile, where $K>0$ 
is a constant (indeed, all the other layers are determined by a symbol in a finite set).

As a consequence of 
Lemma~\ref{lem.count.random.2}, and Lemma~\ref{lem.number.active}, 
the number of patterns over a $2pq_n+2$-order supertile is smaller than 
\[K.{|\mathcal{A}_c|}^{2^{p+1}.(p+1)^{q_n}} . 2^{\alpha(p).{q_n}^{\lambda(p)}.
4^{c^0_{q_n}}.16^{pc^1_{q_n}}}.\]
This implies that 
\[\log_2(N_n (X_z)) \le \log_2 (K) + 2.2^p(p+1)^{q_n} \log_2 (|\mathcal{A}_c|) 
+ \alpha(p).{q_n}^{\lambda(p)}.
4^{c^0_{q_n}}.16^{pc^1_{q_n}},\]
and then 
\[ \begin{array}{ccl} \frac{\log_2 \circ \log_2(N_n (X_z))}{\log_2(n)} 
& \le & \frac{\log_2 (\alpha(p).{q_n}^{\lambda(p)})}{\log_2(n)}
+ 2 \frac{c^0_{q_n}}{\log_2(n)} + 4 \frac{pc^1_{q_n}}{\log_2(n)} \\
& & + \frac{1}{\log_2(n)} \log_2 \left(1+ \frac{\log_2 (K) + 2.2^p(p+1)^{q_n} 
\log_2 (|\mathcal{A}_c|)}{2^{pc^1_{q_n}}} \right)\end{array}\]
The first term of this sum 
tends towards zero, since 
\[\frac{\log_2 (\alpha(p).{q_n}^{\lambda(p)})}{\log_2(n)} = O\left( 
\frac{\log_ 2 \circ \log_2 (n)}{\log_2 (n)} \right)\]
The fact that the third term tends towards zero comes from the choice of $p$.
Indeed, we have 
\[\log_2 (K) + 2.2^p(p+1)^{q_n} 
\log_2 (|\mathcal{A}_c|) = O(2^{\log_2(n) \log_2 (p+1) /2p}) 
= O(n^{m/2(2^m-1)}).\] 
Moreover, by definition 
\[\frac{c^1_{q_n}}{q_n} \rightarrow z/2\]
and as a consequence 
\[2^{pc^1_{q_n}} \ge n^{z/2}\]
Since $m/p<z/2$, this means the third term tends towards zero.
Thus we have 
\[\limsup_n \frac{\log_2 \circ \log_2(N_n (X_z))}{\log_2(n)} \le 
\limsup_n 2 \frac{c^0_{q_n}}{\log_2(n)} + 4 \frac{pc^1_{q_n}}{\log_2(n)}
\le 2 \frac{(1-z/2)}{2p} + 4 p\frac{z/2}{2p} 
= 1/p + z(1-1/2p)\]

\subsubsection{Lower bound}

For the lower bound, we do the reverse inclusion: 
if $m$ is chosen great, any $n$-block 
$P$ in the language of the subshift $X_z$ contains some 
order $2pq'_n+2$ supertile, 
\[q'_n = \lfloor \log_2 (n) /2p \rfloor - 2,\]
since these supertiles are repeated with period $2^{2pq'_n+4} \le \frac{n}{2^{4p-4}}$, 
and have size $2^{2pq'_n+3} \le \frac{n}{2^{4p-3}}$.

Then the number of $n$-blocks in the language of $X_z$ is greater than 
\[2^{
4^{c^0_{q'_n}}.16^{pc^1_{q'_n}}}.\]
Similar computation as done for the upper bound 
leads to 
\[\liminf_n \frac{\log_2 \circ \log_2(N_n (X_z))}{\log_2(n)} \ge 1/p + z(1-1/2p),\]
which means $X_z$ has an entropy dimension and this dimension is equal to 
\[D_h (X_z) = 1/p + z(1-1/2p).\]

\subsection{Minimality}

In this section we prove that the subshift $X_z$ is minimal. 
See~\ref{fig.minimality.proof} for a schema of the proof.
This proof relies on the following lemma: 

\begin{lemma}[Globach's theorem] \label{lm.fermat.numbers}
The numbers $F_n = 2^{2^n}+1$, $n \ge 0$ are coprime. 
\end{lemma}

\begin{proof}
Let $m > n \ge 0$. Then 
\[F_m = 2^{2^m}+1 = (2^{2^n} + 1 - 1)^{2^{m-n}} +1 = \sum_{k=0}^{2^{m-n}} \binom{2^{m-n}}{k} 
(-1)^{-k} F_n^k + 1 \]
\[= F_n \sum_{k=1}^{2^{m-n}} \binom{2^{m-n}}{k} (-1)^{k} F_n^{k-1} + 2 \]
This means that a common divisor of $F_n$ and $F_m$ divide $2$, but $2$ does not divide $F_n$, 
so $F_n$ and $F_m$ are coprime. 
\end{proof}

Consider some block $P$ in the language of 
$X$, and complete it into a pattern $P'$
over an order $2^k p$ three-dimensional cell. Pick some configuration $x \in X$, 
and consider the restriction of $x$ on the hierarchy bits 
layer (over the copy of the Robinson subshift parallel to $\vec{e}^2$ and $\vec{e}^3$).
On can find some $\vec{u}$ such that the projection of the three dimensional cell 
over the hierarchy bits layer appear in position $\vec{u}$ (using 
the same arguments as in Section~\ref{hierarchy.completion}).

For all $i$, a three-dimensional cell 
appears on position  $(2i4^{2^k p})\vec{e}^1 + \vec{u}$.
Comparing the cell in position  $(v+2i4^{2^k p})\vec{e}^1 + \vec{u}$
and  $(v+2(i+1)4^{2^k p})\vec{e}^1 + \vec{u}$, the second one 
has the same linear counters as the first one, and the $j$th hierarchical counters 
values are incremented $4^{(2^k-2^j)p}$ times for all $j$ in the the second one, 
with respect to the first.

Let $t$ be the following application: 
\[\begin{array}{ccccc} t & : & \Z/p_1 \Z \times ... \times \Z/p_k \Z & \rightarrow & \Z/p_1 \Z \times ... \times \Z/p_k \Z\\
& & (i_1 , ... , i_k) & \mapsto & (i_1 + 4^{(2^k-1)p} , ... , i_k +1) 
\end{array},\] where $p_1$, ... , $p_k$ are 
the periods of $k$ first hierarchical counters. This is a minimal application, 
meaning that for all $\vec{i},\vec{j} \in \Z/p_1 \Z \times ... \times \Z/p_k \Z$, 
there exists some $n$ such that $t^n (\vec{i}) = \vec{j}$.
Indeed, considering some $\vec{i}$, denote $n_0$ the smallest 
positive integer such that $t^{n_0} (\vec{i}) = \vec{i}$. 
This means that $p_j$ divides $n_0 4^{(2^k-2^j)p}$ for all $j$, 
and because $p_j$ is a Fermat number, it is odd, and 
this implies that $p_j$ divides $n_0$. 
For the numbers $p_j$ are coprime (Lemma~\ref{lm.fermat.numbers}), 
this implies that $p_1 \times ... \times p_j$  
divides $n_0$. Because this number is a period of the 
application $t$, this means that $p_1 \times ... \times p_j$ is the smallest 
period of every element of $\Z/p_1 \Z \times ... \times \Z/p_k \Z$ under the application $t$.

As a consequence, for all $\vec{i}$, the finite sequence $(t^n (\vec{i}))$, for $n$ going 
from $0$ to $p_1 \times ... \times p_j-1$, takes all the possible values 
in $\Z/p_1 \Z \times ... \times \Z/p_k \Z$. \bigskip

\begin{figure}[ht]
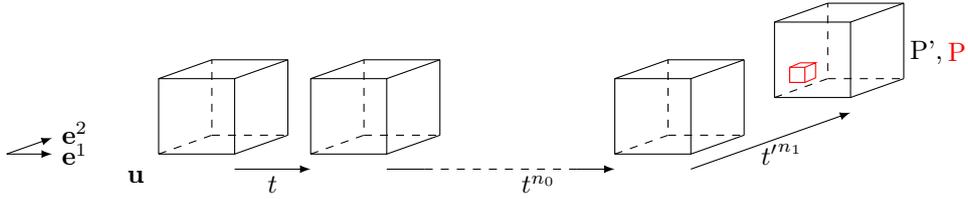

\[
\]
\caption{\label{fig.minimality.proof} Schema of the 
proof for the minimality property of $X_z$.}
\end{figure}

Hence one can find in $x$ an order $2^k p$ three-dimensional cell 
which supports values of the hierarchical counter equal to the ones that are in $P'$, 
iterating $n_0$ times this shift.

Now consider $\vec{u}'$ the position where this cell appears in $x$, 
and the cells appearing in the same configuration on positions 
$2i 4^{2^kp} \vec{e}^2 + \vec{u}'$. Using a similar argument as above, 
on the function \[\begin{array}{ccccc} t' & : & \Z/p'_1 \Z \times ... \times \Z/p'_k \Z & \rightarrow & \Z/p'_1 \Z \times ... \times \Z/p'_k \Z\\
& & (i_1 , ... , i_k) & \mapsto & (i_1 + 4^{(2^k-1)p} , ... , i_k +1) 
\end{array},\] where $p'_1$, ... , $p'_k$ are 
the periods of $k$ first linear counters, 
one can find some $i$ such that on position $2i 4^{2^kp} \vec{e}^2 + \vec{u}'$, 
the cell supports the same values of linear counters than $P'$, 
iterating $n_1$ times this shift.
Moreover, since the hierarchical counters are synchronized in the direction $\vec{e}^2$, 
this cells has also the same hierarchical counter values as $P'$.

For the pattern on the cell is determined by the values of the counters, this cell 
supports the pattern $P'$.

Hence $X_z$ is a minimal SFT.

\begin{remark}
It is trivial to find some minimal $\Z^3$-SFT having entropy dimension equal to zero. 
However, it is not to find one having entropy 
dimension equal to two, and we don't know a simpler 
proof than the construction presented in this text. 
\end{remark}

\begin{figure}[h!]
\begin{center}

\end{center}
\caption{\label{fig.circonvolutions} Illustration of the convolutions rules.}
\end{figure}

\begin{figure}[h]
\[%
\]
\caption{\label{fig.orientation.supertiles0} Possible orientations 
of four neighbor supertiles having the same order (1).}
\end{figure} 

\begin{figure}[h]
\[%
\]
\caption{\label{fig.orientation.supertiles1} Possible orientations 
of four neighbor supertiles having the same order (2).}
\end{figure}

\begin{figure}[h]
\[%
\]
\caption{\label{fig.orientation.supertiles2} 
Possible orientations 
of four neighbor supertiles having the same order (3).}
\end{figure} 

\begin{figure}[h]
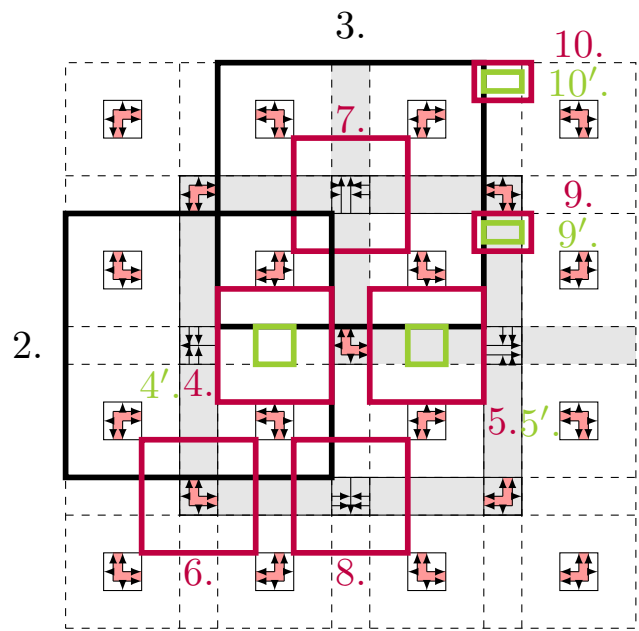

\[%
\]
\caption{\label{fig.completing.supertile} Illustration 
of the correspondance between patterns of 
Figure~\ref{fig.orientation.supertiles1}
and parts of a supertile.}
\end{figure}

\end{document}